\documentclass[10pt]{article}

\usepackage{etex}
\usepackage[utf8]{inputenc} 
\usepackage[american]{babel}
\usepackage[T1]{fontenc} 
\usepackage{latexsym}
\usepackage{empheq}
\usepackage{amsmath,amssymb,amsfonts,amsthm}
\usepackage{makecell}

\usepackage[text={6.25in,9in},centering]{geometry}
\usepackage[american]{babel}
\usepackage{mathrsfs}
\usepackage{graphicx}
\usepackage{tabularx}
\usepackage{caption}
\usepackage[table]{xcolor}
\usepackage{epsfig}
\usepackage{textcomp}
\usepackage{color,xcolor}
\usepackage{mathtools}
\usepackage{microtype}
\usepackage{csquotes} 
\usepackage{algpseudocode}
\usepackage{authblk}
\usepackage{tikz-cd}
\usepackage{comment}
\usepackage{empheq}
\usepackage{algorithm} 
\usepackage{algpseudocode}
\allowdisplaybreaks

\usepackage{tikz}
\usepackage{pgfplots}
\pgfplotsset{width=10cm,compat=1.9}
\usetikzlibrary{arrows,shapes,positioning,calc}
\pgfplotsset{every tick label/.append style={font=\tiny}}

\theoremstyle{plain}

\newtheorem{proposition}{Proposition}[section]

\newtheorem{remark}{Remark}[section]

\theoremstyle{remark}
\newtheorem{definition}{Definition}

\newcommand{\norm}[1]{\|{#1}\|}
\newcommand{\SPAN}[1]{\mbox{span}\,({#1})}
\newcommand{\as}{\ensuremath{\alpha^s}}
\newcommand{\us}{\ensuremath{u^s}}
\newcommand{\ex}{\ensuremath{\mathrm{ex}}}
\newcommand{\der}{\ensuremath{\mathrm{d}}}

\usepackage[
	pdftex,
	bookmarks,
	colorlinks=true, 
	pdftitle={physics-based adaptive SPS},
	pdfauthor={},
	pdfsubject={},
	pdfkeywords={},
	urlcolor= blue,
	linkcolor= blue,
	citecolor= blue,
]{hyperref}

\title{Physics-based adaptivity of a spectral method for the
    Vlasov--Poisson equations based on the asymmetrically-weighted
    Hermite expansion in velocity space}
\author[]{%
Cecilia Pagliantini\thanks{\textit{Corresponding author}. Centre for Analysis, Scientific computing and Applications,
Eindhoven University of Technology,
5600 MB Eindhoven, The Netherlands.
Email: \texttt{c.pagliantini@tue.nl}}\;,
Gian Luca Delzanno\thanks{T-5 ``Applied Mathematics and Plasma Physics Group'',
Los Alamos National Laboratory, Los Alamos, NM 87545, USA.
Email: \texttt{delzanno@lanl.gov}}\;, and
Stefano Markidis\thanks{KTH Royal Institute of Technology, Stockholm, Sweden.
Email: \texttt{stefano.markidis@kth.se}}
}
\date{}

\begin{document}
\maketitle

\begin{abstract}
  We propose a spectral method for the 1D-1V Vlasov--Poisson system
  where the discretization in velocity space is based on
  asymmetrically-weighted Hermite functions, dynamically adapted via a scaling
  $\alpha$ and shifting $u$ of the velocity variable.
  Specifically, at each time instant an adaptivity criterion selects
  new values of $\alpha$ and $u$ based on the numerical solution of
  the discrete Vlasov--Poisson system obtained at that time step.
  Once the new values of the Hermite parameters $\alpha$ and $u$ are
  fixed, the Hermite expansion is updated and the discrete system is further evolved for the next time step.
  The procedure is applied iteratively over the desired temporal
  interval.
  The key aspects of the adaptive algorithm are: the map between
  approximation spaces associated with different values of the Hermite
  parameters that preserves total mass, momentum and energy; and the
  adaptivity criterion to update $\alpha$ and $u$ based on physics
  considerations relating the Hermite parameters to the
  average velocity and temperature of each plasma species.
  For the discretization of the spatial coordinate, we rely on Fourier functions and use
  the implicit midpoint rule for time stepping.
  The resulting numerical method possesses intrinsically the property of
  fluid-kinetic coupling, where the low-order terms of the expansion are akin to the fluid moments of a macroscopic description of the plasma, while kinetic physics is retained by adding more spectral terms.
  Moreover, the scheme features conservation of total mass, momentum and
  energy associated in the discrete, for periodic boundary conditions.
  A set of numerical experiments confirms
  that the adaptive method outperforms the non-adaptive one in terms of accuracy and
  stability of the numerical solution.
\end{abstract}
  

\section{Introduction}
Many problems in plasma physics require the numerical solution of the
kinetic Vlasov--Maxwell equations which describe microscopic physics
via the seven-dimensional (three spatial and three velocity
coordinates plus time) phase space density of the plasma. 
Important examples where the kinetic physics is essential include
mechanisms of energy conversion in the Earth's magnetosphere
\cite{shields,reeves2013}, and in the solar corona \cite{Aschwanden2005}, in the
solar wind \cite{Bruno2005}.

The Vlasov--Maxwell equations are challenging to solve numerically because of their high dimensionality, nonlinearities and large spatial and
temporal scale separation. 
The latter already occurs at the microscopic level, due to the large
difference in mass between electrons and ions, and becomes enormous
when one compares microscopic scales with the characteristic scales of
systems of interest. 
While the development of numerical methods for the solutions of the
Vlasov--Maxwell equations is a very vibrant research area in its own
right, the development of methods that accurately couple microscopic
and macroscopic scales (i.e. the so-called fluid-kinetic coupling)
remains a major challenge for computational plasma physics.

In recent years there has been a renewed interest in the development of spectral methods for multi-scale plasma physics applications, owing in part to their interesting properties in terms of micro/macro coupling~\cite{CDLD06,VD15}. These methods, also called transform methods, expand the plasma
distribution function in suitable basis functions
\cite{EFMO63,Armstrong:1970,Gajewski-Zacharias:1977,K83,H96,SH98,
  NRxx,%
  CDMB16,
  GLD15,%
  VD15,%
  parker_dellar_2015,
  Manzini-Delzanno-Markidis-Vencels:2016,%
  Loureiro16,
  Manzini-Funaro-Delzanno:2017,%
  cai2018,%
  Fatone-Funaro-Manzini:2019,%
  Pezzi_2019,%
  Di2019, KY21}. With a suitable spectral basis (i.e. an asymmetrically-weighted Hermite expansion or a Legendre expansion), the low-order moments of the expansion reproduce the macroscopic/large-scale behavior of the plasma while the kinetic physics can be captured by adding more moments only where and when necessary. Hence, spectral methods possess intrinsic fluid-kinetic coupling. Some recent plasma physics applications of spectral methods that were too computationally expensive for traditional methods can be found in \cite{DR2019,RD2019}.

The majority of work on spectral methods for the Vlasov--Poisson problem to date has focused on spectral expansions based on
Hermite functions in velocity space.
Since the latter are naturally linked to Maxwellian
distribution functions, this approach is
well suited to capture the solution behavior
of problems where the plasma distribution function remains nearly
Maxwellian.
Conversely, problems where the plasma develops a
strong non-Maxwellian behavior might require an undesirably large
number of expansion terms for convergence to an accurate solution.
Grad \cite{Gradb,Grada} was the first to apply Hermite functions with a drift and scaling parameter. Schumer and Holloway \cite{SH98} then showed that a
suitable shifting and rescaling of the argument of the Hermite basis
can significantly improve the convergence of the expansion. 
This was exploited by Camporeale et al. \cite{CDLD06} to find eigenvalues of the linear stability problem of the Vlasov--Maxwell
equations (for the Harris-sheet equilibrium configuration). 
The shift of the argument of the Hermite functions can be interpreted physically as centering the distribution function with the local mean flow for each plasma species. The rescaling of the argument of the Hermite functions can instead be related physically to the temperature or thermal velocity of each species. Since plasmas are characterized by complex dynamics and can develop strong plasma flows and/or heating/cooling locally, it is therefore natural to expect that adapting the argument of the Hermite basis functions to the local plasma conditions could considerably improve the convergence of the spectral expansion. For instance, a drifting Maxwellian distribution function with drift velocity $u$ could be represented exactly with only one term in a Hermite expansion centered at $u$, while an expansion centered at $u_0\ne u$ would require many more terms (more if $|u-u_0|$ is larger) for a given accuracy.
However, the majority of existing time-dependent spectral approaches for plasma models that allow for shifting and rescaling coefficients of the Hermite expansion basis are limited to global constant coefficients, specified by the user at the beginning of the simulation, and are not adapted to the changing conditions in the simulations. In \cite{KY21}, the authors introduce a weighted Galerkin method with two separate scaling parameters for the Hermite polynomials that are fixed in time. In \cite{Baranger}, a scaling and shifting of the distribution function are used as indicator for an automatic construction of a locally refined velocity grid. In \cite{Sarna18}, the authors proposed a Hermite discretization of the Boltzmann equation with Hermite polynomials scaled by the initial temperature. More recently, in \cite{FFM22}, the choice of the Hermite parameters is done via machine learning techniques but only for simple test cases and not for plasma applications.

To the best of our knowledge, only few spectral methods for plasma physics allow for the adaptivity of the Hermite expansion. In \cite{FilbetBesse} a scaling parameter is introduced to ensure stability of the velocity discretization in a weighted $L^2$-norm.
The approach based on the regularized moment method called NR$xx$ \cite{NRxx}, together with its recent extension~\cite{NRxx2}, consists in rewriting the Vlasov equation in a co-moving frame given by the local macroscopic velocity and scaled by the square root of the local temperature, with both quantities given by self-consistent, time-dependent, nonlinear partial differential equations.
In \cite{FilbetRusso,FilbetRey} a rescaling of the distribution function in velocity is performed in order to treat strong non homogeneity in collisional kinetic models and the scaling parameters are updated based on the equation's moments.

In this paper, we consider a different approach which is much simpler from an algorithmic point of view and has different numerical properties. We expand the Vlasov equation with Hermite basis functions where the shifting $u$ and rescaling $\alpha$ coefficients are constant. However, since $u$ and $\alpha$ are related to the first and second moment of the distribution function, we monitor their time variation as we evolve the Vlasov equation numerically and when their change is bigger than a certain threshold, we perform a
transformation from the old basis to a new basis obtained with the updated $u$ and $\alpha$. This transformation is performed \textit{at a fixed instant of time}, which is crucial to avoid errors of order unity in the algorithm. This allows us to include the time adaptivity of the Hermite expansion in the algorithm without having to deal with a much more complex, nonlinear form of the Vlasov equation. Additionally, the properties of the numerical algorithm are also different from the NR$xx$ method and its extension.
We use a Fourier discretization in physical space and an implicit time stepping technique and demonstrate that the resulting numerical algorithm leads to the conservation of the discrete total mass, momentum, and energy (for periodic boundary conditions) of the system. Conversely, the NR$xx$ method for the Vlasov--Maxwell equations can only conserve the discrete total mass and momentum, while its recent extension can only conserve the discrete total mass and energy of the system. Finally, similar to the NR$xx$ method and its extension, we use a physics-based criteria to adapt the $u$ and $\alpha$ coefficients in time, exploiting the fact that they are related to the average velocity and temperature of each plasma species, i.e., the first and second moments of the plasma distribution function. However, while the physics-based choice is embedded in formulation of the NR$xx$ method,
in our case the prescription of how $u$ and $\alpha$ change is flexible and completely decoupled from the underlying model equations, and one could envision other possibilities such as adopting a mathematically-based criteria based on some form of error minimization.
Application of our new method to standard plasma physics test problems 
demonstrates the conservation properties of the algorithm and that the approach can indeed increase the accuracy of the
simulations relative to a non-adaptive approach.

The paper is organized as follows.
In~Section~\ref{sec:problem:formulation}, we introduce the model
problem, which is governed by the system of Vlasov--Poisson equations.
In Section~\ref{sec:Opt}, we present the adaptive framework and the
physically-based strategy that we use to adapt the Hermite parameters.
In Section~\ref{sec:conservation:properties}, we investigate the
conservation properties of the adaptive fully discrete scheme and
prove that it guarantees exact conservation of mass, momentum and
energy.
In Section~\ref{sec:implementation}, we provide implementation
details, and we discuss the computational cost associated with
updating the Hermite parameters.
In Section~\ref{sec:NumExp}, we explore the behavior of the adaptive
fully discrete scheme on a number of manufactured solution tests and
assess its performance on the two-stream instability.
In Section~\ref{sec:concluding:remarks}, we offer our final remarks
and conclusions.

\section{Problem formulation}
\label{sec:problem:formulation}
The focus of this work is on the electrostatic limit of the
Vlasov--Poisson equations with the idea that a similar
adaptive scheme can be designed \textit{mutatis mutandis} for the general case
of the Vlasov--Maxwell equations.
The Vlasov--Poisson system describes the dynamics of a collisionless
magnetized plasma under the action of the self-consistent electric
field.
We focus on a two species plasma consisting of electrons
(labeled by $``\,e\,"$) and singly-charged ions $(``\,i\,")$, whose
evolution at any time $t\in T\subset\mathbb{R}$ is described in terms
of the distribution function $f^s(t,x,v)$ for the plasma species
  $s\in\{e,i\}$.
  Here, $x\in\Omega_x$ is the spatial coordinate and $v\in\Omega_v$ is
  the velocity coordinate.
More in details, let the Cartesian phase space domain be
  denoted by $\Omega=\Omega_x\times\Omega_v\subset\mathbb{R}^2$
with $\Omega_x:=[0,L]$,
$L$ being the length of the spatial domain, and
$\Omega_v=\mathbb{R}$.

Let $L^2(\Omega)$ denote the standard Hilbert space of square integrable functions with inner product $(f,g)_{\Omega}:=\int_{\Omega}
fg \,\der x\der v$. We introduce the spaces of $L^2$ functions
which are periodic in the spatial coordinate and tend asymptotically to zero at infinity in velocity space, namely
\begin{equation*}
  \begin{aligned}
    & V:=\{f(t,\cdot,\cdot)\in L^2(\Omega):\;f>0,\; f(t,0,\cdot)=f(t,L,\cdot),\;
    f(t,\cdot,v)\rightarrow 0\,\mbox{ as }\,|v|\to\infty\},\\
    & W:=\{g(t,\cdot)\in L^2(\Omega_x):\; g(t,0)=g(t,L)\}.
  \end{aligned}
\end{equation*}
The one-dimensional Vlasov--Poisson problem in $\Omega\times T$ reads:
For $f^s_0\in V_{|_{t=0}}$, find $f^s(t,x,v)\in C^1(T;L^2(\Omega))\cap
C^0(T;V)$, and $E(t,x)\in C^0(T;W)$ such that
\begin{equation}\label{eq:VlasovPoisson}
\left\{  \begin{aligned}
    &\partial_tf^s+v\,\partial_x f^s + \dfrac{q^s}{m^s} E\,\partial_v f^s=0, &\qquad\mbox{in}\;T\times\Omega,\;\forall s,\\
    &\partial_x E=\sum_{s}q^s \int_{\Omega_v}f^s\, \der v,                        &\qquad\mbox{in}\;T\times\Omega_x,\\
    &f^s(0,x,v)=f^s_0,														&\qquad\mbox{in}\;\Omega,
  \end{aligned}\right.
\end{equation}
where $q^s$ and $m^s$ denote the charge and mass of the species $s$,
respectively.

We normalize the model equations as follows. Time $t$ is normalized to the electron plasma frequency $\omega_{pe}=\sqrt{e^2 n^{e}_0\slash{}(\varepsilon_0m^e)}$, where $e$ is the elementary charge, $m^e$ is the electron mass, $\varepsilon_0$ is the permittivity of vacuum, and $n^e_0$ is a reference electron density. The velocity coordinate $v$ is
normalized to the electron thermal velocity $v_{t}=\sqrt{k T_e/m_e}$, where $k$ is the Boltzmann constant and $T_e$ is a reference electron temperature; the spatial coordinate $x$ is normalized to the electron Debye length $\lambda_D=\sqrt{\varepsilon_0 k T_e \slash{}(e^2 n_0^e)}$.  We normalize the charge $q_s$ and the mass $m_s$ to the elementary charge $e$ and the mass $m^e$, respectively. For simplicity, we keep the same symbols to denote dimensional and normalized quantities but in what follows we will focus on normalized quantities only.

We refer to the classic book of Glassey~\cite{Glassey96} for wellposedness results and the
analysis of the Cauchy problem for the Vlasov--Poisson system.

\subsection{Hermite spectral discretization in velocity}
For the numerical approximation of the Vlasov--Poisson problem
\eqref{eq:VlasovPoisson}, we pursue an Eulerian method of lines
approach.
We rely on the numerical scheme proposed in \cite{GLD15} by using
spectral methods in phase space and the implicit midpoint rule as time-stepping.
Observe that the adaptive method proposed in this work can accommodate
other spatial and temporal discretizations.
We summarize the main ideas of the numerical approximation in velocity in this
subsection. Details on the spatial discretization used for the numerical tests of Section~\ref{sec:NumExp} are given in Section~\ref{sec:appA}.

Let $\{\psi^{\alpha^s,u^s}_n\}_{n\in\mathbb{N}}$ be the family
of generalized \emph{asymmetrically weighted} (AW) Hermite functions,
\begin{equation}\label{eq:psi}
  \psi^{\as,\us}_n(v):=(\pi 2^n n!)^{-1/2} H_n(\xi) e^{-\xi^2},
  \quad
  \xi(v):=\dfrac{v-\us}{\as},\qquad\as,\,\us\in \mathbb{R},\,\as>0,
\end{equation}
where $H_n$ is the $n$-th Hermite polynomial, and
$\alpha^s$, $u^s$ correspond to a scaling and a shift of the Hermite function, respectively.
The Hermite functions satisfy the orthogonality relations
\begin{equation}\label{eq:HForth}
\int_{\mathbb{R}}\psi_n^{\alpha^s,u^s}(v)\,\psi_m^{\alpha^s,u^s}(v)\,\sqrt{\pi}(\alpha^s)^{-1}e^{\,\xi^2}\,\der v=\delta_{n,m},\qquad
 \forall\,n,m\in\mathbb{N}.
\end{equation}

Let 
$\xi_j=\xi_j(v):=(v-\us_j)/\as_j$ with $\as_j, \us_j\in\mathbb{R}$. 
Based on the Hermite functions, we define the finite
dimensional approximation spaces
\begin{equation}\label{eq:VN}
  \begin{aligned}
    & V^N_{j}:=\SPAN{\{\psi^{\as_j,\us_j}_n(v)\}_{n\in\Lambda_{N_v}}},
    \qquad
    \widehat{V}^N_{j}:=\SPAN{\{(\pi 2^n n!)^{-1/2} H_n(\xi_j)\}_{n\in\Lambda_{N_v}}},
  \end{aligned}
\end{equation}
where $\Lambda_{N_v}:=\{n\in\mathbb{N}:\,0\leq n\leq N_v\}$.
Here $\widehat{V}^N_{j}$ corresponds to the dual space of $V^N_{j}$.

Let $T_h$ be a non-uniform partition of the temporal interval
$T\subset\mathbb{R}$.
Let us define the set of temporal indices as
$\Lambda_{N_t}:=\{j\in\mathbb{N}:\, 0\leq j\leq N_t-1,\,
N_t\in\mathbb{N}\}$
so that $T_h=\bigcup_{j\in\Lambda_{N_t}} T_j$
where $T_{j}:=(t^{j},t^{j+1}]$.
Let us denote with $\Delta t_j:=|T_j|$ the local time step.

We pursue a Galerkin spectral discretization of
\eqref{eq:VlasovPoisson} in the velocity variable, where we omit the superscript $s$ for the sake of readability.
The semi-discretized problem in each time interval $T_j$, $j\in\Lambda_{N_t}$, reads:
given the initial condition $(f_{j}^N,E^N_{j})\in (V^N_{j}\times W)\times W$,
find $(f_{j+1}^N,E^N_{j+1})\in (V^N_{j}\times W)\times W$
  such that
\begin{equation}\label{eq:VFhL2w}
\left\{
\begin{aligned}
    & \big(f_{j+1}^N-f_{j}^N,h^N\big)_{\Omega_v} +
    \Delta t_j \big(v\partial_x f_{j+1/2}^N + \dfrac{q}{m} E^N_{j+1/2}\,\partial_v f_{j+1/2}^N,h^N\big)_{\Omega_v} = 0,\\
    & \partial_x E_{j+1}^N = \sum_s q^s \int_{\Omega_v} f_{j+1}^{s,N}\,\der v,
  \end{aligned}\right.
\end{equation}
for all $h^N\in \widehat{V}^N_{j}$, with
$f_{j+1/2}^N:=(f_{j+1}^N+f_{j}^N)/2$.  
In the interval $T_j$, the approximated function $f_j^N$ 
can be represented in its spectral expansion in velocity space
$V^N_{j}$ as
\begin{equation}\label{eq:fNexp}
  f_j^N(x,v) = \sum_{n\in\Lambda_{N_v}} \widehat{C}_{n}^j(x)\,\psi^{\alpha_j,u_j}_n(v),
\end{equation}
where the linear functionals $\{\widehat{C}^j_{n}(x)[\cdot]: V^N_{j}\rightarrow\mathbb{C}\}_{n}$ 
are the expansion coefficients defined as
\begin{equation*}
  (f_j^N(x,\cdot)\in V^N_j) \longmapsto \left\{\widehat{C}_{n}^j(x)[f_j^N]:=
      \int_{\Omega_v}f_j^N(x,v)\,\psi^{\alpha_j,u_j}_n(v)\,\sqrt{\pi}\alpha^{-1}_je^{\xi_j^2}\,\der v,
    \right\}_{n\in\Lambda_{N_v}},
\end{equation*}
and, by construction, they form a basis for the dual spaces of $V^N_{j}$.

By virtue of the orthogonality relations \eqref{eq:HForth}, the
discrete problem \eqref{eq:VFhL2w} can be recast as a set of
$N_t(N_v+1)$ coupled equations in the unknown spectral coefficients. The resulting system reads:
  For each $j\in\Lambda_{N_t}$, given the initial
    spectral coefficients
    $\{\widehat{C}_{n}^{j}(x)\}_{n\in\Lambda_{N_v}}$, find
    $\{\widehat{C}_{n}^{j+1}(x)\}_{n\in\Lambda_{N_v}}$
    such that,
\begin{equation}\label{eq:ode}
\left\{ \begin{aligned}
    &\dfrac{\widehat{C}_{n}^{j+1}(x)-\widehat{C}_{n}^j(x)}{\Delta t_j}
    + \sqrt{\dfrac{n}{2}}\, \alpha_j\, \partial_x\widehat{C}_{n-1}^{j+1/2}(x)
    + \, u_j\, \partial_x\widehat{C}_{n}^{j+1/2}(x)
    + \sqrt{\dfrac{n+1}{2}}\,\alpha_j\,\partial_x\widehat{C}_{n+1}^{j+1/2}(x)\\[.2em]
    &\qquad\qquad\qquad\qquad -\sqrt{2n}\,\dfrac{q}{m}\,\dfrac{1}{\alpha_j}\widehat{C}_{n-1}^{j+1/2}(x)\,E^N_{j+1/2}(x)=0, 
    \qquad \forall\,n\in\Lambda_{N_v},\\[.5em]
    & \partial_x E_{j+1}^N(x)-\sum_s q^s\alpha_j^s\,\widehat{C}_{0}^{s,j+1}(x)=0.
  \end{aligned}\right.
\end{equation}
For the closure of the system of equations \eqref{eq:ode} 
we set $\widehat{C}^j_{n}(x)=0$, $j\in\Lambda_{N_t}$, $x\in\Omega_x$,
whenever $n\notin\Lambda_{N_v}$. 

\subsection{Artificial collisional operator}\label{sec:Coll}
The collisionless plasmas of interest here can develop smaller and
smaller scales in velocity space as time evolves, a phenomenon called
\emph{filamentation}.
The discretization of the velocity space is associated with a minimum
wavelength that can be resolved to the extent that any simulation will
inevitably run out of resolution.
Closing system \eqref{eq:ode} with $\widehat{C}^j_{n}=0$ for
$n\notin\Lambda_{N_v}$ determines the finest resolution in velocity
space.
The filamentation is also associated with the well-known recurrence
problem which is typical of higher-order Eulerian--Vlasov and spectral
methods \cite{CN76,JKM71,canosa74}.  
In order to mitigate the filamentation, it is customary to add a physical or
artificial collisional term to the right hand side of the Vlasov
equation.
Specifically, the Vlasov equation with artificial collisions reads:
Find $f^s\in C^1(T;L^2(\Omega))\cap C^0(T;V)$, such that
\begin{equation*}
  \begin{aligned}
    \partial_t f^s + v\partial_x f^s + \dfrac{q}{m} E^N\,\partial_v f^s =
    \nu\,\mathcal{C}(f^s),
    &\qquad\mbox{in}\;\Omega\times T,\;\forall s,
  \end{aligned}
\end{equation*}
where $\nu$ is a positive bounded constant, and $\mathcal{C}$ is the
operator modeling the artificial collisionality.
Filtering techniques \cite{parker_dellar_2015,Di2019,FilbetXiong} have also been used in combination with spectral approximation to suppress the numerical recurrence effect.

We follow \cite[Section 4]{CDBM16} and consider the collisional
operator defined, for any function $f^N_j\in V^N_{j}\times W$ and
for all $n\in\Lambda_{N_v}$, according to the
discrete formulation \eqref{eq:VFhL2w}, namely
\begin{equation}\label{eq:CollDC}
  \begin{aligned}
    \sqrt{\pi}(\alpha_j)^{-1}\big(\mathcal{C}(f^N_j),\psi^{\alpha_j,u_j}_ne^{\xi_j^2}\big)_{\Omega_v} =-
    \dfrac{n(n-1)(n-2)}{(N_v-1)(N_v-2)(N_v-3)}\,\widehat{C}_{n}^j(x).
  \end{aligned}
\end{equation}
This operator does not act directly on the first three modes of the
Hermite expansion and, therefore, maintains the conservation laws of
total mass, momentum and energy \cite{CDBM16}.
It corresponds to adding a reaction-advection-diffusion global
operator in velocity space:
indeed the term \eqref{eq:CollDC} can be obtained as a nonlinear
combination of continuous terms involving the Lenard--Bernstein
operator \cite{LB58},
\begin{equation}\label{eq:LenBern}
  \mathcal{C}(f) = \partial_v\left(vf+\dfrac12\partial_v f\right).
\end{equation}
The function \eqref{eq:LenBern} is an instance of the Fokker--Planck
form of the Landau collision operator
$\mathcal{C}(f):=\partial_v(\mathbb{A}f+\partial_v(\mathbb{D}f))$,
with the choices $\mathbb{A}=v-u$, $u=0$ and $\mathbb{D}=T\mathbb{I}$,
$T=1/2$.
With such choice the Lenard--Bernstein operator vanishes at the
Maxwell--Boltzmann equilibrium $f=\pi^{-1/2} e^{v^2}$.

Additionally, it is known that the AW discretization of the Vlasov--Poisson system considered here can become numerically unstable (i.e. the $L_2$ norm of the distribution function is not bounded and can grow in time) \cite{SH98}. This typically signals a lack of resolution such that the number of Hermite modes used in the simulation does not provide an adequate representation of the distribution function dynamics. The presence of a collisional operator adds numerical stability to the discretized equations, although in our experience the numerical instability is typically fixed by increased resolution. We remark that the techniques presented in this paper improve the numerical stability of the resulting algorithm since the dynamical adaptivity of the Hermite functions (by properly adjusting their shift and scaling argument) is intended the capture more accurately the dynamics of the distribution function. Indeed, the numerical experiments presented in Sec.~\ref{sec:NumExp} confirm the better numerical stability of the adaptive algorithm relative to an algorithm that does not use adaptivity.

Overall, we emphasize that, since the collisional term is a nondegenerate high
order diffusion, it must always be used in a convergence sense (this is, in fact, the real meaning of the collisionless approximation~\cite{stix}). 
In other words, the collisional coefficient $\nu$ has to be chosen in
such a way that the problem remains consistent with the dynamics of
collisionless plasmas of interest and, at the same time, introduces enough
artificial viscosity to prevent the recurrence effect associated with
the filamentation.

\section{Dynamical adaptivity of the Hermite functions}
\label{sec:Opt}
The discretization in velocity space is a crucial aspect for the
numerical treatment of kinetic models.
Spectral approximations have the potential of yielding highly accurate
solutions even in the presence of highly nonlinear dynamics.
However, a poor choice of the basis functions might require spectral
expansions with a large number of modes to achieve even moderate
accuracy.

In order to illustrate this point, let us consider the case of a shifted Maxwellian distribution function
\begin{equation}
    f=\exp\left(-\left(\frac{v-u_0}{\alpha_0}\right)^2\right),
    \label{f0}
\end{equation}
with $u_0\ne0$ varied parametrically and $\alpha_0=1$.
Let us try to represent $f$ with an AW Hermite expansion centered at $u=0$ and $\alpha=1$, Eq. (\ref{eq:fNexp}). We keep the total number of Hermite modes equal to 30. We represent the distribution function on 2000 points equally spaced point in velocity space between -2 and 5 and compute the relative error as the $L_2$ norm between the analytic and the expansion centered at $u=0$. Figure \ref{fig:expansion} (left) shows the relative error as a function of $u_0$. One can see that the error is practically zero and flat for $u_0<1$ but it starts to rise dramatically for $u_0>1$. For $u_0=2.4$ the error is $\sim 2\%$, while for $u_0=3$ the expansion solution is diverging with relative error $\sim 10^4$. Figure \ref{fig:expansion} (right) shows the analytic and the reconstructed distribution functions for the case $u_0=2.4$. One can notice that quite high oscillations begin to appear for $v<2$, particularly where the analytic distribution function \eqref{f0} is asymptotically going to zero. On the other hand, if we had centered the expansion with $u=u_0$, we would have captured the analytic function exactly with only one Hermite mode. This simple exercise illustrates nicely (1) how a significant departure in the shift of the argument of the Hermite functions inevitably leads to a strong lack of accuracy and, hence, (2) that centering the Hermite expansion appropriately is critical. In Section~\ref{sec:NumExp}, we will show with numerical experiments this might lead to numerical instabilities when using an algorithm that does not adapt in time the shift and scaling of the Hermite functions.  

\begin{figure}[H]
  \centering
  \includegraphics[width=0.8\textwidth]{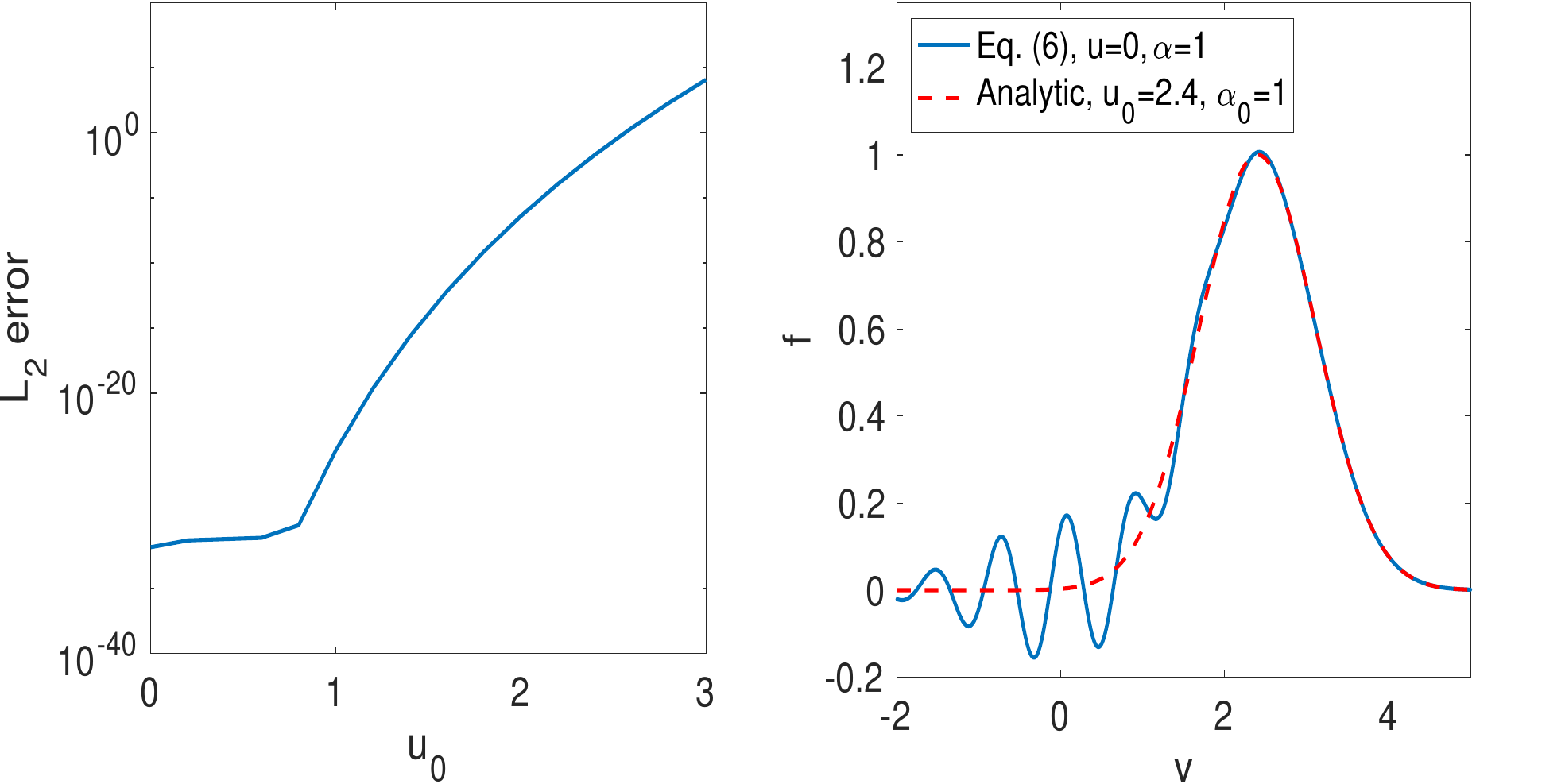}
  \caption{Relative error between the analytic distribution function given by Eq. \eqref{f0} ($u_0$ variable and $\alpha_0=1$) and the distribution function given by expansion \eqref{eq:fNexp} with $u=0$, $\alpha=1$ and $30$ Hermite modes as a function of $u_0$ (left); analytic and reconstructed distribution functions for $u_0=2.4$ (right).}
  \label{fig:expansion}
\end{figure}

We propose a numerical scheme where the discretization in velocity is
adapted over time via a dynamical selection of the parameters
$\alpha^s$ and $u^s$ entering the Hermite basis functions
\eqref{eq:psi}.
The gist of the proposed adaptive scheme is to update the Hermite
basis functions at the beginning of each temporal interval by finding a spectral
approximation able to ``follow'' the evolution of the numerical solution.
Specifically,
we allow the approximation
space spanned by the Hermite polynomials to change between temporal intervals, at fixed time.
If $V^N_{j-1}$ is the approximation space in the interval
$T_{j-1}=(t^{j-1},t^j]$, we can select, at time $t^{j}$, new Hermite
parameters so that in the following interval $T_{j}=(t^{j},t^{j+1}]$
the discrete Vlasov--Poisson system is solved in the updated space
$V^N_{j}$, which ideally has better approximation properties for a
fixed number of Hermite modes $N$. See Figure~\ref{fig:sketchadaptivity} for a sketch of the adaptive algorithm.

\begin{figure}[H]
\centering
\begin{tikzpicture}
\centering
\draw (3,-0.7) -- (7,-0.7);
\draw (1,-0.7) node {$\ldots$};
\draw (3,-0.9) -- (3,-0.5);
\draw (7,-0.9) -- (7,-0.5);
\draw (3,-0.5) node[above] {$t^{j-1}$};
\draw (3,-0.9) node[below] {\footnotesize $\alpha_{j-1}{=}1$};
\draw (3,-1.3) node[below] {\footnotesize $u_{j-1}{=}1$};
\draw (5,-0.9) node[below] {\footnotesize $V^N_{j-1}$};
\draw (7,-0.5) node[above] {$t^{j}$};
\draw (7.75,-0.9) node[below] {\footnotesize $\alpha_{j}{=}\sqrt{2}$};
\draw (7.75,-1.3) node[below] {\footnotesize $u_{j}{=}2$};

\draw[->] (7,-1) -- (7,-2.2);
\draw (7,-2.5) -- (11,-2.5);
\draw (7,-2.7) -- (7,-2.3);
\draw (11,-2.7) -- (11,-2.3);
\draw (9,-2.7) node[below] {\footnotesize $V^N_{j}$};
\draw (11,-2.3) node[above] {$t^{j+1}$};
\draw (13,-2.5) node {$\ldots$};

\draw[->] (4,1.5) -- (5.25,1.5);
\hspace{5em}
\begin{axis}[name=plot1,at={(10,70)}, anchor=west, height=3.5cm,width=4cm,axis lines = left, legend style={at={(0.7,1.38)},anchor=north, font=\footnotesize}]
   \addplot[domain=-2:6, samples=100, color=blue]{exp(-(x - 1)^2)};
   \addlegendentry{$\exp(-(v - 1)^2)$}
\end{axis}
\begin{axis}[name=plot2,at={($(plot1.east)+(2cm,0)$)},anchor=west,height=3.5cm,width=4cm,axis lines = left,  legend style={at={(0.8,1.38)},anchor=north, font=\footnotesize} ]
    \addplot[domain=-2:6, samples=100, color=blue]{exp(-(x - 2)^2/2)};
    \addlegendentry{$\exp(-\frac{(v - 2)^2}{2})$}
  \end{axis}
\end{tikzpicture}
\caption{Sketch of the adaptive algorithm. Assume that, at time $t^{j-1}$, the distribution function is given by $f(v)=\exp(-(v-1)^2)$. The optimal Hermite parameters are then given by $(\alpha_{j-1},u_{j-1})=(1,1)$.  At time $t^j$ new Hermite parameters $(\alpha_j,u_j)$ are computed to better approximate the current solution, here $f(v)=\exp(-(v-2)^2/2)$.}
\label{fig:sketchadaptivity}
\end{figure}
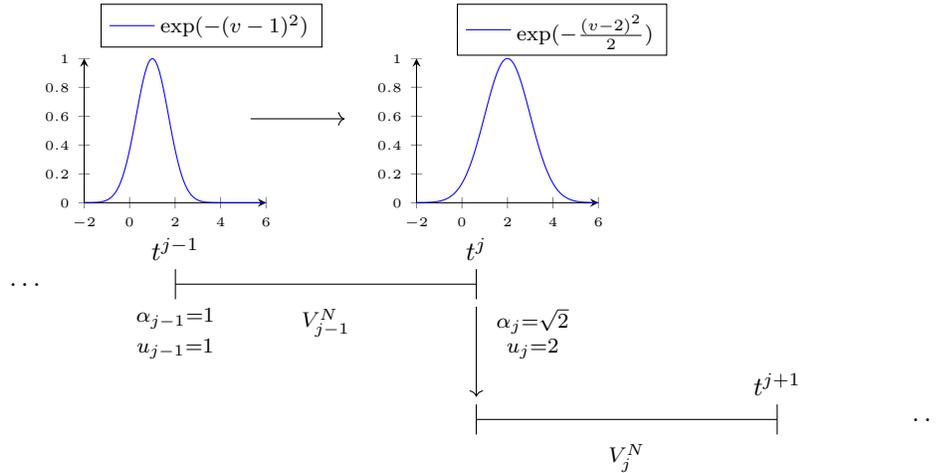

Moreover, in updating the approximation space, we need to map the numerical
solution at time $t^{j}$ into the updated space $V^N_{j}$ in such a
way that it can be a valid initial condition for the Vlasov--Poisson
discrete problem posed in the interval $T_{j}$.
This is a crucial aspect for the mathematical validity of the approximation.
We observe that in the NR$xx$ algorithm \cite{NRxx} this point is not discussed, while it is included in its recent extension~\cite{NRxx2}.
\begin{remark}
A critical point in this adaptive approach is that we adapt the approximation space \emph{at fixed
  time}.
In other words, the discrete system in each time interval is always evolved with fixed
parameters $u$ and $\alpha$, while these parameters can only change
between time steps, i.e. at the interface between temporal intervals.
Attempting to update the parameters during a time step might result in
order-unity errors associated with mixing the two approximation spaces
$V^N_{j-1}$ and $V^N_{j}$ at each fixed time $t^j$. Indeed, a spectral discretization relies on linear operations between the coefficients of the expansion of the distribution function with respect to the spectral basis, as shown in \eqref{eq:ode}. Those operations are valid only if the spectral coefficients refer to the same basis. The reason is simply that a function expanded in two different bases produces different expansion coefficients even if the bases span the same space. Analogously, when changing the basis, the expansion coefficients need to be suitably re-computed.
\end{remark}
Note that, in this paper, we consider only the time variation of
$u$ and $\alpha$ and leave to a future work the case where these
parameters change both in space and time.

Two main ingredients are needed in the adaptive approach highlighted
above: first, a criterion to update the Hermite parameters
$(\alpha,u)$, and second, a suitable definition of the operator
$\mathcal{P}^N_{j}$ mapping $V^N_{j-1}$ into $V^N_{j}$.
Let us start from the latter task.

\subsection{Operator from \texorpdfstring{$V^N_{j-1}$}{} into \texorpdfstring{$V^N_{j}$}{}}

Assume that at a given time $t^j$ we have determined the ``optimal''
Hermite parameters $(\alpha_j,u_j)$ such that the adapted
finite-dimensional velocity space is defined as
$V^N_{j}:=\SPAN{\{\psi^{\alpha_j,u_j}_n(v)\}_{n\in\Lambda_{N_v}}}$.
To evolve the discrete Vlasov--Poisson problem in the new space, we
need to map the numerical solution $f^N_j$ at time $t^j$ from $V_{j-1}^N$ to
$V^N_{j}$.
There are in principle many ways to define the operator mapping
$V^N_{j-1}$ into $V^N_{j}$.  We chose to map $f^N_j$ into its
orthogonal projection onto $V^N_{j}$ with respect to the $L^2$-inner
product weighted by $\omega_j$.
\begin{definition}\label{def:Pj}
  Let $V^N_{j-1}$ and $V^N_{j}$ be the finite-dimensional spaces
  introduced in \eqref{eq:VN}.
  We define $\mathcal{P}^N_{j} :V^N_{j-1}\subset L^2(\mathbb{R})
  \rightarrow V^N_j\subset L^2(\mathbb{R})$ as the operator that maps
  any $f^N\in V^N_{j-1}$ into the function $\mathcal{P}^N_{j}f^N$
  satisfying
  \begin{equation}\label{eq:L2omegaOrt}
    (\sqrt{\omega_j}(\mathcal{P}^N_{j}f^N-f^N),\sqrt{\omega_j}h^N)_{\Omega_v} = 0\,,\qquad \forall\, h^N\in V^N_{j},
  \end{equation}
  where $\omega_j=\sqrt{\pi}\alpha^{-1}_je^{\xi_j^2}$ and $\xi_j=(v-u_j)/\alpha_j$.
\end{definition}
The inner product in \eqref{eq:L2omegaOrt} is well-defined since
$\sqrt{\omega_j}\mathcal{P}^N_{j}f^N\in L^2(\mathbb{R})$, and
$\sqrt{\omega_j}h^N\in L^2(\mathbb{R})$.
Moreover, since $f^N\in L^2(\mathbb{R})$ and
$\{\psi_n^{\alpha_j,u_j}\}_{n\in\mathbb{N}}$ is a basis of
$L^2(\mathbb{R})$, $f^N$ admits a spectral expansion in
$\{\psi_n^{\alpha_j,u_j}\}_{n\in\mathbb{N}}$, from which we can
immediately see that also $\sqrt{\omega_j}f^N\in L^2(\mathbb{R})$.

In order to relate the spectral coefficients
$\{\widehat{C}_n(t^j,x)\}_{n\in\Lambda_{N_v}}$ of $f^N$ to the
spectral coefficients $\{\widehat{P}_n(t^j,x)\}_{n\in\Lambda_{N_v}}$
of $\mathcal{P}^N_{j}f^N$ at time $t^j$, we derive the matrix representation of
the linear mapping $\mathcal{P}^N_{j}$ as follows.
Owing to Definition~\ref{def:Pj},
it holds
\begin{equation}\label{eq:coeffPj}
  \widehat{P}_n(t^j,x)=\sum_{m\in\Lambda_{N_v}}\widehat{C}_m(t^j,x)(\sqrt{\omega_j}\psi_m^{\alpha_{j-1},u_{j-1}},\sqrt{\omega_j}\psi_n^{\alpha_{j},u_{j}})_{\Omega_v}
  \qquad\forall\,n\in\Lambda_{N_v},\,\forall\, x\in\Omega_x.
\end{equation}
Hence, the transformation matrix
$\mathbb{P}^{\,j}\in\mathbb{R}^{(N_v+1)\times (N_v+1)}$ is defined,
for all $n,m\in\Lambda_{N_v}$, as
\begin{equation}\label{eq:transfM}
  \mathbb{P}^{\,j}_{n,m}=(\psi_m^{\alpha_{j-1},u_{j-1}},\psi_n^{\alpha_{j},u_{j}})_{\Omega_v,\,\omega_j}
  := \int_{\Omega_v}\psi_m^{\alpha_{j-1},u_{j-1}}(v)\psi_n^{\alpha_j,u_j}(v)\,\omega_j\, \der v.
\end{equation}
In order to give an explicit expression of the entries of
$\mathbb{P}^{\,j}$ we need to distinguish three different cases,
associated with the change of only one parameter or both.
Let us introduce the coefficients
\begin{equation*}
  a:=\dfrac{\alpha_j}{\alpha_{j-1}}\in\mathbb{R}\setminus\{0\},\qquad b:=\dfrac{u_j-u_{j-1}}{\alpha_{j-1}}\in\mathbb{R},
\end{equation*}
and let $K_{n,m}:=2^{\frac{m-n}{2}}(m!)^{-\frac12}(n!)^{\frac12}\in\mathbb{R}$
for all $n,m\in\Lambda_{N_v}$.
Hence,
\begin{itemize}
\item[(i)] If $\alpha_j\neq\alpha_{j-1}$ and $u_j\neq u_{j-1}$\,:
  \begin{equation*}
    \mathbb{P}^{\,j}_{n,m} = K_{n,m}\dfrac{1}{a^{m+1}}
    \sum_{\substack{k=m \\ k-m\;even}}^n\dfrac{1}{(n-k)!\left(\frac{k-m}{2}\right)!}\left(-\dfrac{2b}{a}\right)^{n-k}\left(\dfrac{1}{a^2}-1\right)^{\frac{k-m}{2}}.
  \end{equation*}
\item[(ii)]
  If $\alpha_j=\alpha_{j-1}$ and $u_j\neq u_{j-1}$\,:
  \begin{equation*}
    \mathbb{P}^{\,j}_{n,m} = K_{n,m}\dfrac{1}{(n-m)!}(-2b)^{n-m}.
  \end{equation*}
\item[(iii)]
  If $\alpha_j\neq\alpha_{j-1}$ and $u_j= u_{j-1}$\,:
  \begin{equation*}
    \mathbb{P}^{\,j}_{n,m} = K_{n,m}\dfrac{1}{a^{m+1}}
    \dfrac{1}{\left(\frac{n-m}{2}\right)!}
    \left(\dfrac{1}{a^2}-1\right)^{\frac{n-m}{2}},\qquad\mbox{for $n-m$ even}.
  \end{equation*}
\end{itemize}
Note that, in all the above cases, if $n<m$ then
$\mathbb{P}^{\,j}_{n,m}=0$, namely the transformation matrix
$\mathbb{P}^{\,j}$ is lower triangular.
By assuming $\alpha>0$, the transformation matrix is invertible since
it is lower triangular and all the diagonal elements are
positive. Indeed $\mathbb{P}^{\,j}_{n,n} =a^{-(n+1)}$ for all
$n\in\Lambda_{N_v}$.

  The transformation matrix $\mathbb{P}^{\,j}$ has been derived from
  the analytic calculation of the integrals \eqref{eq:transfM} using
  the orthogonality properties of the Hermite polynomials together
  with the formula for the $n$-th Hermite polynomials
  \cite[p. 255]{MOS66},
    \begin{equation*}
    \begin{aligned}
      H_n(Ay+B) & = \sum_{k=0}^n \binom{n}{k}(2B)^{n-k} H_k(Ay) \\
      & = \sum_{k=0}^n \binom{n}{k}(2B)^{n-k}\sum_{j=0}^{\lfloor k/2 \rfloor} A^{k-2j}(A^2-1)^j\binom{k}{2j}\dfrac{(2j)!}{j!} H_{k-2j}(y),
    \end{aligned}
  \end{equation*}
  valid for all $A\in\mathbb{R}\setminus\{0,1\}$,
  $B\in\mathbb{R}\setminus\{0\}$.

\subsection{Physics-based choice of the Hermite parameters \texorpdfstring{$\alpha$}{} and \texorpdfstring{$u$}{}}\label{sec:PhysB}

The adaptivity criterion for the choice of the Hermite parameters
should ideally minimize, at any given time $t^j$, the error between
the numerical solution $f^N_{j}\in V^N_{j-1}$ and the exact solution
$f^{\ex}(t^{j})$.
Since the exact solution or an error indicator are in general not
available, the adaptivity criterion for the selection of $(\alpha,u)$
remains one of the main challenges of an adaptive algorithm. We suggest here a criterion that is based on physics considerations.
 
At the initial time, the Hermite parameters $(\alpha_0,u_0)$ are
chosen to minimize the projection error of the initial condition
$f(t^0,x,v)$ onto the finite dimensional space associated with
$(\alpha_0,u_0)$ for a fixed number of modes $N$. (Note that in the examples treated in section 5 we start from a Maxwellian plasma that can be described exactly with only one Hermite mode, hence the projection error is exactly zero.)
For $t^j>0$, one might envision a choice of the Hermite parameters
driven by the minimization of the projection error of the approximate
solution.  However, this criterion does not ensure that
$\norm{f^{\ex}(t^{j})-f^N_{j}}$ is minimized, and, moreover, it
``tends'' to select as new Hermite parameters the $(\alpha,u)$ that are closest
(in terms of Euclidean distance) to the old parameters.

By contrary, a physics-based choice aims at ``correcting'' the Hermite
parameters according to the evolving dynamics.
The idea of the physics-based adaptivity criterion that we propose is
to select, at each time step, the shift parameter $u$ as the average
velocity of species s, and the scaling parameter $\alpha$ as the average thermal
velocity.
In particular, we set both quantities as the ratio of spatial averages
of suitable physical quantities, see \eqref{eq:physAU}.
The rationale for using the ratio of spatial averages is that we aim
at capturing the amplitude and position of the Maxwellian reproducing
the average behavior of the solution, rather than the oscillations
associated with the deviation from a Maxwellian behavior. Nevertheless, in the future we plan to extend the algorithm presented here to include spatially dependent quantities, i.e. $u(t,x)$ and $\alpha(t,x)$.

Let $f^N_j\in V^N_{j-1}\times W$ be the numerical solution of the
semi-discrete Vlasov--Poisson problem \eqref{eq:VFhL2w} at time $t^{j}$.
We update the Hermite parameters as
\begin{equation}\label{eq:physAU}
  \begin{aligned}
    & u_{j}=\dfrac{1}{n_j}\int_{\Omega}vf_j^N(x,v) \,\der v \der x,\qquad\mbox{with}\quad n_j:=\int_{\Omega}f_j^N(x,v)\, \der v\der x\\[.5em]
    & \alpha_{j} = \sqrt{2}\sqrt{\dfrac{1}{n_j}\int_{\Omega}(v-u_{j})^2f^N_j(x,v)\, \der v \der x}.
  \end{aligned}
\end{equation}
Using the spectral expansion \eqref{eq:fNexp} of $f^N_j$ in
$V^N_{j-1}$ given by
\begin{equation*}
  f^N_j(x,v) = \sum_{n\in\Lambda_{N_v}} \widehat{C}^{j-1}_{n}(x)\psi^{\alpha_{j-1},u_{j-1}}_n(v),
\end{equation*}
and the recurrence relation of the Hermite functions
\begin{equation}\label{eq:Hrec}
  v\,\psi_n^{\alpha,u}(v)=\alpha\sqrt{\dfrac{n+1}{2}}\,\psi_{n+1}^{\alpha,u}(v)+
  \alpha\sqrt{\dfrac{n}{2}}\,\psi_{n-1}^{\alpha,u}(v)+u\,\psi_n^{\alpha,u}(v),\qquad\forall\, n\in\mathbb{N},
\end{equation}
the Hermite parameters in \eqref{eq:physAU} can be efficiently updated as
\begin{equation*}
  u_{j} = u_{j-1} + \dfrac{\alpha_{j-1}}{\sqrt{2}}\,\dfrac{\int_{\Omega_x}\widehat{C}^{j-1}_{1}(x)\,\der x}{\int_{\Omega_x}\widehat{C}^{j-1}_{0}(x)\,\der x},\qquad
  \alpha_{j} = \alpha_{j-1}\sqrt{1+\sqrt{2}\,\dfrac{\int_{\Omega_x}\widehat{C}^{j-1}_{2}(x)\,\der x}{\int_{\Omega_x}\widehat{C}^{j-1}_{0}(x)\,\der x}-
    \left(\dfrac{\int_{\Omega_x}\widehat{C}^{j-1}_{1}(x)\,\der x}{\int_{\Omega_x}\widehat{C}^{j-1}_{0}(x)\,\der x}\right)^2}.
\end{equation*}

Since the physics-based definition of the Hermite parameters depends
on the first two moments of the distribution function and on the
kinetic energy, the fact that the AW scheme satisfies exactly the
conservation laws of total mass, momentum and energy (which we
anticipate here, for details see Section~\ref{sec:conservation:properties}) is of fundamental
importance.
Moreover, since the spectral numerical discretization
\eqref{eq:VFhL2w} is not generally positivity preserving (which in
fact signals lack of resolution), the Hermite parameter $\alpha$ might
become complex.
In the numerical experiments of Section \ref{sec:NumExp}, the updated
$\alpha_j$ is taken as the real part of \eqref{eq:physAU}.

Note that the Hermite parameters are not directly dependent on the value of the
collisional term since the latter is constructed not to act on the
first three modes of the Hermite expansion.  However, since the
evolution of the expansion coefficients depends on the collisional
term, different choices of $\nu$ might yield a different evolution of
the Hermite parameters.

\subsection{Implementation and computational complexity of the adaptive algorithm}
\label{sec:implementation}
The adaptive pseudo-algorithm reads as follows.
\begin{algorithm}[H]
  \caption{Input: $f_0^N$, $\alpha_0$, $u_0$, $N_t$}\label{alg:adaptive}
  \begin{algorithmic}
  \State Initialize $E^N_0\in W$ via Eq. \eqref{eq:ode}
    \For{$1\leq j\leq N_t$}
    \State Solve problem \eqref{eq:VFhL2w} in the interval $T_{j-1}$, with velocity approximation space $V^N_{j-1}$, to obtain\\ \quad\, $f_{j}^N\in V^N_{j-1}\times W$ and $E_{j}^N\in W$
    \State Compute the updated Hermite parameters $\alpha_{j}$ and $u_{j}$ from $f^N_j$ using \eqref{eq:physAU}
    \State Compute $\mathcal{P}^N_{j}f^N_{j}\in V^N_{j}\times W$ using \eqref{eq:L2omegaOrt}
    \State Set $\mathcal{P}^N_{j}f^N_{j}$ as initial condition for the distribution function in the interval $T_j$
    \EndFor
  \end{algorithmic}
\end{algorithm}
In the non-adaptive strategy, system \eqref{eq:ode} is solved on a velocity approximation space $V^N$ that does not change in time. Compared to the non-adaptive approach,
the adaptive Algorithm~\ref{alg:adaptive} incurs an
increased computational cost associated with updating the Hermite
parameters, building the lower triangular matrix $\mathbb{P}$, and
computing the spectral coefficients of the discrete distribution function projected into the updated approximation space.
The construction of the transfer matrix $\mathbb{P}$ requires
one operation for each of the $(N_v+1)N_v/2$ matrix entry; this is possible
even in the case when both Hermite parameters are changing,
by computing the elements of $\mathbb{P}$ in a suitable order.
The computational complexity of the construction of $\mathbb{P}$ is, therefore,
$O(N_c^2)$, where $N_c$ is defined as the number $N_v$ of Hermite modes 
times the number of plasma species.
Updating the expansion coefficients $P_{n}(x)=\sum_{m\leq n}
\mathbb{P}_{n,m} C_{m}(x)$ for all
$n\in\Lambda_{N_v}$ requires $O(N_c^2)$
operations every time the basis is updated.
Despite this extra computational costs,
the bulk of the computational effort during the simulation results presented here -- for both non-adaptive and adaptive algorithms -- is taken by the nonlinear solver and the cost of the adaptive step is negligible. This is confirmed in Fig. \ref{fig:TSIcpu}, which shows the computational time versus number of degrees of freedom for the numerical tests presented in Section 5.2.2. As discussed in Sec. 5, the numerical algorithm uses a Jacobian-Free Newton Krylov solver with GMRES for the
inner (linear) iterations, preconditioned via incomplete LU. One can also notice the almost optimal (linear) scaling of the algorithm versus number of degrees of freedom.

\begin{figure}[ht!]
\centering
    \includegraphics[width=0.6\textwidth]{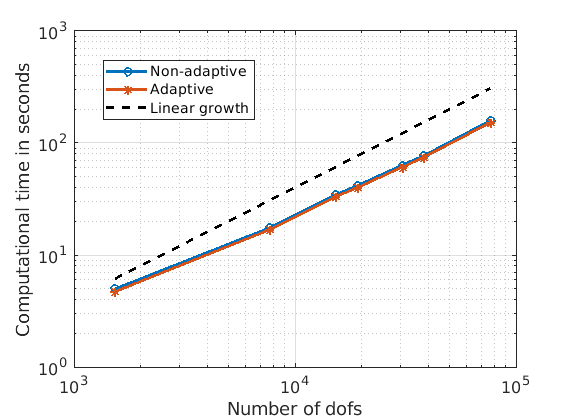}
\caption{Numerical test presented in Sec. 5.2.2: plot of the computational time (in seconds) of the final simulation time versus number of degrees of freedom (dofs) $N_s(N_x+1)N_v$ where $N_s=3$ is the number of species, $N_x=50$ and $N_v\in\{10,50,100,125,200,250,500\}$. All other parameters are specified in Section 5.2.2.}
\label{fig:TSIcpu}
\end{figure}

We emphasize that the optimization and efficiency of our numerical
MATLAB code is not the key point of this work, and we acknowledge
that more efficient implementations are conceivable.

\section{Conservation properties}\label{sec:conservation:properties}
One of the most important results of this paper is that the adaptive fully discrete scheme guarantees exact conservation of
mass, momentum and energy.
In order to prove this statement we proceed in two steps: first we show that, within each time interval, the conservation of mass, momentum and energy are satisfied;
second we prove that the operator
$\mathcal{P}^N_{j}$ leaves the conserved quantities invariant. The first step of the proof boils down to show that the non-adaptive spectral method preserves the first three moments of the system and it is a well-known result in the spectral approximation of the Vlasov--Poisson and Vlasov--Maxwell equations using AW Hermite functions, see \cite[Section 3]{GLD15}. Here we briefly summarize the main properties and derivations that will be used in the second step.

Let us define a \emph{conserved quantity} $\mathcal{M}$ as the function
$\mathcal{M}\circ g:t\in T \mapsto
\mathcal{M}(g(t))\in\mathbb{R}$, such that $\mathcal{M}\circ g\in C^1(T)$ and satisfies
\begin{equation*}
  \dfrac{d\mathcal{M}(g(t))}{dt}=0,
\end{equation*}
whenever $g(t)=f(t,\cdot,\cdot)\in L^2(\Omega)$ is the (weak)
solution of the Vlasov--Poisson problem \eqref{eq:VlasovPoisson}.
As a first step, we need to show that the advancement in time
preserves the conserved quantities of the continuous model: if
$f^N_{j+1}\in V^N_j\times W$ is the numerical solution of the
semi-discrete problem \eqref{eq:VFhL2w} at time $t^{j+1}$, then the
quantity $\mathcal{M}(f^N_{j+1})$ is still conserved over time,
i.e. $\mathcal{M}(f^N_{j+1})=\mathcal{M}(f^N_{j})$.
In greater detail, let us consider the following conservation
properties.
\begin{enumerate}
\item \textbf{Mass}. Let the mass of the density function $f^s$ for
  the species $s$, be defined as
  \begin{align*}
    M(f^s):=m^s\int_{\Omega} f^s(t,x,v)\,\der x \der v.
  \end{align*}
  Note that the mass is proportional to the zero-th moment of the
  distribution function $f^s$.
  The mass of the approximate function $f^{s,N}_{j+1}\in V^N_{j}$ is
  \begin{align}\label{eq:discrMass}
    M(f^{s,N}_{j+1}) 
    = m^s\sum_{n}\int_{\Omega}\widehat{C}^{s,{j+1}}_{n}(x)\psi_n^{\alpha^s_j,u^s_j}(v)\,\der v\der x
    = m^s\alpha^s_j\,\int_{\Omega_x}\widehat{C}^{s,{j+1}}_{0}(x)\,\der x.
  \end{align}
Integrating \eqref{eq:ode}, for $n=0$, with respect to $x\in\Omega_x$ and using periodic boundary conditions,
it holds
  \begin{equation*}
    \begin{aligned}
      M(f^{s,N}_{j+1})= m^s\alpha^s_j\int_{\Omega_x}\widehat{C}^{s,j+1}_{0}(x)\,\der x
      = m^s\alpha^s_j\int_{\Omega_x}\widehat{C}^{s,j}_{0}(x)\,\der x=M(f^{s,N}_j).
    \end{aligned}
  \end{equation*}

\item \textbf{Momentum}. The total momentum of $f$ is defined as
  \begin{align*}
    P(f):=\sum_s m^s\int_{\Omega} v f^s(t,x,v)\,\der x\der v,
  \end{align*}
  and it is therefore proportional to the first moment of the
  distribution function.
  The momentum of the approximate solution $f^{s,N}_{j+1}\in
  V^N_{j}\times W$ is
  \begin{equation}\label{eq:discrMom}
    \begin{aligned}
      P(f^{s,N}_{j+1}) 
      & 
      = m^s\sum_{n}\int_{\Omega}\widehat{C}^{s,j+1}_{n}(x)\,v\,\psi_n^{\alpha^s_j,u^s_j}(v)\,\der v\der x\\
      &	= m^s\alpha^s_j\left(u^s_j\,\int_{\Omega_x}\widehat{C}^{s,j+1}_{0}(x)\,\der x +\dfrac{\alpha_j^s}{\sqrt{2}}\,\int_{\Omega_x}\widehat{C}^{s,j+1}_{1}(x)\,\der x\right)\\
      & = u_j^s\,M(f^{s,N}_{j+1}) + m^s\dfrac{(\alpha_j^s)^2}{\sqrt{2}}\,\int_{\Omega_x}\widehat{C}^{s,j+1}_{1}(x)\,\der x.
    \end{aligned}
\end{equation}
Integrating the Vlasov equation \eqref{eq:ode} in $\Omega_x$ and using periodic boundary conditions results in
\begin{equation*}
  \begin{aligned}
    P(f^{N}_{j+1}) 
    &= \sum_{s} u_j^s M(f^{s,N}_{j+1}) + \sum_{s}m^s\dfrac{(\alpha_j^s)^2}{\sqrt{2}}\int_{\Omega_x}
    \widehat{C}^{s,j}_{1}(x)\,\der x\\
    &\quad + \Delta t_j\sum_{s}\alpha_j^s\,q^s
    \int_{\Omega_x}\widehat{C}^{s,j+1/2}_{0}(x)\,E^{N}_{j+1/2}(x)\,\der x.
  \end{aligned}
\end{equation*}
Substituting the semi-discrete Poisson equation \eqref{eq:ode} yields
\begin{equation*}
  \Delta t_j \sum_{s}\alpha_j^s q^s
  \int_{\Omega_x}\widehat{C}^{s,j+1/2}_{0}(x)\,E^{N}_{j+1/2}(x)\,\der x
  = \Delta t_j\int_{\Omega_x}\partial_x E^N_{j+1/2}\,E^{N}_{j+1/2}\,\der x=0,
\end{equation*}
on account of the periodic boundary conditions.
Hence, $P(f^{N}_{j+1})=P(f^{N}_{j})$.
\item \textbf{Total Energy}.
  Let $\mathcal{E}_k$ be the kinetic energy, defined as
  \begin{align*}
    \mathcal{E}_k(f^s):=\dfrac{m^s}{2}\int_{\Omega} v^2 f^s(t,x,v)\,\der x\der v,
    \qquad \forall\,s\in\{e,i\}.
  \end{align*}
  The kinetic energy associated with the approximate solution
  $f^{s,N}_{j+1}\in V^N_{j}\times W$ at time $t^{j+1}$ is
  \begin{equation}\label{eq:discrKinE}
    \begin{aligned}
    \mathcal{E}_k(f^{s,N}_{j+1})  := &\,
      \dfrac{m^s}{2}\sum_{n}\int_{\Omega}\widehat{C}^{s,j+1}_{n}(x) v^2\,\psi_n^{\alpha_j^s,u_j^s}(v)\,\der v\der x\\
       = &\, \dfrac{m^s\alpha_j^s}{2}\left((u_j^s)^2+\dfrac{(\alpha_j^s)^2}{2}\right)\,\int_{\Omega_x}\widehat{C}^{s,j+1}_{0}(x)\,\der x\\
        & + \dfrac{m^s\alpha_j^s}{2}\left[\sqrt{2} u^s_j\alpha_j^s\,\int_{\Omega_x}\widehat{C}^{s,j+1}_{1}(x)\,\der x+\dfrac{(\alpha_j^s)^2}{\sqrt{2}}\,\int_{\Omega_x}\widehat{C}^{s,j+1}_{2}(x)\,\der x\right],
    \end{aligned}
  \end{equation}
  where the recurrence relation \eqref{eq:Hrec} has been used twice.
  Using the Vlasov--Poisson equations
  \eqref{eq:ode}, results in
  \begin{equation*}
    \mathcal{E}_k(f^{s,N}_{j+1})
    = \mathcal{E}_k(f^{s,N}_{j})+\Delta t_j q^s \alpha_j^s\int_{\Omega_x}\left(\dfrac{\alpha_j^s}{\sqrt{2}}\widehat{C}^{s,j+1/2}_{1}(x)
      + u^s_j\widehat{C}^{s,j+1/2}_{0}(x)\right)\,E^{N}_{j+1/2}(x)\,\der x.
  \end{equation*}
  The potential energy associated with the discrete electric field is
  \begin{equation}\label{eq:discrPotE}
  \mathcal{E}_p(E^N_{j+1})
      = \dfrac{\epsilon_0}{2} \int_{\Omega_x}(E^N_{j+1}(x))^2\,\der x.
  \end{equation}
  To derive a discrete evolution of the potential energy let us first
  consider the current density
  \begin{equation*}
    J^s(t,x) = q^s \int_{\Omega_v} v f^s(t,x,v)\,\der v.
  \end{equation*}
  The current density associated with the numerical solution
  $f^{s,N}_{j+1}\in V^N_{j}\times W$ is
  \begin{equation*}
    J(f^{s,N}_{j+1})(x) = q^s\alpha_j^s
    \left(u_j^s\widehat{C}^{s,j+1}_0(x)+\dfrac{\alpha_j^s}{\sqrt{2}}\widehat{C}^{s,j+1}_1(x)\right).
  \end{equation*}  
  The variation of the potential energy reads
  \begin{equation}\label{eq:Enpot}
    \begin{aligned}	
      \mathcal{E}_p(E^N_{j+1})-\mathcal{E}_p(E^N_{j})
      & = \dfrac{\epsilon_0}{2} \int_{\Omega_x}(E^N_{j+1}(x))^2\,\der x-
      \dfrac{\epsilon_0}{2} \int_{\Omega_x}(E^N_{j}(x))^2\,\der x\\
      & = \epsilon_0 \int_{\Omega_x}(E^N_{j+1}(x)-E^N_{j}(x))E^N_{j+1/2}(x)\,\der x.
    \end{aligned}
  \end{equation}
  Using the midpoint rule for the temporal discretization of
  Amp\`ere's equation
  \begin{equation*}
    \epsilon_0\,\partial_t E +\sum_s J^s(t,x)=0,
  \end{equation*}
  and substituting in \eqref{eq:Enpot} we obtain
  \begin{equation*}
    \begin{aligned}	
      \mathcal{E}_p(E^N_{j+1})-\mathcal{E}_p(E^N_{j})
      & = \epsilon_0 \int_{\Omega_x}(E^N_{j+1}(x)-E^N_{j}(x))E^N_{j+1/2}(x)\,\der x\\
      & = -\Delta t_j \sum_s q^s\alpha_j^s
      \int_{\Omega_x}\left(u_j^s\widehat{C}^{s,j+1/2}_0(x)+
        \dfrac{\alpha_j^s}{\sqrt{2}}\widehat{C}^{s,j+1/2}_1(x)\right) E^N_{j+1/2}(x)\,\der x.
    \end{aligned}
  \end{equation*}
  Hence, it can be inferred that the total energy is conserved in the
  discrete problem, namely
  \begin{equation*}
    \begin{aligned}	
      \sum_s \big(\mathcal{E}_k(f^{s,N}_{j+1}) - \mathcal{E}_k(f^{s,N}_{j})\big)
      + \mathcal{E}_p(E^N_{j+1})-\mathcal{E}_p(E^N_{j})=0.
    \end{aligned}
  \end{equation*}
\end{enumerate}
As a second step, we show that the operator from
Definition~\ref{def:Pj} preserves the invariants of the system.
\begin{proposition}\label{prop:CPAdapt}
  Let $\mathcal{P}^N_{j}$ be the operator from Definition~\ref{def:Pj}
  and let $f^N\in V^N_{j-1}\times W$. 
  It holds
  \begin{equation*}
    \mathcal{M}(\mathcal{P}^N_{j}f^N)=\mathcal{M}(f^N),
  \end{equation*}
  where the functional $\mathcal{M}$ is either mass, momentum, kinetic
  and potential energy as defined in \eqref{eq:discrMass}, \eqref{eq:discrMom},
  \eqref{eq:discrKinE}, and \eqref{eq:discrPotE}, respectively.
\end{proposition}

\begin{proof}
  Let {$\{\widehat{C}_{n}\}_{n}$} denote the spectral coefficients
  of $f^N$ and let {$\{\widehat{P}_{n}\}_{n}$} be the spectral
  coefficients of $\mathcal{P}^N_{j}f^N$.  
  The matrix $\mathbb{P}^{\,j}$ defined in \eqref{eq:transfM} can be written in a more compact way
  as follows.  For $n,m\in\Lambda_{N_v}$ fixed, with $n\geq m$, let us
  defined $\Lambda_j^{n,m}\subset \mathbb{N}$ as the set of numbers
  $\ell\in[m,n]\cap$ such that $\ell-m$ is even and
  \begin{equation*}
    \ell=\left\{
      \begin{array}{ll}
	n, &\mbox{if}\quad u_{j-1}=u_j,\\[.5em]
	m, &\mbox{if}\quad \alpha_{j-1}=\alpha_j.
      \end{array}\right.
  \end{equation*}
  Then the entries of $\mathbb{P}^{\,j}$ read, (with the convention
  that $0^0=1$),
  \begin{equation*}
    \mathbb{P}^{\,j}_{n,m} = K_{n,m}\dfrac{1}{a^{m+1}}
    \sum_{\ell\in\Lambda^{n,m}_j}\dfrac{1}{(n-\ell)!
      \left(\frac{\ell-m}{2}\right)!}\left(-\dfrac{2b}{a}\right)^{n-\ell}\left(\dfrac{1}{a^2}-1\right)^{\frac{\ell-m}{2}}.
  \end{equation*}
  Let us consider each conserved quantity at a time.
  Using the definitions above, we have
  \begin{itemize}
  \item[(i)] \textbf{Mass}. 
    According to \eqref{eq:discrMass}, the mass of $f^N\in V^N_{j-1}\times {W}$ is
    $M(f^N)={m\alpha_{j-1}\int_{\Omega_x}\widehat{C}_{0}(x)\,\der x}$ and the mass of
    $\mathcal{P}^N_{j}f^N\in V^N_{j}\times {W}$ is
    $M(\mathcal{P}^N_{j}f^N)={m\alpha_{j}\int_{\Omega_x}\widehat{P}_{0}(x)\,\der x}$. 
    Using the definition of the spectral coefficients from
    \eqref{eq:coeffPj} results in
    \begin{align*}
      M(\mathcal{P}^N_{j}f^N)={m \alpha_{j} \sum_{n\in\Lambda_{N_v}}\int_{\Omega_x}\mathbb{P}^{\,j}_{0,n}\widehat{C}_{n}(x)\,\der x}.
    \end{align*}
    Since the matrix $\mathbb{P}^{\,j}$ is lower triangular,
    $\mathbb{P}^{\,j}_{0,n}=0$ for all $n\neq 0$.
    Moreover, the definition of $\mathbb{P}^{\,j}_{0,0}$ gives,
  \begin{equation}\label{eq:CPMAdapt}
    M(\mathcal{P}^N_{j}f^N)={m \alpha_{j}\,  \int_{\Omega_x}\mathbb{P}^{\,j}_{0,0}\widehat{C}_{0}(x)\,\der x =
    m \alpha_{j} \dfrac{\alpha_{j-1}}{\alpha_{j}}\, \int_{\Omega_x}\widehat{C}_{0}(x)\,\der x} = M(f^N).
  \end{equation}
\item[(ii)] \textbf{Momentum}. 
  The total momentum of $f^N\in V^N_{j-1}\times {W}$ is defined as in \eqref{eq:discrMom}, namely
  \begin{equation*} 
    P(f^N) = u_{j-1}\,M(f^N)+{m\dfrac{\alpha_{j-1}^2}{\sqrt{2}}\,\int_{\Omega_x}\widehat{C}_{1}(x)\,\der x}.
  \end{equation*}
  The total momentum of $\mathcal{P}^N_{j}f^N\in V^N_{j}\times {W}$
  satisfies
  \begin{equation*} 
    P(\mathcal{P}^N_{j}f^N) = u_{j}\,M(\mathcal{P}^N_{j}f^N) + { m\dfrac{\alpha_{j}^2}{\sqrt{2}}\,\int_{\Omega_x}\widehat{P}_{1}(x)\,\der x}
    \overset{\text{\eqref{eq:CPMAdapt}}}{=} u_{j}\,M(f^N) + {m\dfrac{\alpha_{j}^2}{\sqrt{2}}\int_{\Omega_x}\sum_{n\leq 1}\mathbb{P}^{\,j}_{1,n}\widehat{C}_{n}(x)\,\der x}.
  \end{equation*}
  Using the expression for the entries of
  the {transformation} matrix $\mathbb{P}^{\,j}$,
  \begin{equation*}
    \mathbb{P}^{\,j}_{1,0}=-\sqrt{2}\,\dfrac{\alpha_{j-1}}{\alpha_j^2}(u_{j}-u_{j-1}),\qquad
    \mathbb{P}^{\,j}_{n,n}=\left(\dfrac{\alpha_{j-1}}{\alpha_j}\right)^{n+1},\quad \forall\,n\in\Lambda_{N_v},
  \end{equation*}
  we can infer
  \begin{equation*}
    \begin{aligned}
      P(\mathcal{P}^N_{j}f^N) 
      & = u_{j}\,M(f^N) +{m\dfrac{\alpha_{j}^2}{\sqrt{2}}\left(
	-\sqrt{2}\,\dfrac{\alpha_{j-1}}{\alpha_j^2}(u_{j}-u_{j-1})\int_{\Omega_x}\widehat{C}_{0}(x)\,\der x+
	\left(\dfrac{\alpha_{j-1}}{\alpha_j}\right)^{2}\int_{\Omega_x}\widehat{C}_{1}(x)\,\der x\right)}\\
      & = u_{j}\,M(f^N) - u_{j}\,M(f^N) + u_{j-1}\,M(f^N)
      + {m \dfrac{\alpha_{j-1}^2}{\sqrt{2}}\,\int_{\Omega_x}\widehat{C}_{1}(x)\,\der x}=P(f^N).
    \end{aligned}
  \end{equation*}
  
\item[(iii)] \textbf{Energy}.
  Let us first consider the kinetic energy $\mathcal{E}_k$ defined in \eqref{eq:discrKinE}.
  For $f^N\in V^N_{j-1}\times {W}$, we have
  \begin{equation*} 
    {\mathcal{E}_k(f^N) = \dfrac12  m\alpha_{j-1}\left(\dfrac{\alpha_{j-1}^2}{\sqrt{2}}\,\int_{\Omega_x}\widehat{C}_{2}(x)\,\der x+
      u_{j-1}\alpha_{j-1}\sqrt{2}\,\int_{\Omega_x}\widehat{C}_{1}(x)\,\der x+\left(u_{j-1}^2+\dfrac{\alpha_{j-1}^2}{2}\right)\int_{\Omega_x}\widehat{C}_{0}(x)\,\der x\right)}.
  \end{equation*}
  Using the definition of the transformation matrix
  $\mathbb{P}^{\,j}_{n,m}$ for $m\leq n\leq 2$,
  \begin{equation*}
    \mathbb{P}^{\,j}_{2,0}=\dfrac{\sqrt{2}}{2}\,\dfrac{\alpha_{j-1}}{\alpha_j}\left(\dfrac{2(u_{j}-u_{j-1})^2}{\alpha_j^2}+\dfrac{\alpha_{j-1}^2}{\alpha_j^2}-1\right),\qquad
    \mathbb{P}^{\,j}_{2,1}=-2\,\dfrac{\alpha_{j-1}^2}{\alpha_j^3}(u_{j}-u_{j-1}),
  \end{equation*}
  the kinetic energy associated with $\mathcal{P}^N_{j}f^N\in
  V^N_{j}\times {W}$ satisfies
  \begin{equation*} 
    \begin{aligned}
      \mathcal{E}_k(\mathcal{P}^N_{j}f^N) 
      = &\, \dfrac{m\alpha_{j}^3}{2\sqrt{2}}\int_{\Omega_x}\sum_{n\leq 2}\mathbb{P}^{\,j}_{2,n}\widehat{C}_{n}(x)\,\der x+
      mu_{j}\alpha_{j}^2\dfrac{\sqrt{2}}{2}\int_{\Omega_x}\sum_{n\leq 1}\mathbb{P}^{\,j}_{1,n}\widehat{C}_{n}(x)\,\der x \\
      & + \dfrac12 m\alpha_{j}\left(u_{j}^2+\dfrac{\alpha_{j}^2}{2}\right)\int_{\Omega_x}\mathbb{P}^{\,j}_{0,0}\widehat{C}_{0}(x)\,\der x\\
      = &\,\dfrac12 m\bigg(\alpha_{j-1}u_{j-1}^2+\dfrac{\alpha_{j-1}^3}{2}\bigg)\,\int_{\Omega_x}\widehat{C}_{0}(x)\,\der x\\
      & + \dfrac{\sqrt{2}}{2}m\alpha_{j-1}^2u_{j-1}\,\int_{\Omega_x}\widehat{C}_{1}(x)\,\der x + \dfrac{m\alpha_{j-1}^3}{2\sqrt{2}}\int_{\Omega_x}\widehat{C}_{2}(x)\,\der x = \mathcal{E}_k(f^N).
    \end{aligned}
  \end{equation*}
  Concerning the potential energy \eqref{eq:discrPotE}, since the operator $\mathcal{P}^N_{j}$ is independent of the spatial discretization and the potential energy involves only the first Hermite mode, conservation follows from $\alpha^s_{j}\,\widehat{P}^s_{0}(x)=\alpha^s_{j}\,\mathbb{P}^{\,j,s}_{0,0}\,\widehat{C}_{0}^s(x) = \alpha^s_{j-1}\,\widehat{C}^s_{0}(x)$ for all $x\in\Omega_x$.
\end{itemize}
\end{proof}
Note that the conservation properties of the fully discrete adaptive
scheme are independent of the choice of the Hermite parameters.

\section{Numerical experiments}
\label{sec:NumExp}

For the numerical simulations of this work we consider a spectral discretization in space using Fourier basis functions.
Let $\eta_k(x):=e^{\frac{2\pi ik}{L}x}$, $x\in\Omega_x:=[0,L]$ denote the Fourier basis function with wavenumber $k$ in $\Lambda_{N_x}:=\{k\in\mathbb{Z}:\,-N_x\leq k\leq N_x\}$.
Let $N:=(N_v,N_x)\in\mathbb{N}\times\mathbb{N}$, we consider the approximation space $W^N:=\SPAN{\{\eta_k(x)\}_{k\in\Lambda_{N_x}}}$, so that, in the temporal interval $T_j$, the approximated functions $f_j^N$ and $E_j^N$
can be represented in their phase space spectral expansion in
$V^N_{j}\times W^N$ and in $W^N$, respectively, as
\begin{equation*}
  f_j^N(x,v) = \sum_{n\in\Lambda_{N_v}}\sum_{k\in\Lambda_{N_x}} \widehat{C}_{n,k}^j\,\psi^{\alpha_j,u_j}_n(v)\,\eta_k(x),
\end{equation*}
\begin{equation*}
  E_j^N(x) = \sum_{k\in\Lambda_{N_x}} \widehat{E}_k^j\,\eta_k(x).
\end{equation*}
The algebraic equations satisfied by the spectral coefficients $\{\widehat{C}_{n,k}^{j}\}_{(n,k)\in\Lambda_{N_v}\times\Lambda_{N_x}}$
    and $\{\widehat{E}_k^j\}_{k\in\Lambda_{N_x}}$, in each $T_j$, are reported in Appendix~\ref{sec:appA}.

The implementation of the algorithm for the solution of \eqref{eq:ode} is the same as
for the non-adaptive strategy.
We follow Refs.~\cite{CDBM16,GLD15,VD16} and consider a
Jacobian-Free Newton Krylov solver \cite{kelley} with
GMRES for the inner (linear) iterations.
Preconditioning strategies which are useful to reduce the number of
iterations per time step are discussed in Ref. \cite{GLD15}.
Since the change of $(u,\alpha)$ basis is performed at the end of a
time step, it amounts to a trivial modification of the implementation of the non-adaptive algorithm.

In order to efficiently update the Hermite parameters
and avoid rounding error associated with small variations of the
Hermite parameters between two consecutive time steps, we fix two
tolerances $u_{\mathrm{tol}}$ and $\alpha_{\mathrm{tol}}$.
The adaptive Algorithm~\ref{alg:adaptive} computes the new Hermite parameters
at time $t^j$ based on the strategy described in
Section~\ref{sec:PhysB}, but the update is performed only if
\begin{equation}
  \label{eq:AdaptTol}
  |u_{j-1}^s-u_{j}^s| \geq u_{\mathrm{tol}},\quad\mbox{and}\quad
  \dfrac{|\alpha_{j-1}^s-\alpha_{j}^s|}{|\alpha_{j-1}^s|} \geq \alpha_{\mathrm{tol}}.
\end{equation}
If not otherwise specified, the parameters for the nonlinear iterative
solver are set as follows: 
the maximum number of nonlinear iterations is $500$; 
the maximum number of linear iterations is $1000$; 
the maximum error tolerance for residual in the inner iteration is
$\eta_{\max}= 0.9$; if not otherwise specified, we take $10^{-9}$ as
absolute and relative error tolerances of the nonlinear iteration.

Last, all the numerical experiments reported in this paper used the collisional operator defined in Sec.~\ref{sec:Coll}. We have also performed some numerical experiments where we compared the filtering technique of \cite{FilbetXiong} with the artificial collisional term used in this work on two test cases. 
The first test case is the one used in this work (Section 5.2) and analogous to the test of \cite{Di2019} in Section 4.3. The second test case is the one proposed in Section 4.2 of \cite{FilbetXiong}. For both test cases, the collisional operator used in this work and the filtering technique of \cite{FilbetXiong} produced results that are qualitatively similar (although we note that in the first test unphysical oscillations were wider for the filtering technique versus our collisional operator with $\nu=5$) and hence are not reported here.

A detailed numerical study of the effect of the artificial collisional operator on the adaptive and non-adaptive methods is reported in Appendix~\ref{sec:NumExpColl}.

\subsection{Manufactured solution}
We begin by using the method of manufactured solutions to study the
new algorithm under controlled conditions.
Let us consider the domain $\Omega:=\Omega_x\times\Omega_v$ with
$\Omega_x=[0,2\pi]$ and $\Omega_v=\mathbb{R}$, and the temporal
interval $T=[0,1]$. Let us assume that the exact solution is given by
\begin{equation}\label{eq:fex}
   f^{\ex}(t,x,v)=\big(2-\cos(2x-2\pi t)\big)\pi^{-1/2}e^{-\left(\frac{v-w(t)}{\beta(t)}\right)^2}, \qquad (x,v)\in\Omega,\; t\in T,
\end{equation}
where $w(t)$ and $\beta(t)\neq 0$ will be specified case by case in our
numerical experiments.
Assume that $f^e=f^i= f^{\ex}$ from \eqref{eq:fex} is the exact solution
of \eqref{eq:VlasovPoisson} for all $(x,v)\in\Omega$ and $t\in T$.
This entails that $E(t,x)\equiv 0$ for all $x\in\Omega_x$, $t\in T$ if
we consider the homogeneous Dirichlet boundary condition $E(t,0)=0$ for all $t\in T$.
Therefore, $f^{\ex}$ in \eqref{eq:fex} satisfies a simplified
Vlasov--Poisson problem, namely the scalar advection problem:
For $f_0:= f^{\ex}(0,x,v)\in V_{|_{t=0}}$, and $S\in
  C^0(T;L^2(\Omega))$, find $f(t,x,v)\in C^1(T;L^2(\Omega))\cap
  C^0(T;V)$ such that
\begin{equation}\label{eq:VlasovPoissonSimpl}
  \begin{aligned}
    \partial_tf+v\,\partial_x f =S, &\qquad\mbox{in}\;\Omega\times T,\\
    f(0,x,v)=f_0,                   &\qquad\mbox{in}\;\Omega.
  \end{aligned}
\end{equation}
with
\begin{equation}\label{eq:rhs}
  \begin{aligned}
    S(t,x,v) := 
    & \,\partial_t f^{\ex}+v\,\partial_x  f^{\ex} =
    2(v-\pi)\sin(2x-2\pi t)\pi^{-1/2}e^{-\left(\frac{v-w}{\beta}\right)^2} \\
    & + 2(2-\cos(2x-2\pi t))\pi^{-1/2}\dfrac{v-w}{\beta} e^{-\left(\frac{v-w}{\beta}\right)^2}
    \left(\dfrac{d_t w}{\beta}+\dfrac{v-w}{\beta}\dfrac{d_t\beta}{\beta}\right),
  \end{aligned}
\end{equation}
where $d_tw$ and $d_t\beta$ are the time derivative of $w$
  and $\beta$, respectively.
The functions $ f^{\ex}(t,\cdot,\cdot),\, S(t,\cdot,\cdot)\in L^2(\Omega)$
for every $t\in T$, admit the spectral expansion
\begin{equation}\label{eq:expansion}
  \begin{aligned}
     f^{\ex}(t,x,v) &= \sum_{n\in\mathbb{N}}\sum_{k\in\mathbb{Z}} C_{n,k}(t)\,\psi^{\beta(t),w(t)}_n(v)\,\eta_k(x),\\
    S(t,x,v)     &= \sum_{n\in\mathbb{N}}\sum_{k\in\mathbb{Z}} R_{n,k}(t)\,\psi^{\beta(t),w(t)}_n(v)\,\eta_k(x),
  \end{aligned}
\end{equation} 
with non-zero coefficients 
$C_{0,0}(t)= 2$,
$C_{0,2}(t) = -e^{-i 2\pi t}/2$ and
$C_{0,-2}(t) = -e^{i 2\pi t}/2$.
Similarly, algebraic manipulations of the right hand side
\eqref{eq:rhs} using the Hermite recurrence relations yield the
following non-zero expansion coefficients,
\renewcommand{\arraystretch}{2}
\begin{equation*}
  \begin{array}{lll}
    R_{0,0}(t) = 2 \dfrac{d_t\beta}{\beta},
    & \; R_{0,2}(t) = \bigg(i\pi-iw -\dfrac12\dfrac{d_t\beta}{\beta}\bigg)e^{-i 2\pi t} ,
    & \; R_{0,-2}(t) = \bigg(iw-i\pi-\dfrac12\dfrac{d_t\beta}{\beta}\bigg) e^{i 2\pi t},\\
    R_{1,0}(t) = 2\sqrt{2} \dfrac{d_t w}{\beta},
    & \; R_{1,2}(t) = -\dfrac{1}{\sqrt{2}}\left(i\beta + \dfrac{d_t w}{\beta} \right)e^{-i 2\pi t},
    & \; R_{1,-2}(t) = \dfrac{1}{\sqrt{2}}\left(i\beta-\dfrac{d_t w}{\beta}\right) e^{i2\pi t},\\
    R_{2,0}(t) = 2\sqrt{2} \dfrac{d_t\beta}{\beta},
    & \; R_{2,2}(t) = -\dfrac{1}{\sqrt{2}}\dfrac{d_t\beta}{\beta}e^{-i 2\pi t},
    & \; R_{2,-2}(t) = -\dfrac{1}{\sqrt{2}}\dfrac{d_t\beta}{\beta}e^{i 2\pi t}.
  \end{array}
\end{equation*}
In other words, $ f^{\ex}$ and $S$ belong to the approximation spaces
$V_{\beta,w}^{N}=\SPAN{\{\psi^{\beta,w}_n(v)\}_{n\in\Lambda_{N_v}}}\times W^N$,
where $N=(N_v,N_x)$, with $N_v=0$, $N_x=2$ for $f^{\ex}$ and $N_v=2$, $N_x=2$ for $S$.

Let us consider the spectral Galerkin discretization of
\eqref{eq:VlasovPoissonSimpl} in $ V^N_{\alpha,u}$, where we take
$N_v\geq2$ and $N_x\geq2$.
In particular, by taking $N_x\geq 2$, we ensure that the numerical
discretization is not affected by any spatial error
but only by an approximation error in velocity and time.
Since $S\in V^N_{\beta,w}$, using its spectral representation
\eqref{eq:expansion}, results in
\begin{equation*}
  \begin{aligned}
    (S,(\pi 2^n n!)^{-1/2}H_n(\xi)\eta_{-k}(x))_{L^2(\Omega)} 
    & =
    \sum_{m\in\Lambda_{N_v}} R_{m,k}(t) \int_{\Omega_v}\psi^{\beta,w}_m(v)\psi^{\alpha,u}_n(v)
    e^{\left(\frac{v-u}{\alpha}\right)^2}\dfrac{\sqrt{\pi}}{\alpha}\,\der v\\
    & = \sum_{m\in\Lambda_{N_v}} \mathbb{P}_{n,m}R_{m,k}(t),
  \end{aligned}
\end{equation*}
where $\mathbb{P}\in\mathbb{R}^{N_v\times N_v}$ is the matrix
\eqref{eq:transfM} associated with the operator in
Definition~\ref{def:Pj} from $V^N_{\beta,w}$ onto $V^N_{\alpha,u}$.
Since $R_{m,k}\equiv 0$ for all $m\in\Lambda_{N_v}\setminus\{0,1,2\}$,
the fully discrete approximation \eqref{eq:ode} of
\eqref{eq:VlasovPoissonSimpl}, can be recast as,
\begin{equation}\label{eq:MMSode}
  \begin{aligned}
    \dfrac{\widehat{C}_{n,k}^{j+1}-\widehat{C}_{n,k}^j}{\Delta t_j} + 
    & \sqrt{\dfrac{n}{2}}\dfrac{2\pi}{L}\, ik\, \alpha_j\, \widehat{C}_{n-1,k}^{j+1/2}
    + \dfrac{2\pi}{L} i k\, u_j\, \widehat{C}_{n,k}^{j+1/2}
    + \sqrt{\dfrac{n+1}{2}}\dfrac{2\pi}{L}\, ik\,\alpha_j\, \widehat{C}_{n+1,k}^{j+1/2}=\\[.5em]
    & \mathbb{P}_{n,0}R_{0,k}^{j+1/2}+\mathbb{P}_{n,1}R_{1,k}^{j+1/2}+\mathbb{P}_{n,2}R_{2,k}^{j+1/2},
    \qquad \forall\,(n,k)\in\Lambda_{N_v}\times\Lambda_{N_x},\, j\in\Lambda_{N_t},
  \end{aligned}
\end{equation}
where the right hand side vanishes for any
$k\in\Lambda_{N_x}\setminus\{-2,0,2\}$.
If not otherwise specified, the spectral numerical discretization is
performed without the artificial collisional operator ($\nu=0$).

\subsubsection{Test case 1: the exact and numerical solutions belong to
  the same approximation space; $\beta(t)=1$, $w(t)=0$, for any $t\in T$, $\alpha=1$,
  $u=0$.}
As a first test case, to benchmark the code, we consider as exact
solution the function in \eqref{eq:fex} with $\beta(t)=1$ and $w(t)=0$ for
all $t\in T$. We solve the fully-discrete problem \eqref{eq:MMSode}
with the non-adaptive algorithm.
The spectral discretization \eqref{eq:MMSode} is performed in the
finite-dimensional space $V^N_{\alpha,u}$ with $\alpha=1(=\beta)$ and
$u=0(=w)$.
The number of Fourier modes is $N_x=4$ and the time step is set to
$\Delta t=10^{-2}$.
The plot on the left of Figure~\ref{fig:MMSb1w0a1u0} shows the
evolution of the $L^2$-error for different number of Hermite modes.
One can observe that the approximation error is dominated by the
temporal error, and it is independent of the numbers of Hermite modes (see also
the plot on the right of Figure~\ref{fig:MMSb1w0a1u0}). This is expected
since $\alpha=\beta$, $u=w$ and the exact solution $ f^{\ex}$ belongs to
the finite dimensional space $V^{(0,2)}_{\beta,w}$.
On the right of Figure~\ref{fig:MMSb1w0a1u0}, the $L^2$-error obtained
with $N_v=16$ Hermite modes and $N_x=4$ Fourier modes is computed for
different time steps.
As expected, the convergence rate in time is $2$.

\begin{figure}[H]
  \centering
  \includegraphics[width=0.475\textwidth]{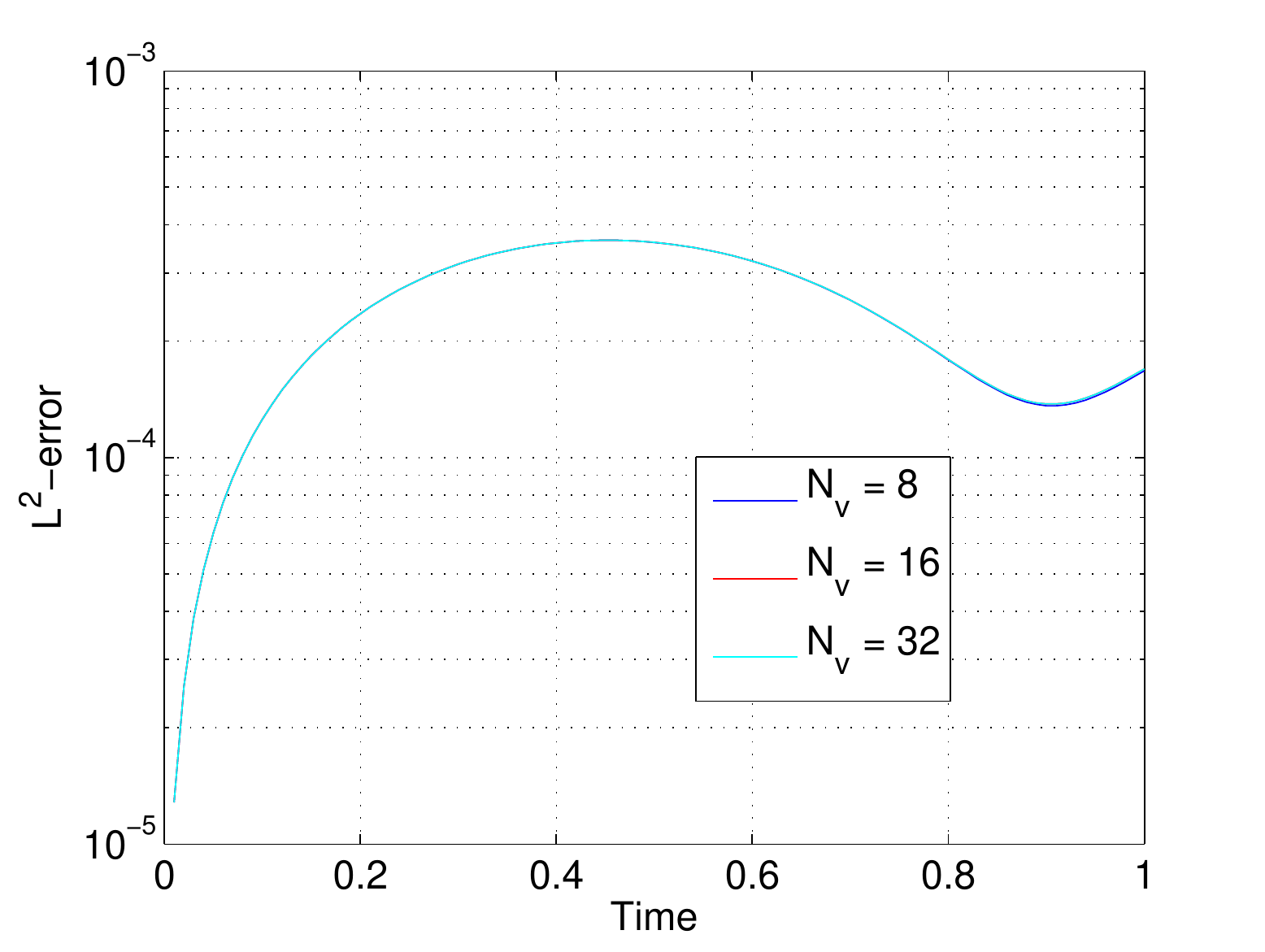}
  \includegraphics[width=0.475\textwidth]{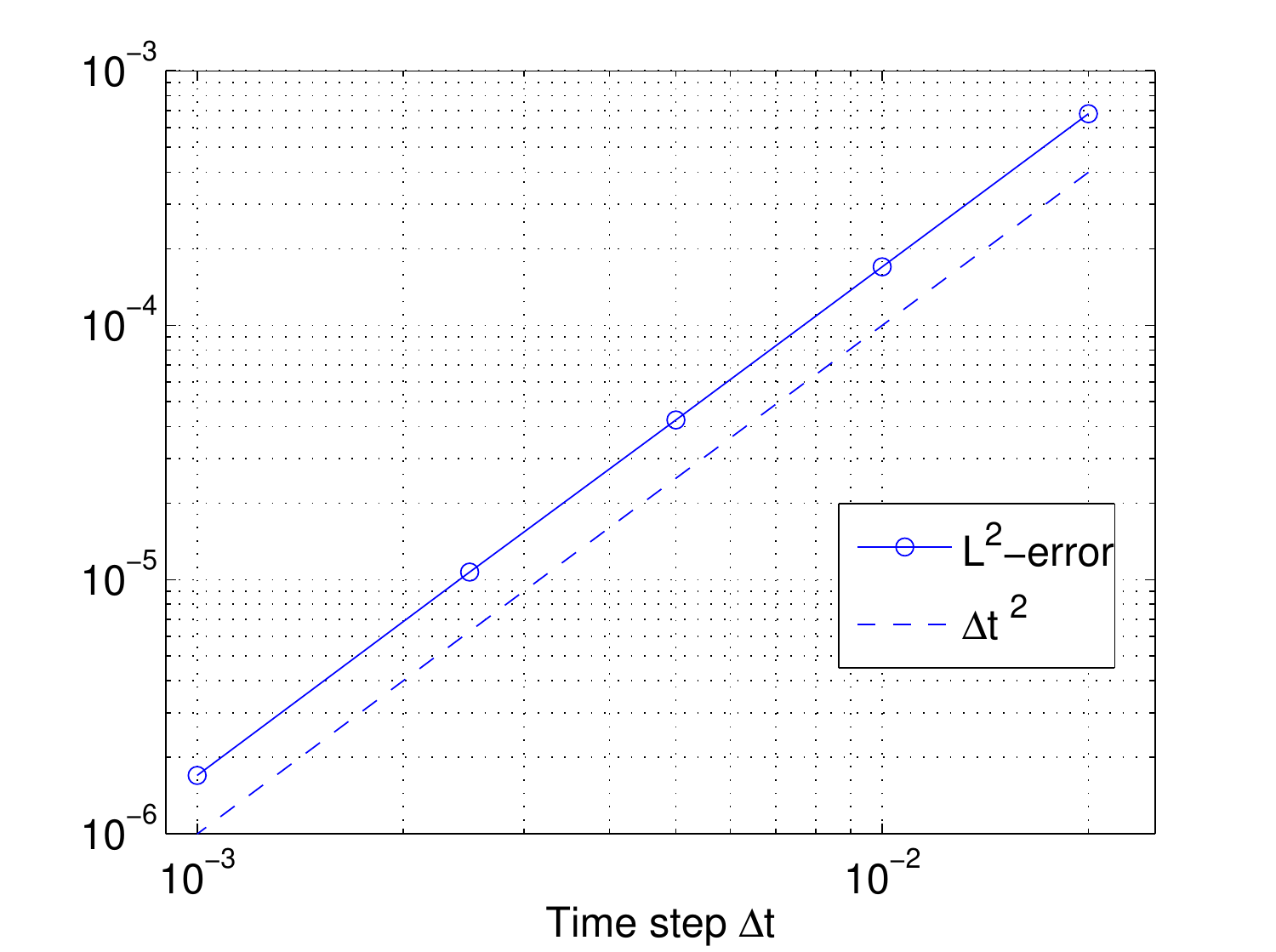}
  \caption{Manufactured solution. Evolution of the $L^2$-error for different numbers of Hermite modes (left)
    with $N_x=4$ Fourier modes and $\Delta t=10^{-2}$.
    $L^2$-error vs. time step (right) for $N=(16,4)$.}
  \label{fig:MMSb1w0a1u0}
\end{figure}

\subsubsection{Test case 2: the exact and numerical solutions belong to
  different approximation spaces; $\beta(t)=1+t$, $w(t)=0$, for all $t\in T$,
  $\alpha=1$, $u=0$.}
As a second test case we take as exact solution the function
$ f^{\ex}$ from~\eqref{eq:fex} with $\beta(t)=1+t$ and $w(t)\equiv 0$.
We consider again the non-adaptive algorithm.
In order to show that failing to select accurate Hermite parameters,
and hence approximation spaces, can have detrimental effects on the
numerical solution (cf. Sec. 3), we use a spectral discretization in the
finite-dimensional space $V^N_{\alpha,u}$ with constant $\alpha=1$ and
$u=0$.

Figure~\ref{fig:MMSbtw0a1u0} {(left)} shows the $L^2$-error obtained
for $\Delta t=10^{-2}$ with $N_x=4$ Fourier modes and various $N_v$,
while the evolution of the exact and ``approximate'' Hermite scaling
parameters is presented in Figure~\ref{fig:MMSbtw0a1u0} {(right)}.
The discrepancy between $\alpha$ and the exact value of $\beta$
increases linearly over time.
In light of this behavior, we can identify three regions in the error
evolution {(left plot)}: 
$(1)$ at the beginning of the simulation, until around time $t=0.05$,
the accuracy of the simulation is only affected by the temporal error
dominating over the error in velocity and all curves for different
values of $N_v$ overlap perfectly;
$(2)$ then, until around time $t=0.4$, the error in velocity dominates
and increasing the number of Hermite modes is effective in improving
the accuracy of the simulation;
$(3)$ for $t>0.45$ the accumulation of error and the lack of accuracy
introduced by selecting the wrong Hermite parameters yield to an
{unstable behavior.}
One can notice the similarities with the considerations made in Sec. 3.

\begin{figure}[H]
  \centering
  \includegraphics[width=0.48\textwidth]{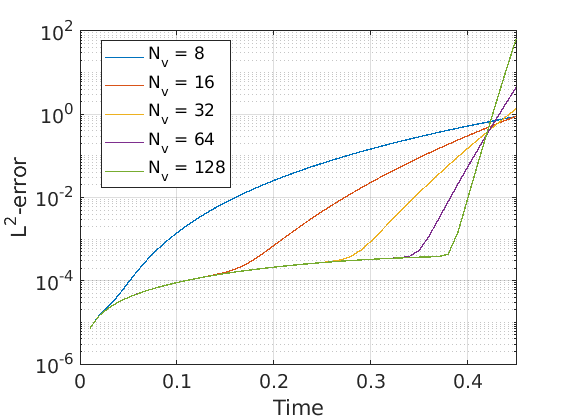}
  \includegraphics[width=0.48\textwidth]{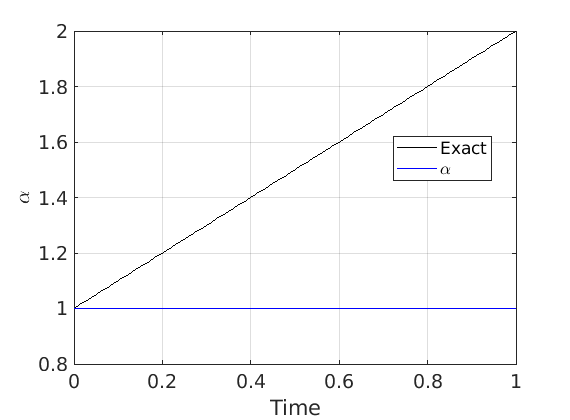}
  \caption{{Manufactured solution. Evolution of the $L^2$-error for
    different numbers of Hermite modes (top left) with $N_x=4$ Fourier
    modes and $\Delta t=10^{-2}$.  Evolution of $\beta$ (exact) and
    $\alpha$ (top right).
    }}
  \label{fig:MMSbtw0a1u0}
\end{figure}

\subsubsection{Test case 3: $\beta(t)=1+t$, $w(t)=0$, for all $t\in T$, while
  $\alpha_j$ and $u_j$ are chosen according to the physics-based
  criterion.}
{We now consider} a spectral discretization in each time interval $T_j$ with
finite-dimensional spaces $V^N_{j}$ where the Hermite parameters are
adapted using Algorithm~\ref{alg:adaptive} and the physics-based criterion of Section~\ref{sec:PhysB}.
The $L^2$-error for different numbers of Hermite modes and time steps
{$\Delta t\in\{10^{-2}, 10^{-3}\}$}
is shown in Figure~\ref{fig:MMSbtw0atu0PhyB} (left).
On the right plot of Figure~\ref{fig:MMSbtw0atu0PhyB}, the evolution of
the Hermite scaling parameter $\alpha$ is presented for different
values of $N_v$.
{One can notice that the solution converges
to the right solution and the $L^2$-error is dominated by the
approximation in the temporal variable, with a minimal dependence on the number of Hermite modes $N_v$. Moreover, the results are practically indistinguishable from the ones obtained by imposing $\alpha_j =\beta(t^j)$ and $u_j=w(t^j)$, which are not reported for brevity. This proves} 
that the physics-based algorithm can track the correct Hermite parameters effectively.

Similar results are obtained with the physics-based adaptive algorithm
when also the shifting parameter is changing in time and are not
reported for brevity.

\begin{figure}[H]
  \centering
  \includegraphics[width=0.475\textwidth]{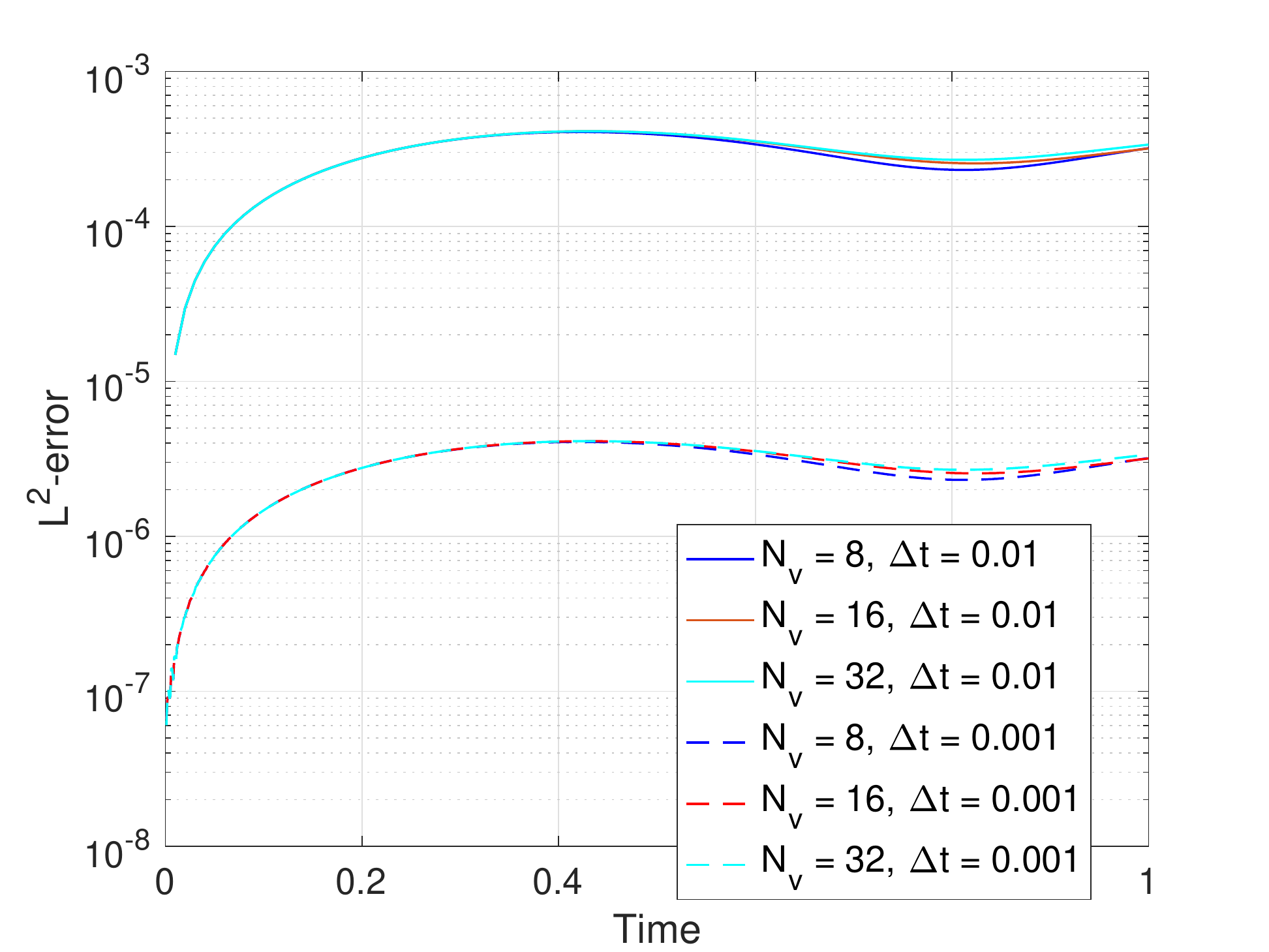}
  \includegraphics[width=0.475\textwidth]{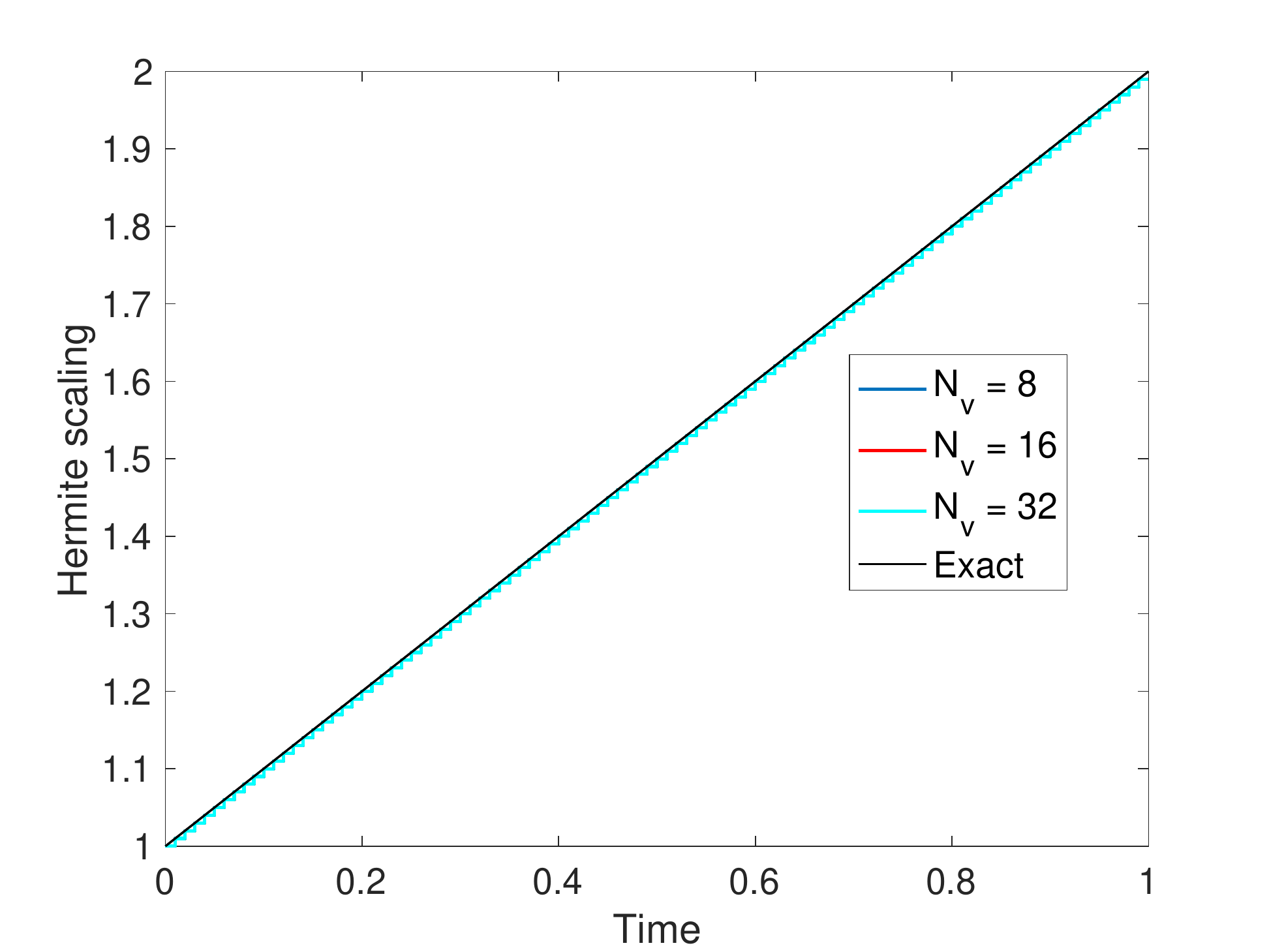}  
  \caption{Manufactured solution. Physics-based adaptivity of the
    Hermite basis functions.  Evolution of the $L^2$-error for
    different numbers of Hermite modes (left).  Evolution of
    $\beta=1+t$ and $\alpha$ for $\Delta t=10^{-2}$ (right). The
    Hermite shift is constant over time, $w=0$.}
  \label{fig:MMSbtw0atu0PhyB}
\end{figure}

\subsubsection{Effects of $\alpha_{\mathrm{tol}}$ and $u_{\mathrm{tol}}$ on the adaptive physics-based algorithm}
In this test case we take as exact solution the function
$ f^{\ex}$ from~\eqref{eq:fex} with $\beta(t)=1.2+\tanh(t-5)$ and $w(t)\equiv 0$ in the temporal interval $(0,10]$. We consider a spectral discretization in each time interval $T_j$ with finite-dimensional spaces $V^N_{j}$, $N=(4,2)$, where the Hermite parameters are adapted using Algorithm~\ref{alg:adaptive} and the physics-based criterion of Section~\ref{sec:PhysB}. In particular we study the effect of the tolerance $\alpha_{\mathrm{tol}}$ in \eqref{eq:AdaptTol} on the performances of the adaptive strategy.

In Figure~\ref{fig:MMStol} we report the $L^2$-error for different values of the time step $\Delta t\in\{10^{-1},10^{-2},10^{-3}\}$ and for different tolerances $\alpha_{\mathrm{tol}}\in\{10^{-1},10^{-2},10^{-3}\}$.
We observe that for $N_v=8$, plotted on the left, the error improves as $\alpha_{\mathrm{tol}}$ decreases except for sufficiently large time steps where the error introduced by the temporal discretization is dominant and no effect of the tolerance $\alpha_{\mathrm{tol}}$ is observed. Moreover, as the number of Hermite modes grows, the effect of $\alpha_{\mathrm{tol}}$ is negligible, as it can be seen on the right plot of Figure~\ref{fig:MMStol}. This is expected, since with more Hermite modes available rescaling the Hermite basis becomes less important, and indeed the plot shows that the error is dominated by the temporal discretization error in all cases considered. Overall, $\alpha_{\mathrm{tol}}\sim 10^{-2}$ is sufficient to ensure good accuracy in velocity space for this example.
In Figure~\ref{fig:MMStolalpha}, we monitor the evolution of $\alpha$: For tolerances of the order of $10^{-2}$ or lower, the Hermite scaling computed with the physics-based criterion is able to reproduce very accurately the exact behavior.

\begin{figure}[H]
  \hspace{-1.1em}\includegraphics[scale=0.4]{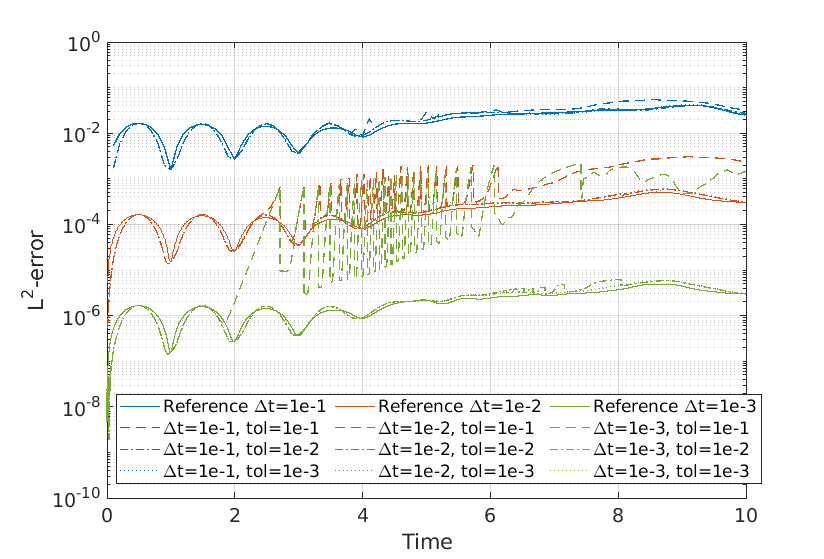}\hspace{-2.05em}
  \includegraphics[scale=0.4]{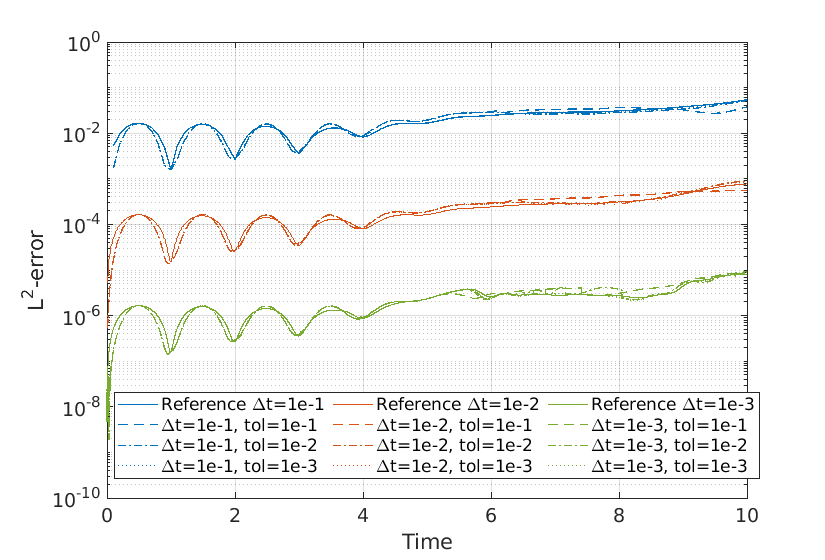}
  \caption{Manufactured solution. Physics-based adaptivity of the
    Hermite basis functions with different $\alpha_{\mathrm{tol}}$ and different time step $\Delta t$. Evolution of the $L^2$-error for $N_v=8$ (left) and $N_v=32$ (right). The
    Hermite shift is constant over time, $w=0$.}
  \label{fig:MMStol}
\end{figure}

\begin{figure}[H]
  \includegraphics[width=0.325\textwidth]{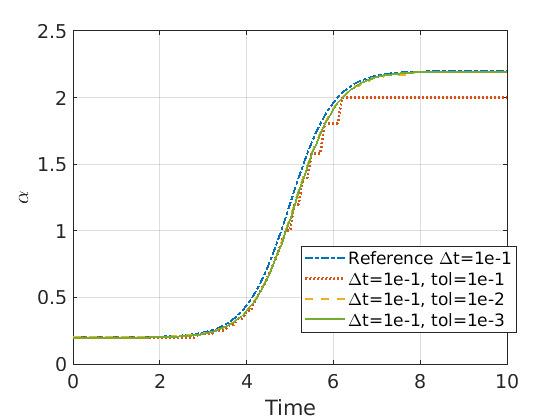}
  \includegraphics[width=0.325\textwidth]{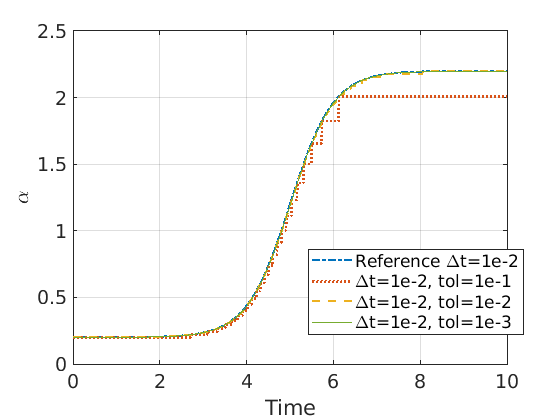}
  \includegraphics[width=0.325\textwidth]{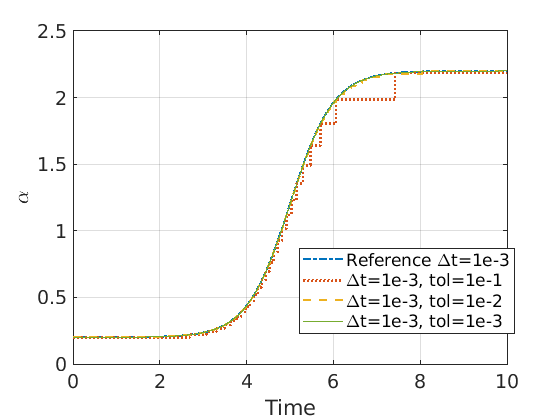}
  \caption{Manufactured solution. Physics-based adaptivity of the
    Hermite basis functions with different $\alpha_{\mathrm{tol}}$. and different time step $\Delta t$. Evolution of $\alpha$ for $\Delta t=10^{-1}$ (left), $\Delta t=10^{-2}$ (center), and $\Delta t=10^{-3}$ (right). The number of Hermite modes is $N_v=8$.}
  \label{fig:MMStolalpha}
\end{figure}

In summary, the results shown in this subsection demonstrate the
importance of choosing the correct approximation space.
In the examples considered, while being too far off from the correct approximation space triggered some numerical instability, the
physics-based algorithm is capable of tracking the correct evolution
of the Hermite parameters and maintain an accurate numerical solution.

\subsection{Two-stream instability}\label{sec:TSI}
We now consider the two-stream instability, a classical physics benchmark for kinetic plasma codes.
It consists of a linear instability excited by two populations of
particles counter-streaming with a sufficiently large relative speed.

The initial configuration consists of two drifting Maxwellian electron populations with equal temperature and a Maxwellian at rest neutralizing ion population
\begin{equation}
  \begin{aligned}
    & f^e_0(x,v)=\dfrac{n_0^e}{\sqrt{\pi}}\dfrac{1}{\alpha^e}\left(1+\dfrac{\epsilon}{2}\cos\left(\frac{2\pi}{L}x\right)\right)e^{-\xi_e^2},
    & \qquad \xi_e:=\dfrac{v-u^e}{\alpha^e}, \\
    & f^i_0(x,v)=\dfrac{1}{\alpha^i\sqrt{\pi}} e^{-\xi_i^2},
    & \qquad \xi_i:=\dfrac{v-u^i}{\alpha^i},   \label{ic}
  \end{aligned}
\end{equation}
with $\epsilon=10^{-3}$, $n_0^e=(\frac{1}{2},\frac{1}{2})$ and
$\alpha^e=(\frac{1}{2},\frac{1}{2})$, $u^e=(1,-1)$ and
$\alpha^i=\sqrt{2\frac{m_e}{m_i}\frac{T_i}{T_e}}$, $u^i=0$.
We set $\frac{m_i}{m_e}=1836$ and $\frac{T_i}{T_e}=1$.
Let us consider the Vlasov--Poisson problem in the computational
domain $\Omega_x:=[0,L]$, with $L=2\pi$, and $\Omega_v:=\mathbb{R}$.
The temporal interval $T:=[0,T_f]$ is divided into $N_t=T_f/\Delta t$
subintervals with uniform $\Delta t=0.05$ and the temporal
discretization is performed using the implicit midpoint rule.
In the adaptive scheme, we take as tolerances \eqref{eq:AdaptTol} for
the Hermite parameters $\alpha_{\mathrm{tol}}=10^{-1}$ and
$u_{\mathrm{tol}}=10^{-2}$.

\subsubsection{Comparison for fixed number of Hermite modes and fixed $\nu$}\label{case0}
As a first numerical test, let us consider the temporal interval
$T=[0,100]$, and the spectral discretization \eqref{eq:VFhL2w} with
$2N_x+1=101$ Fourier modes and $N_v=250$ Hermite modes.
The collisional coefficient of the operator \eqref{eq:CollDC} is set
to $\nu=5$.

As is well known, the instability associated with two counter-propagating beams of particles with equal charge proceeds with the development of particle bunching and trapping and the creation of a vortex in phase space. As a result, the drift velocity of each beam decreases while the width of each beam particle distribution increases. This is evident in Figure~\ref{fig:TwoSfeColl} which shows the evolution of the distribution function for the adaptive (left) and non-adaptive (right) schemes. The evolution of the Hermite parameters is shown in
Figure~\ref{fig:TwoSAeUeColl}, where we can notice a $15\%$ decrease of the beam drift velocity (for each beam). The $\alpha$ parameter, which is associated with the width of the particle distribution, increases by $60\%$.

The adaptive scheme performs better than the non-adaptive one, as
evident when comparing the plots of the distribution function in
Figure~\ref{fig:TwoSfeColl}, where can see that, {at fixed resolution in velocity}, the oscillations of the numerical distribution function are eliminated by the adaptive algorithm.
{More filaments and features of the solution can be captured by increasing the resolution in velocity, as shown in Figure~\ref{fig:TSIconv}.}
Similar considerations can be drawn from Fig.~\ref{fig:TSIfmaxminTime} (left), which shows the time evolution of the minimum and maximum values of the distribution function in the phase space domain. In collisionless plasmas, the maximum and minimum values of the distribution function are conserved and therefore monitoring these quantities provides a measure of accuracy of the overall numerical scheme. One can see that the non-adaptive algorithm has large oscillations of both values while in the adaptive algorithm both quantities are preserved quite well: the maximum of the distribution function changes slightly, with relative error smaller than $5\%$ over the whole time interval, while the minimum becomes slightly negative at the end of the simulation but with error (relative to the peak value of the distribution function) less than $3\%$.

Figure~\ref{fig:TSIfmaxminTime} (right) reports the evolution of the error
in the conserved quantities momentum and energy, evaluated at the approximate solution with respect to the initial condition. Note that conservation of mass is not shown in the figure since its error is exactly zero.
{The conservation laws of momentum and energy} are satisfied up to the tolerance of the nonlinear solver~\cite{GLD15}.
{This numerical experiment verifies the theoretical prediction of Sect.~\ref{sec:conservation:properties}.}

\begin{figure}[H]
  \centering
  \includegraphics[width=0.4\textwidth]{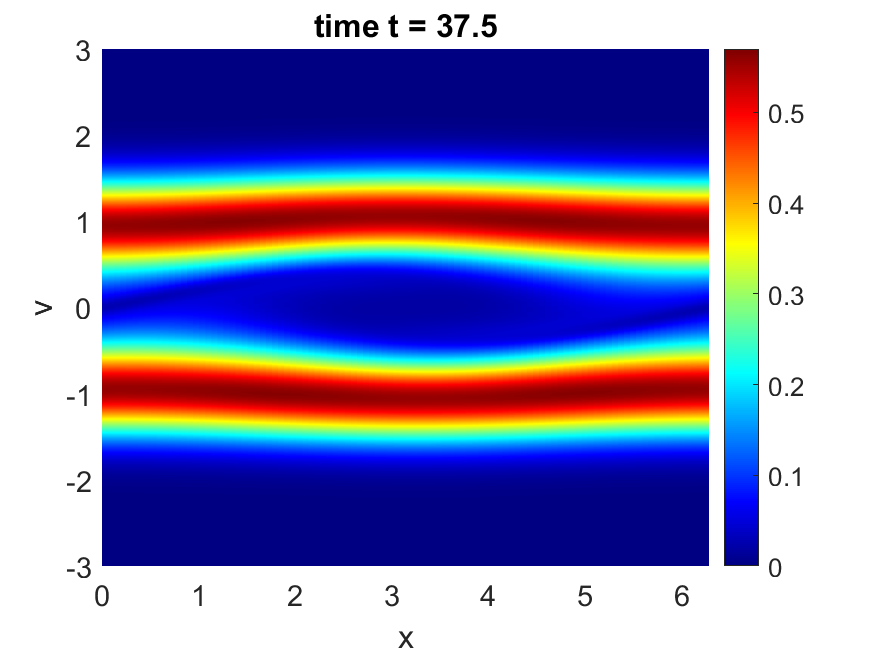}
  \includegraphics[width=0.4\textwidth]{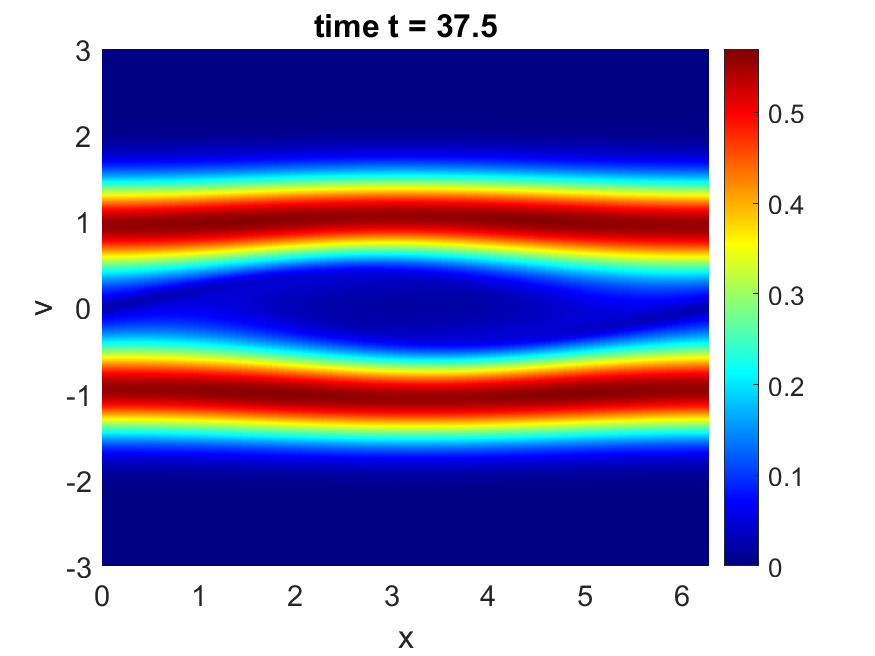}
  \includegraphics[width=0.4\textwidth]{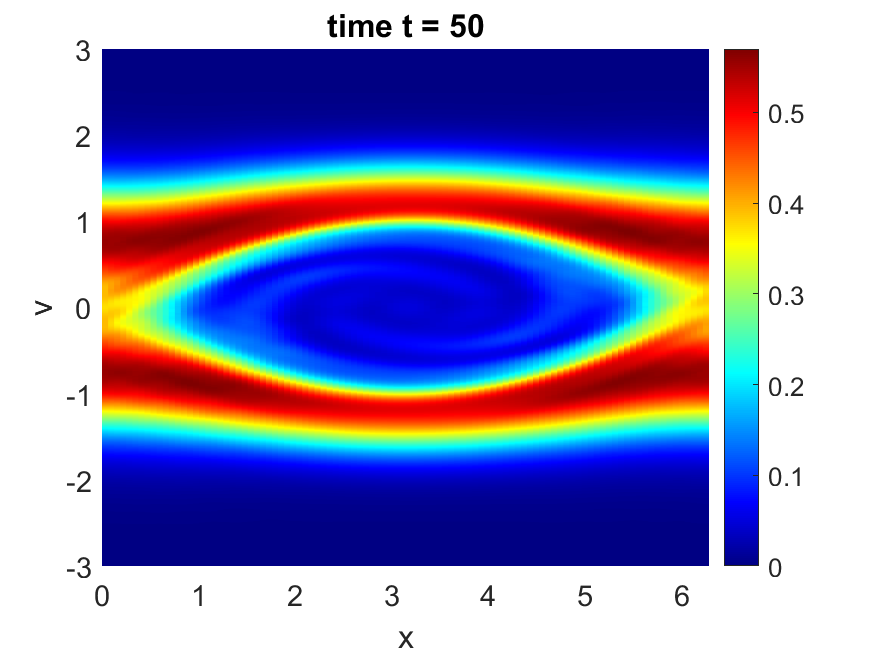}
  \includegraphics[width=0.4\textwidth]{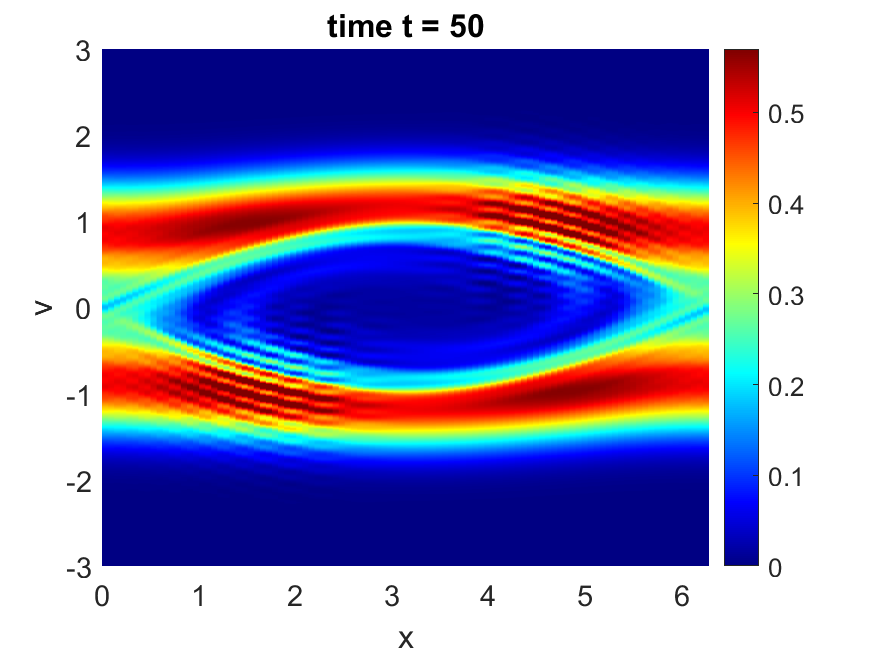}
  \includegraphics[width=0.4\textwidth]{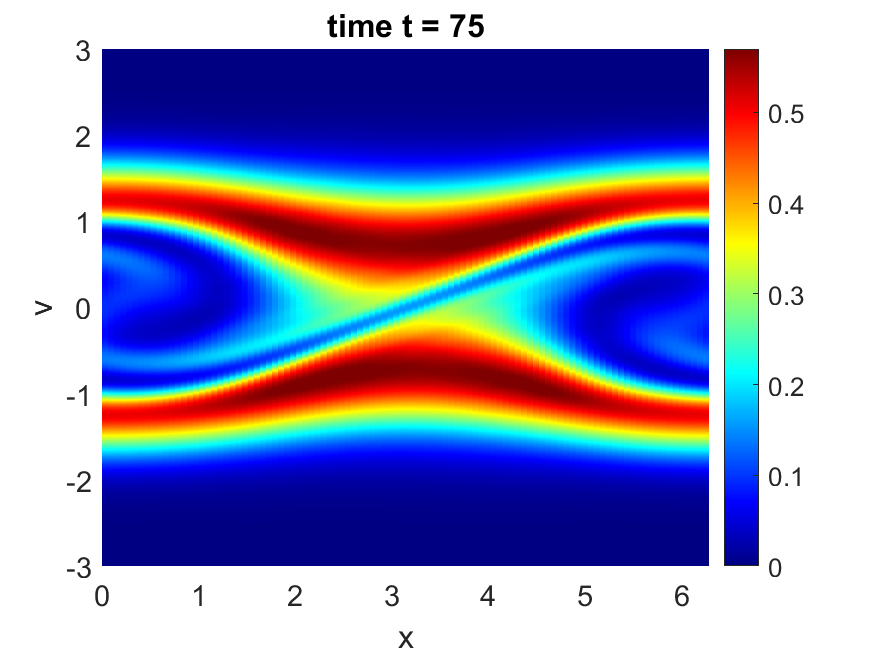}
  \includegraphics[width=0.4\textwidth]{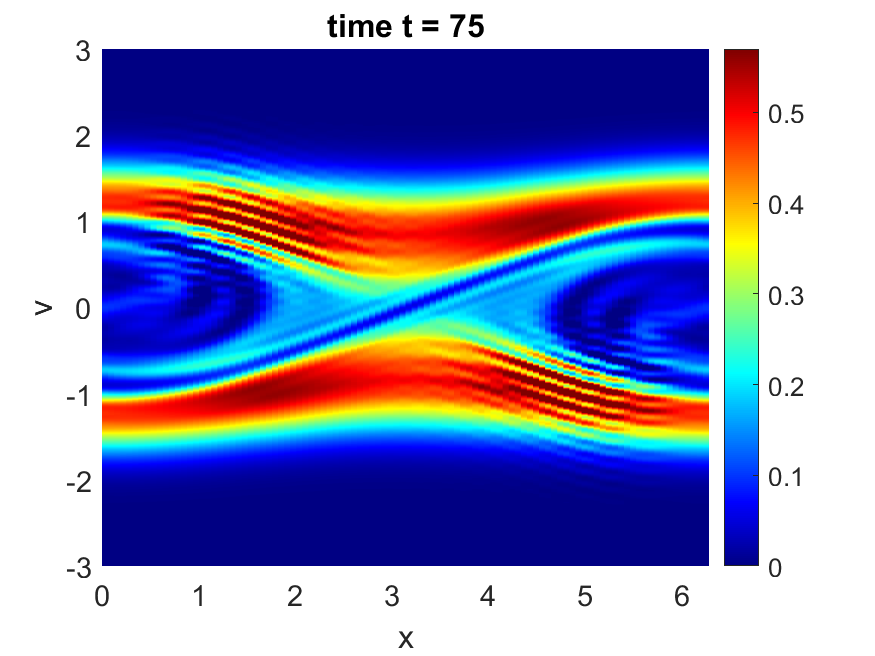}
  \includegraphics[width=0.4\textwidth]{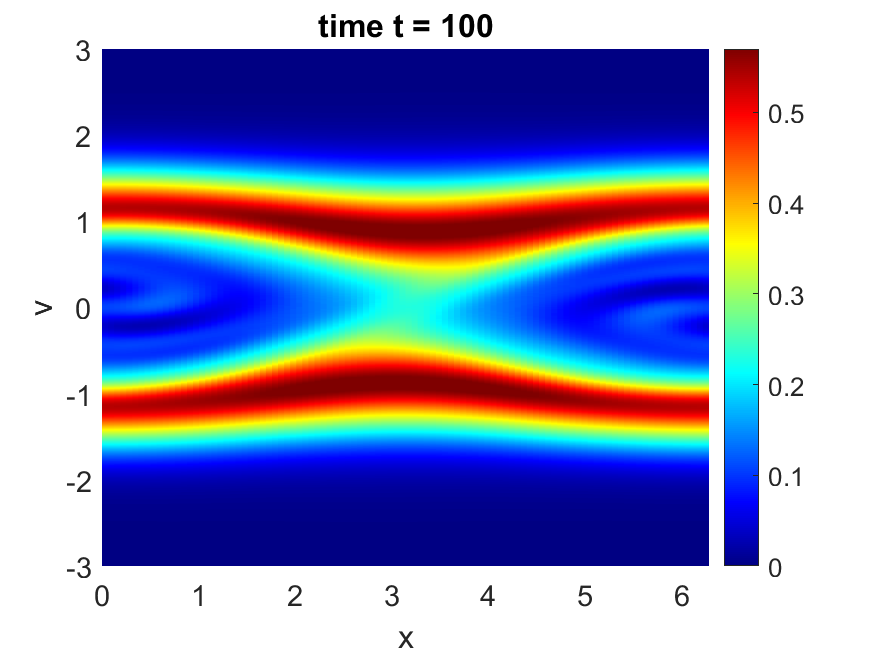}
  \includegraphics[width=0.4\textwidth]{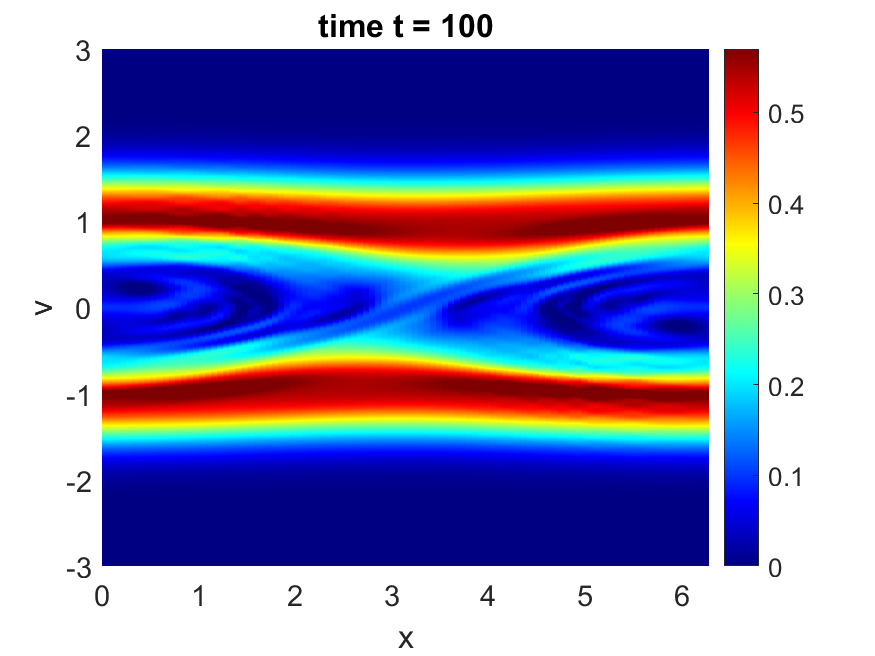}
  \caption{Two-stream instability. Distribution function at different times obtained with the adaptive (left column) and non-adaptive (right column) schemes. 
    The number of Hermite modes is $N_v=250$, the number of Fourier modes is $N_x=50$ and the collisional coefficient is $\nu=5$.
    }
  \label{fig:TwoSfeColl}
\end{figure}

\begin{figure}[H]
  \centering
  \includegraphics[width=0.32\textwidth]{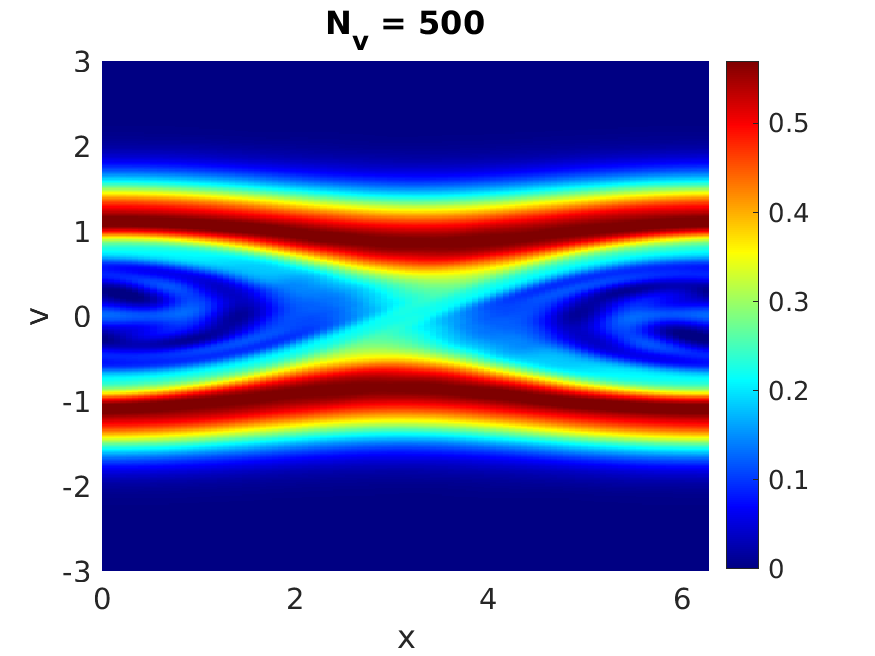}
  \includegraphics[width=0.32\textwidth]{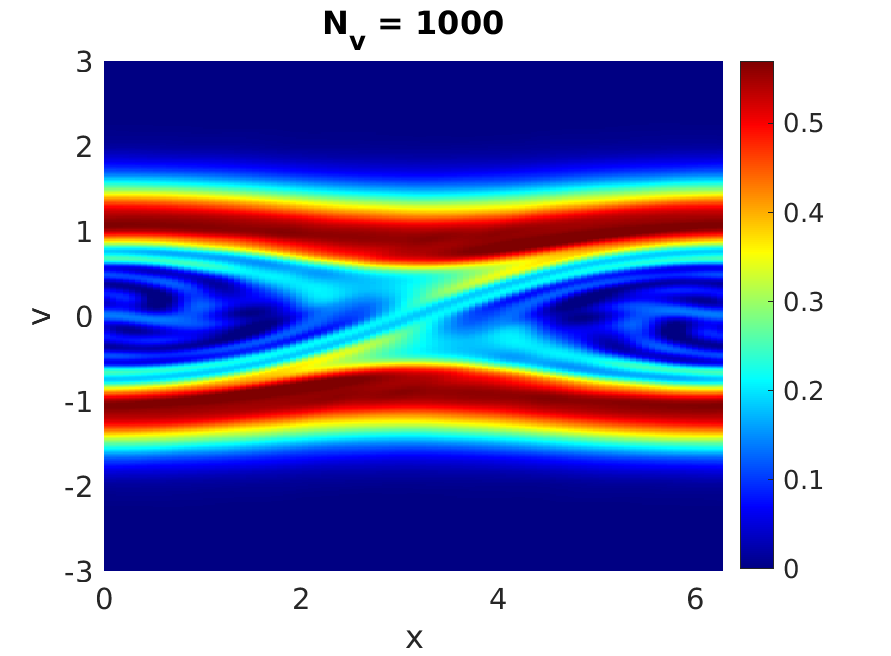}
  \includegraphics[width=0.32\textwidth]{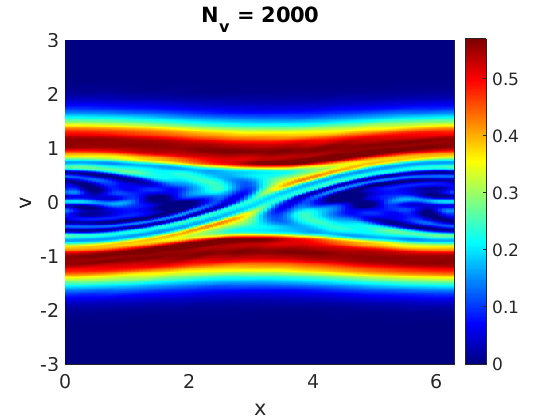}
  \caption{Two-stream instability. Distribution function at time $t=100$ obtained with the adaptive scheme for different number of Hermite modes $N_v=500$ (left), $N_v=1000$ (center), and $N_v=2000$ (right).}
  \label{fig:TSIconv}
\end{figure}

\begin{figure}[H]
  \centering
  \includegraphics[width=0.45\textwidth]{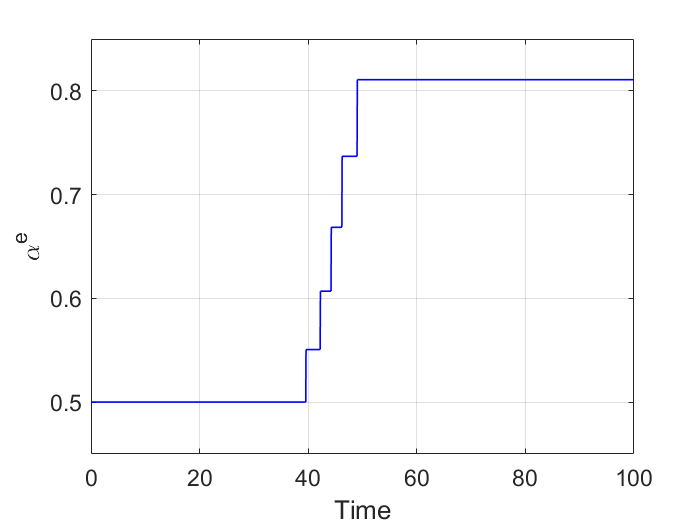}
  \includegraphics[width=0.45\textwidth]{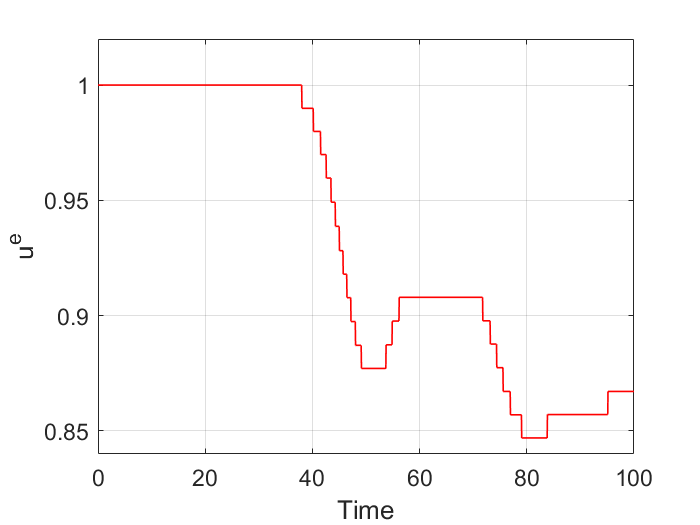}
  \caption{Two-stream instability. Evolution of the physics-based Hermite coefficients. Hermite scaling for electrons $\alpha^e$ (left) and
    Hermite shift for electrons $u^e$ (right).}
  \label{fig:TwoSAeUeColl}
\end{figure}

\begin{figure}[H]
  \centering
  \includegraphics[width=0.45\textwidth]{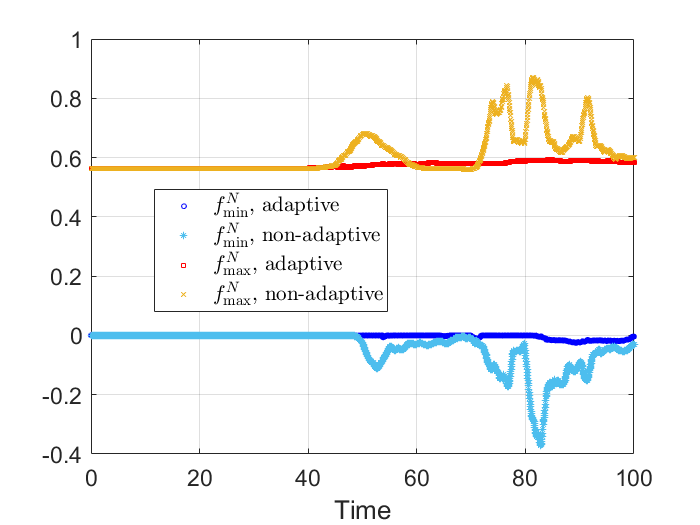}
  \includegraphics[width=0.45\textwidth]{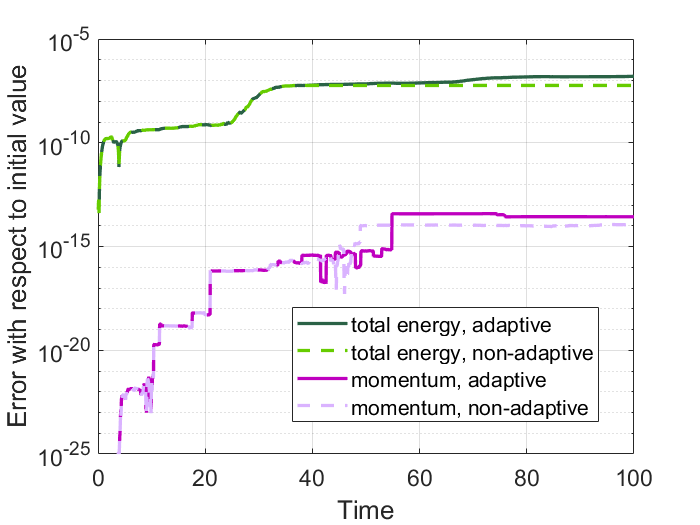}
\caption{\small On the left, evolution of the maximum and minimum values of the distribution function for the electron species in the bounded domain $\Omega=[0,L]\times[-3,3]$ ($400$ points are used to reconstruct the distribution function in the velocity domain).
On the right, numerical verification of the conservation of momentum and energy.
The line corresponding to the mass is not shown since the error is exactly zero at all times.}
\label{fig:TSIfmaxminTime}
\end{figure}

Interestingly, the differences between the two algorithms are not significant when
looking at a macroscopic quantity such as the electric field.
This is likely due to the fact that only one mode is linearly unstable for the parameters considered and this mode, whose linear behavior is captured well by both algorithms, dominates the evolution of the electric field.

\subsubsection{Comparison with different number of Hermite modes and $\nu$ fixed}
Let us now consider the temporal interval
$T=[0,50]$, and the spectral discretization \eqref{eq:VFhL2w} with
$2N_x+1=101$ Fourier modes. We let the number of Hermite modes vary.
The collisional coefficient of the operator \eqref{eq:CollDC} is set
to $\nu=5$.

Figure~\ref{fig:TwoSNvColl} shows the distribution function at the end of the simulations and the time history of the electric field versus the number of Hermite modes, for both adaptive (first two columns) and non-adaptive algorithms (last two columns). One can notice that the adaptive algorithm is not accurate when $N_v=10$, while its solutions are visually indistinguishable for all other values of $N_v$ in the range 50 to 250. Conversely, the non-adaptive algorithm shows oscillations of the distribution function throughout the range of $N_v$ considered. These considerations are confirmed in Table \ref{tab:TSIfmaxminTest1} showing the mininum and maximum values of the distribution function versus number of Hermite modes.  For the parameters considered here, the maximum value of the electron distribution function is 0.56 [Eq. (\ref{ic})].
For the adaptive algorithm, the minimum of the distribution function on the domain considered, although taking negative values, rapidly converges to zero as $N_v$ increases (polynomial rate of convergence of around $5$ for $N_v> 50$).
The maximum value of the distribution function is quite well conserved for all cases except $N_v=10$, with a relative error of only $2\%$ for $N_v=50$ and further decreasing for larger values of $N_v$. The non-adaptive algorithm, on the other hand, shows a strong lack of positivity for the lower values of $N_v$ and a slow decay to zero as $N_v$ increases.
For $N_v=250$, however, the relative error on the maximum values of the distribution function remains high, $\sim 21\%$. Generally speaking, the trends of the maximum value of the distribution function in Table \ref{tab:TSIfmaxminTest1} show a slow convergence of the non-adaptive algorithm with the number of Hermite modes $N_v$.
Finally, similar to the results shown in Sect. \ref{case0}, the behavior of the electric field does not seem to be much affected by the oscillations of the distribution function and the two algorithms produce very similar results. 

We also observe that the Hermite scaling $\alpha^e$ evolves in basically the same way for any value of the number $N_v$ of spectral modes. This suggests that the behavior of the Hermite scaling parameter is dominated by the macroscopic behavior of the
system, at least in this numerical test.

\begingroup
\tabcolsep = 10pt
\def\arraystretch{1.1}
\begin{table}[H]
  \centering
  \begin{tabular}{c | c c | c c}
    & \multicolumn{2}{c|}{$f^N_{\min}(t=50)$} & \multicolumn{2}{c}{$f^N_{\max}(t=50)$}  \\
    &\; non-adaptive\; &\; adaptive\; & \;non-adaptive\; & adaptive\\
    \hline
    $N_v=10$ \;& $-0.8767$ & $-0.0056$ & $0.8837$ & $0.4709$\\
    $N_v=50$ \;& $-0.4273$ & $-0.0776$ & $0.6742$ & $0.5836$ \\
    $N_v=100$ \;&$-0.1383$  & $-0.0275$  & $0.7535$ & $0.5799$\\
    $N_v=125$ \;& $-0.0780$ & $-0.0072$ & $0.7246$ & $0.5805$ \\
    $N_v=250$ \;& $-0.0227$ & $-3.28\cdot 10^{-4}$ & $0.6982$ & $0.5733$ \\
  \end{tabular}
  \caption{\small Maximum and minimum values of the distribution function for the electron species in the bounded domain $\Omega=[0,L]\times[-3,3]$ for different numbers of Hermite modes.}
  \label{tab:TSIfmaxminTest1}
\end{table}
\endgroup

\begin{figure}[H]
  \centering
  \includegraphics[width=0.24\textwidth]{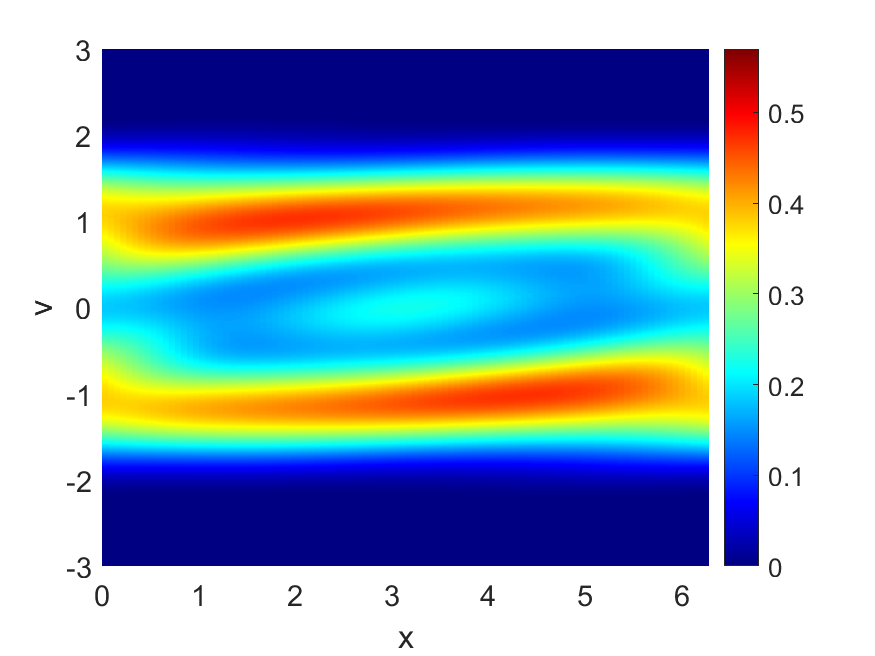}
  \includegraphics[width=0.24\textwidth]{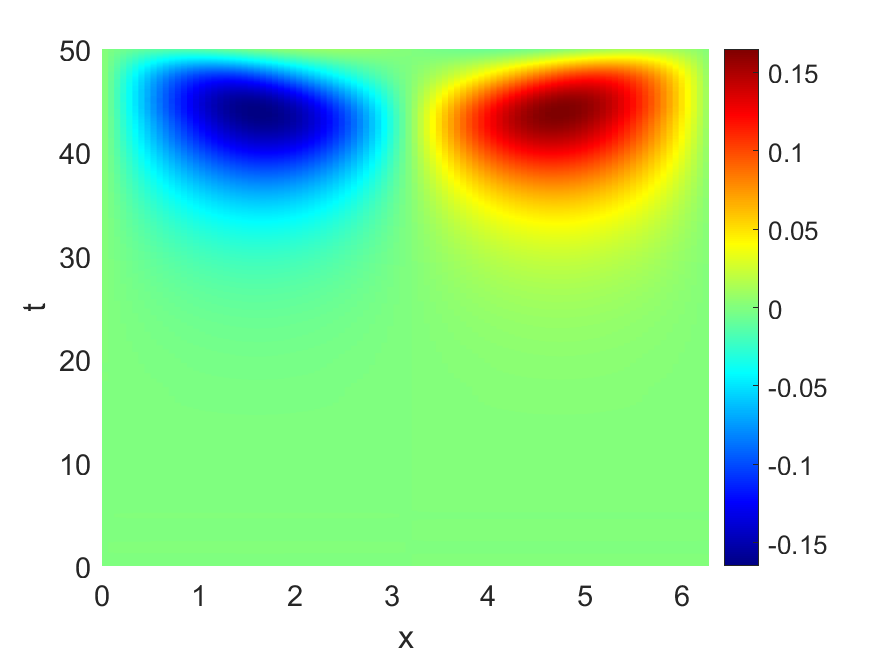}
  \includegraphics[width=0.24\textwidth]{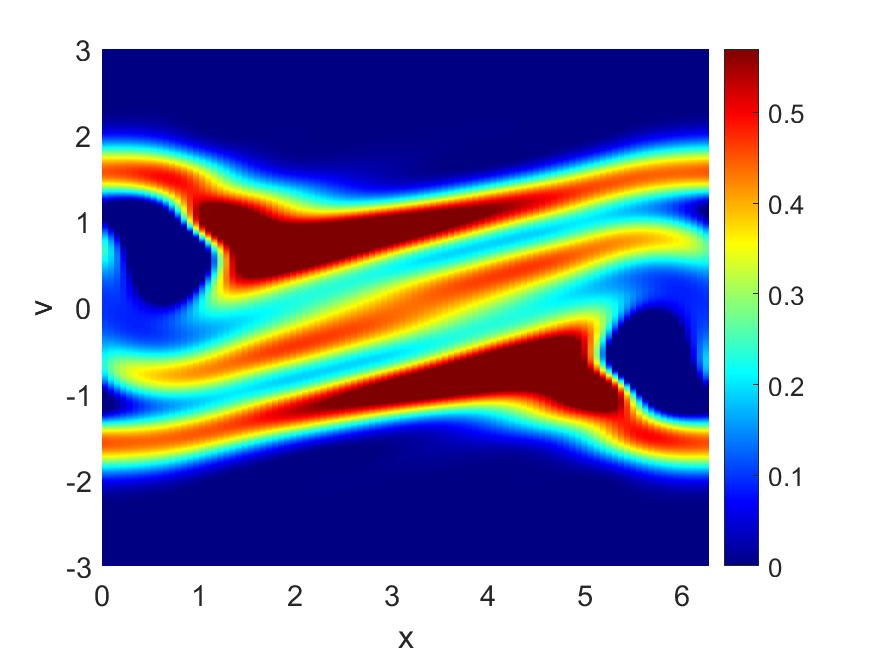}
  \includegraphics[width=0.24\textwidth]{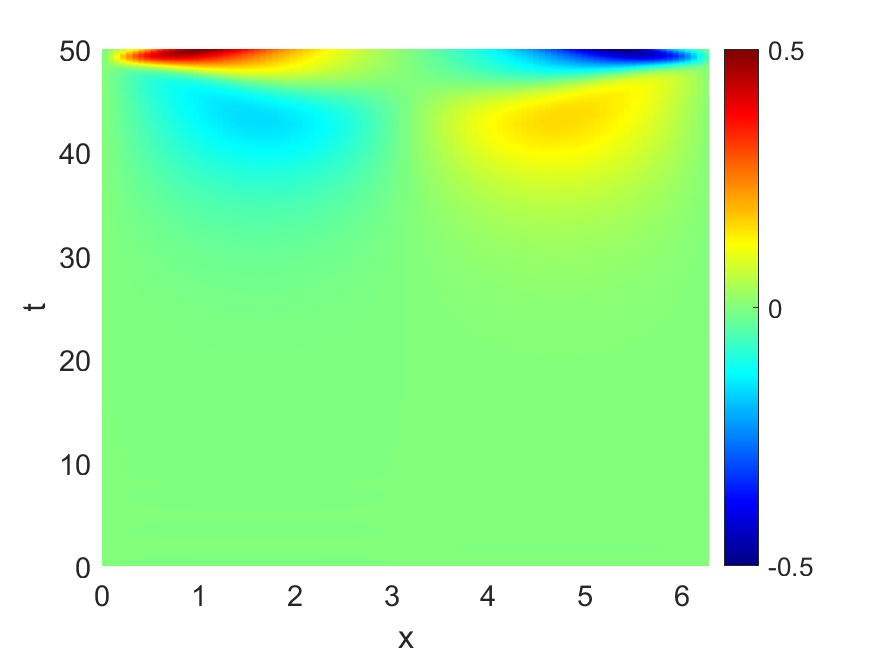}
  \includegraphics[width=0.24\textwidth]{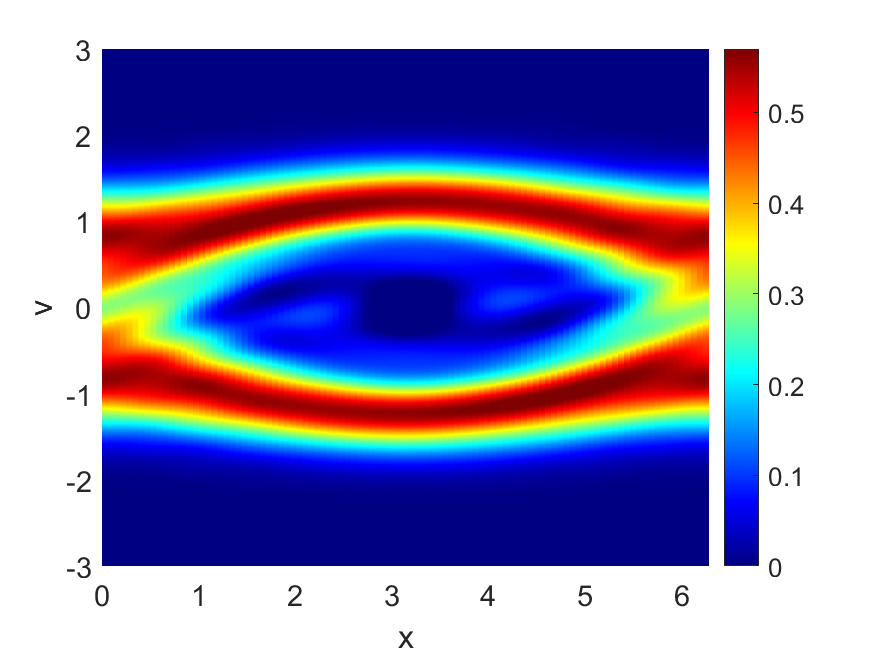}
  \includegraphics[width=0.24\textwidth]{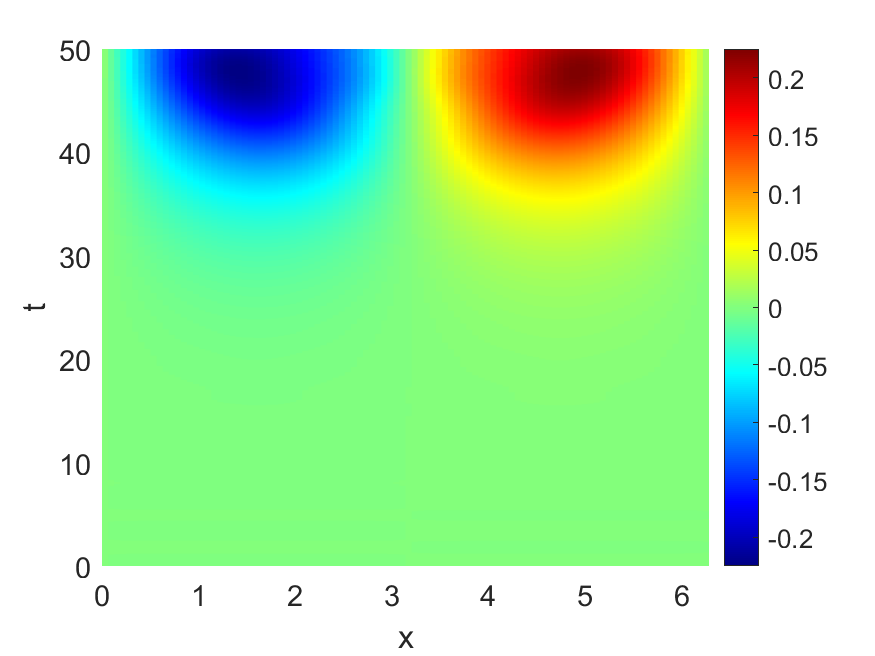}
  \includegraphics[width=0.24\textwidth]{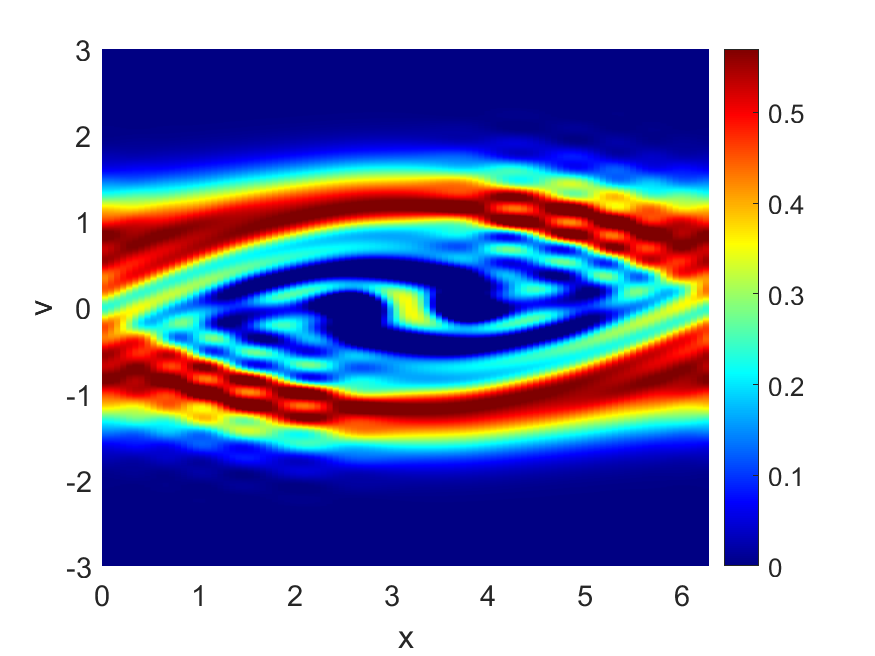}
  \includegraphics[width=0.24\textwidth]{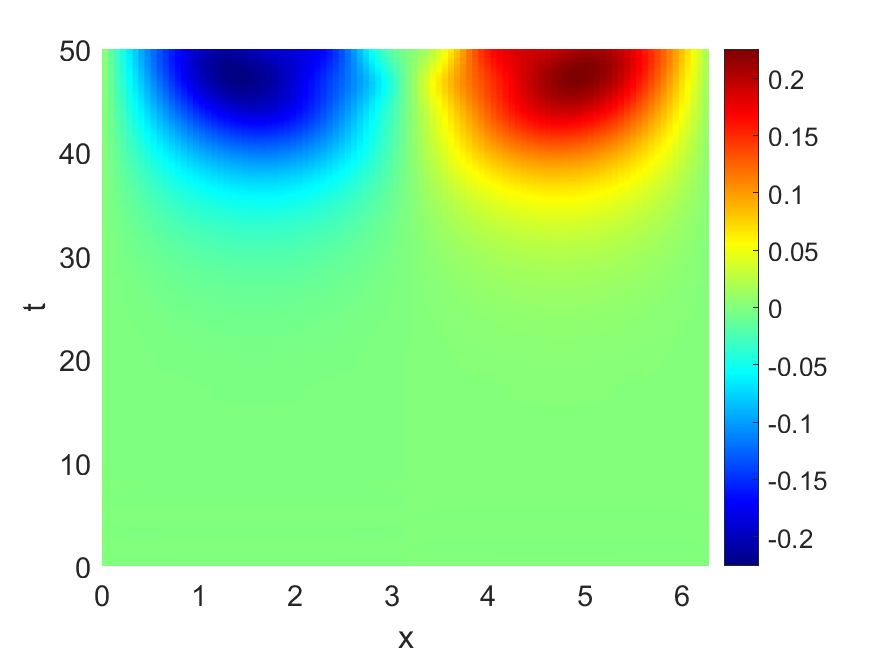}
  \includegraphics[width=0.24\textwidth]{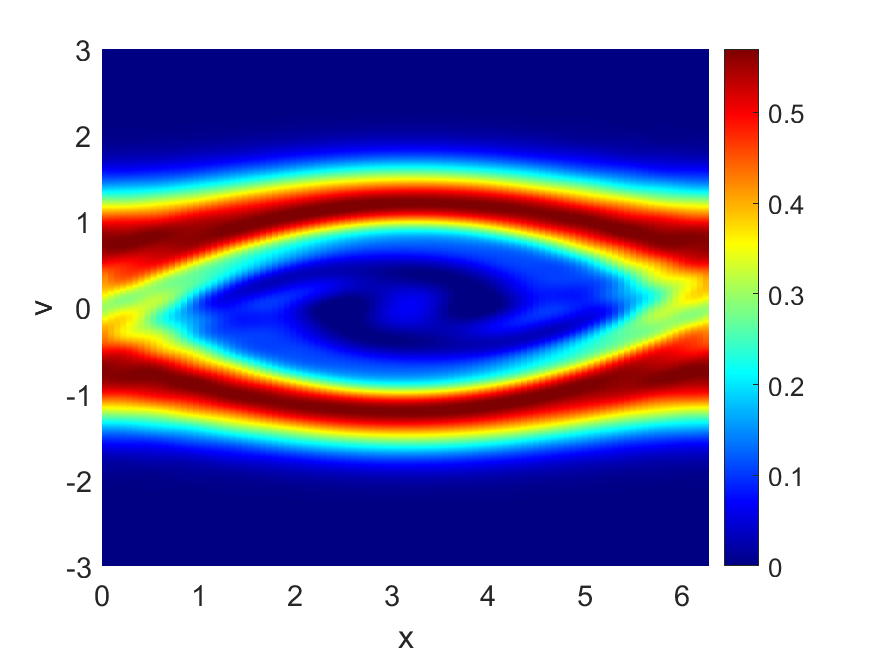}
  \includegraphics[width=0.24\textwidth]{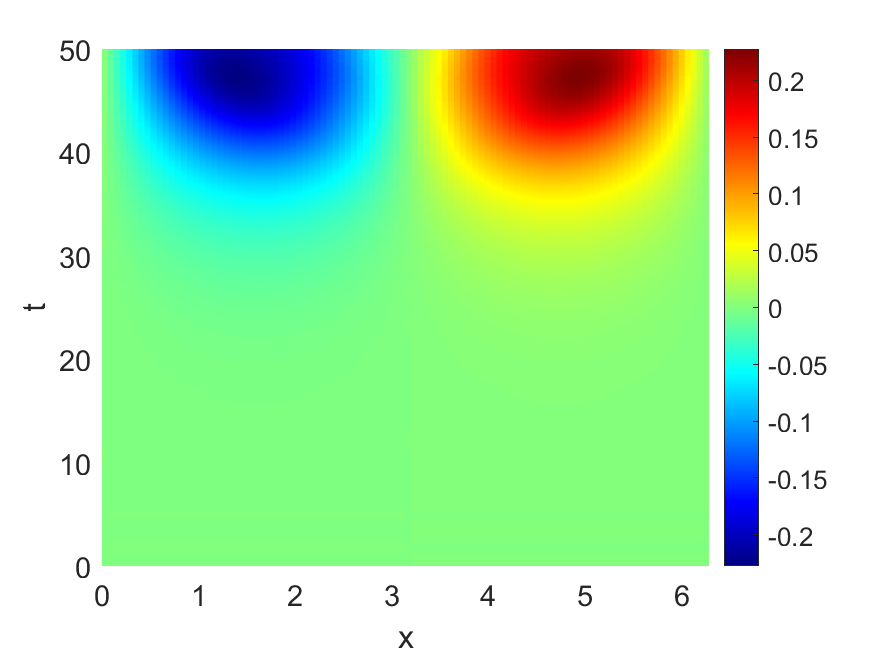}
  \includegraphics[width=0.24\textwidth]{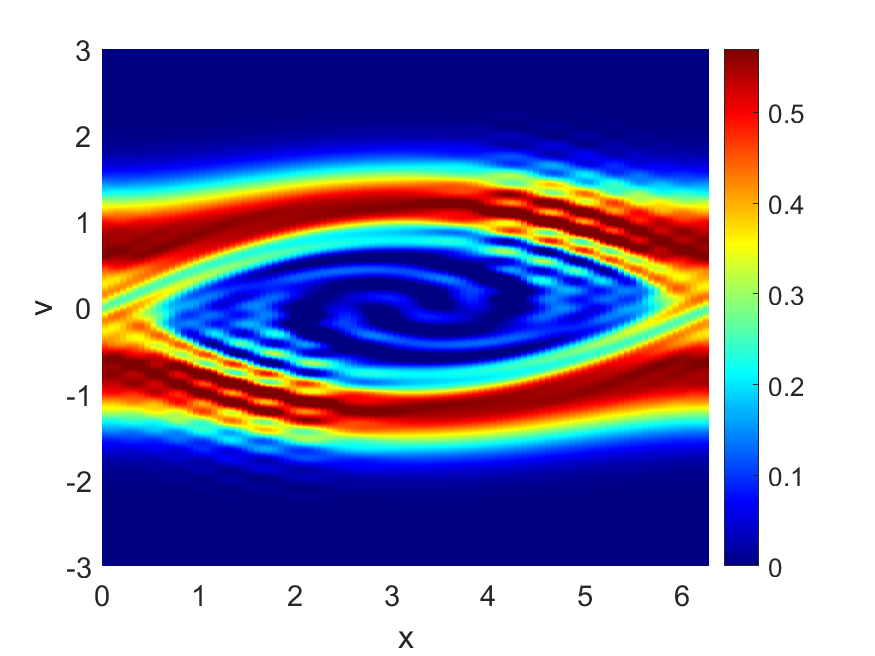}
  \includegraphics[width=0.24\textwidth]{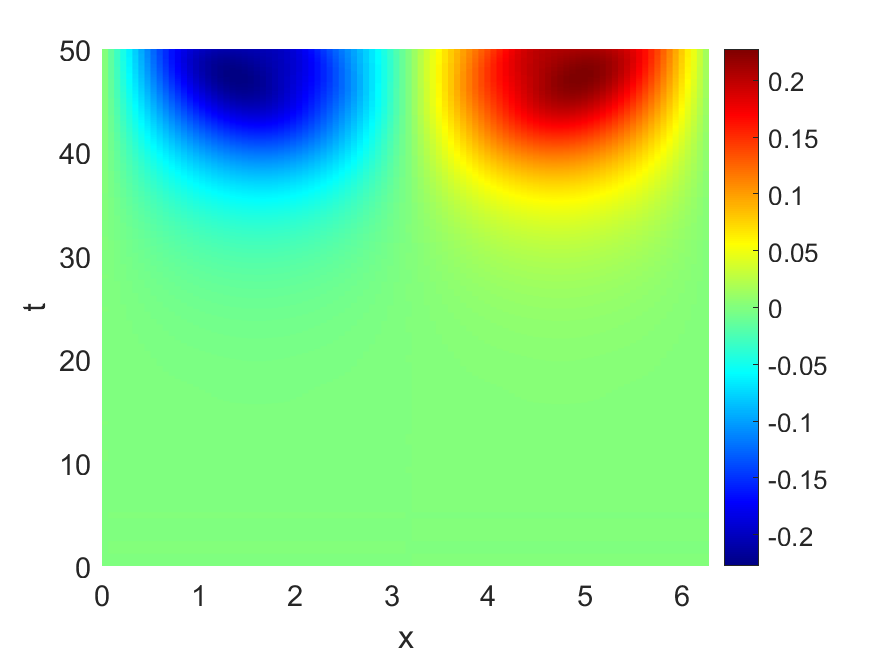}
  \includegraphics[width=0.24\textwidth]{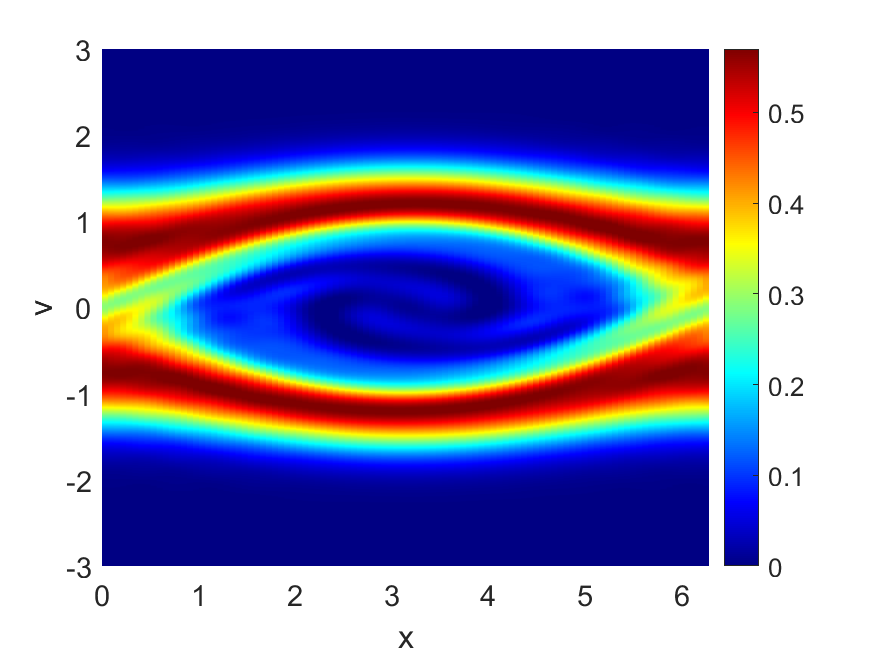}
  \includegraphics[width=0.24\textwidth]{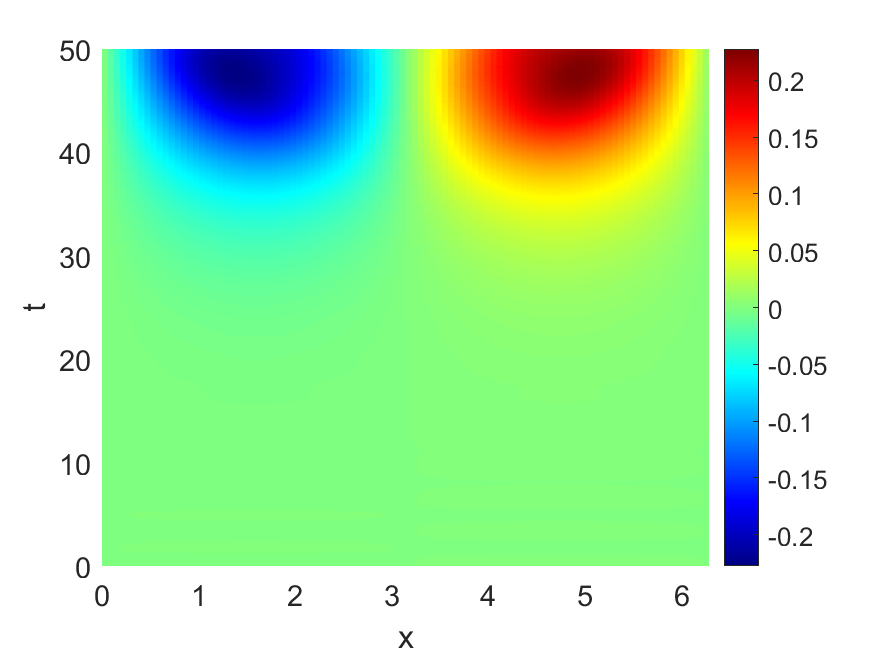}
  \includegraphics[width=0.24\textwidth]{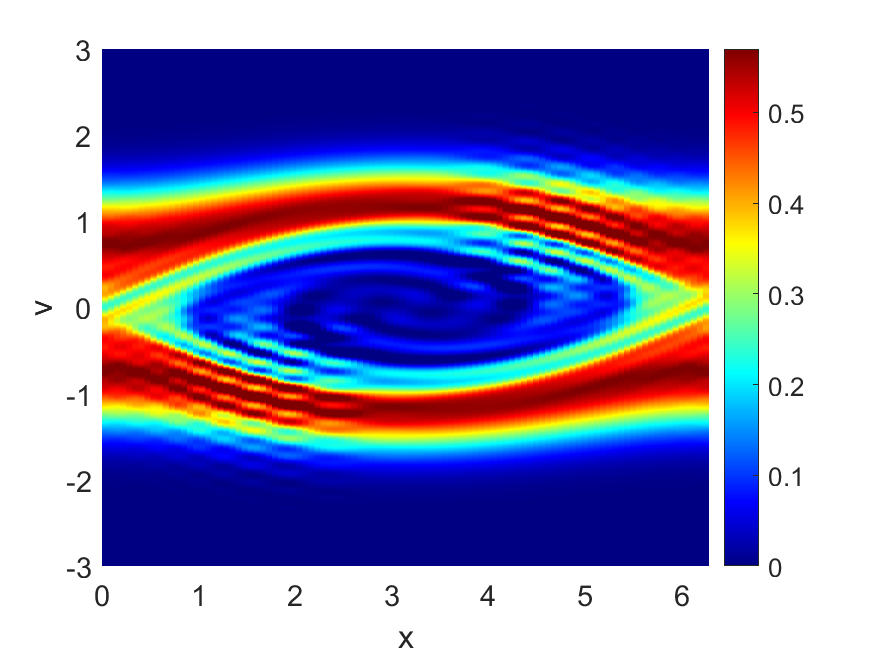}
  \includegraphics[width=0.24\textwidth]{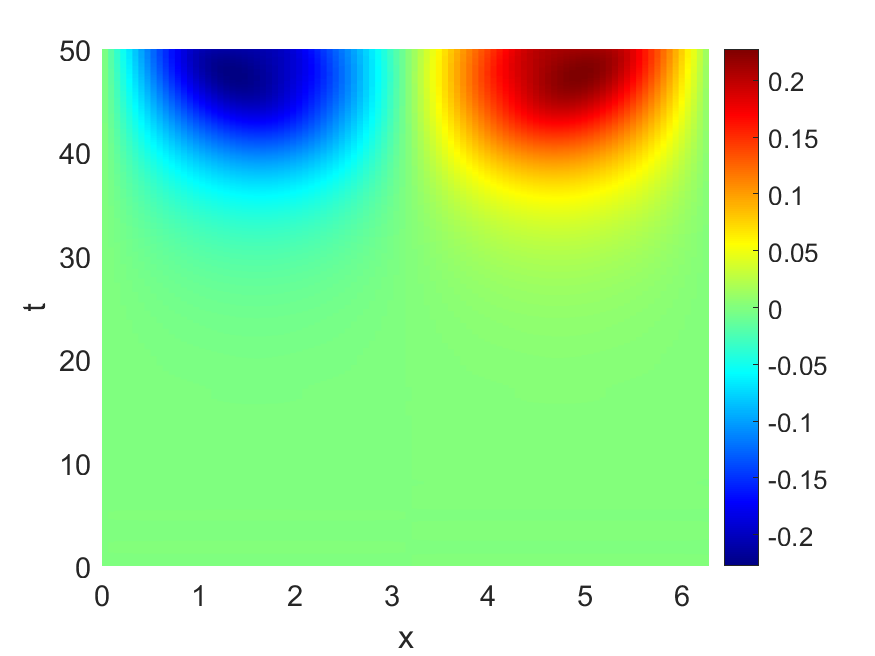}
  \includegraphics[width=0.24\textwidth]{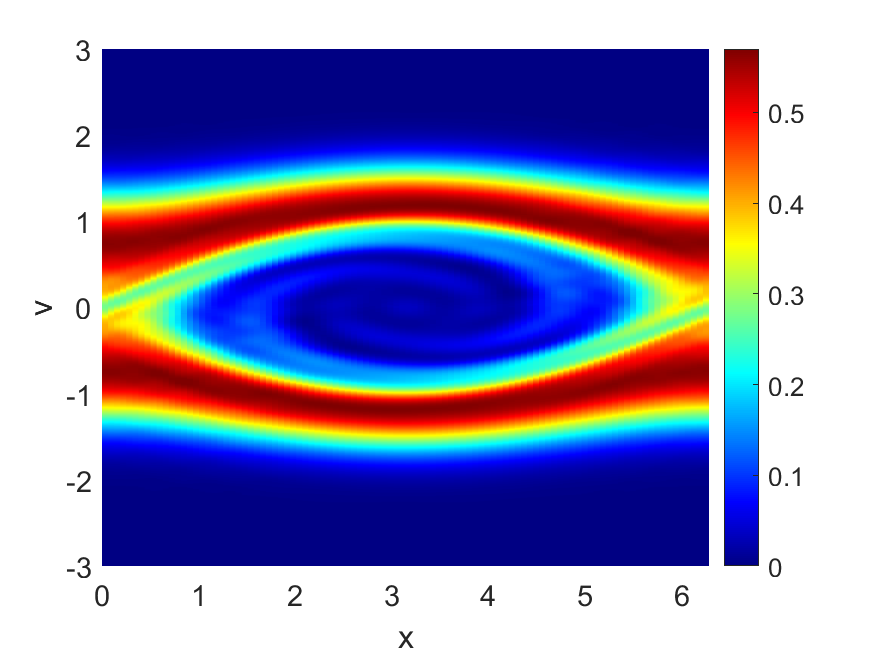}
  \includegraphics[width=0.24\textwidth]{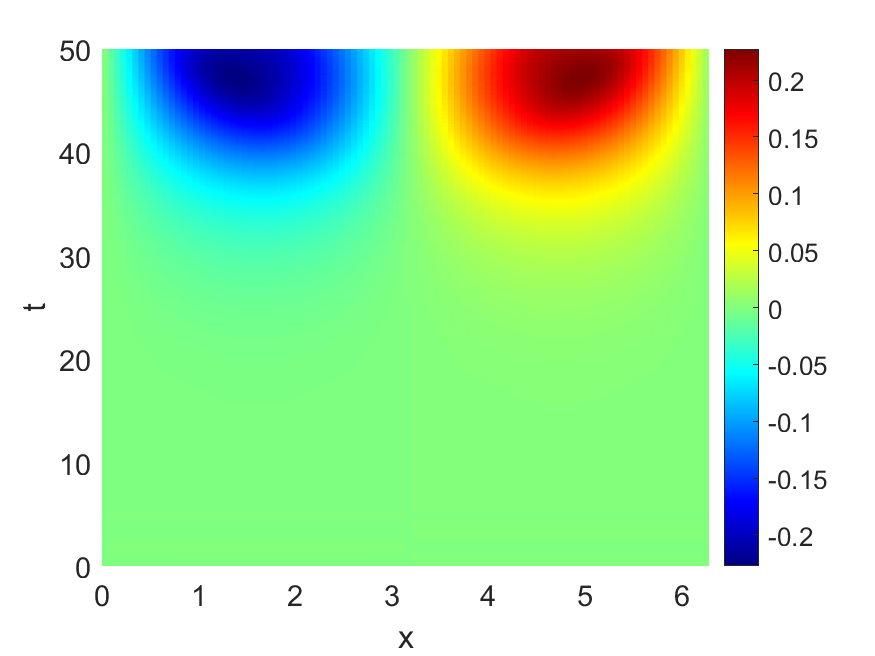}
  \includegraphics[width=0.24\textwidth]{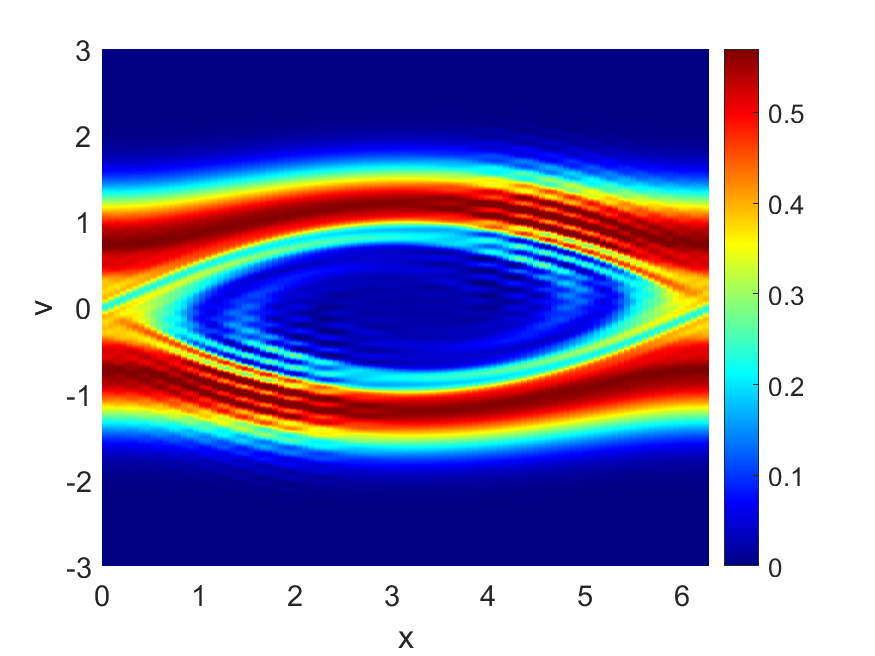}
  \includegraphics[width=0.24\textwidth]{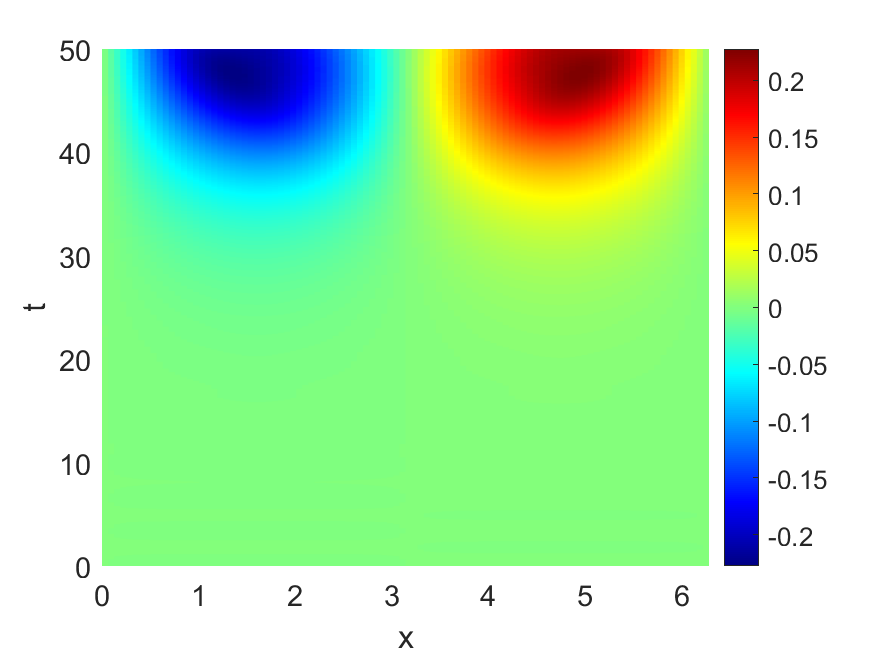}
  \caption{Two-stream instability. First and third columns: distribution function at time $t=50$ obtained with the adaptive (first column) and non-adaptive (third column) schemes. Second and forth columns: electric field in space and time obtained with the adaptive (second column) and non-adaptive (forth column) schemes.
    The number of Hermite modes, from top to bottom, is $N_v=10,50,100,125,250$.
  The collisional coefficient is $\nu=5$ and the number of Fourier modes is $N_x=50$.
    }
  \label{fig:TwoSNvColl}
\end{figure}

In Figure~\ref{fig:TSIerr} we compare the adaptive and non-adaptive algorithms in terms of relative error versus a reference solution computed separately, as done for instance in \cite{CDBM16}. The reference solution is obtained with the non-adaptive algorithm with $N_x=50$, $N_v=1000$, $\Delta t=0.01$, and $\nu=10$.
Figure~\ref{fig:TSIerr} (left) shows the $L_2$ norm of the relative error as a function of time and for different values of the number of Hermite modes $N_v$. The solid (dashed) lines are for the non-adaptive (adaptive) method. A couple of considerations can be made. First, as the two-stream instability grows and structures develop in the electron distribution function, the error grows in time in both approaches. Second, the error is essentially the same in the early part of the simulation for both approaches. This corresponds to the linear phase of the instability and reflects the fact that at this stage the $u$ and $\alpha$ parameters have not changed considerably (cf. Fig. \ref{fig:TwoSAeUeColl}). Third, in the final part of the simulation the error in the adaptive method is lower than for the non-adaptive method, reflecting that the change of basis is beneficial in improving the accuracy of the solution. This is consistent with the considerations made above, for instance for Fig.~\ref{fig:TwoSNvColl}. The right panel of Fig.~\ref{fig:TSIerr} shows the $L_2$ norm of the relative error at $t=50$ as a function of the number of Hermite modes $N_v$, for both adaptive and non-adaptive methods. While the error decreases with increasing velocity space resolution for both methods, it is evident that the adaptive method is more accurate. It also appears to have a faster asymptotic convergence rate than the non-adaptive method.

\begin{figure}[H]
  \centering
  \includegraphics[width=0.475\textwidth]{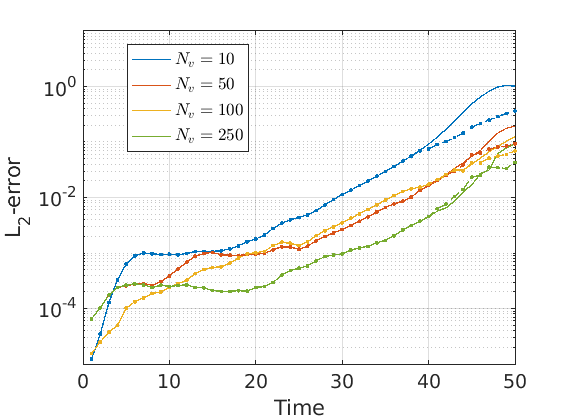}
  \includegraphics[width=0.475\textwidth]{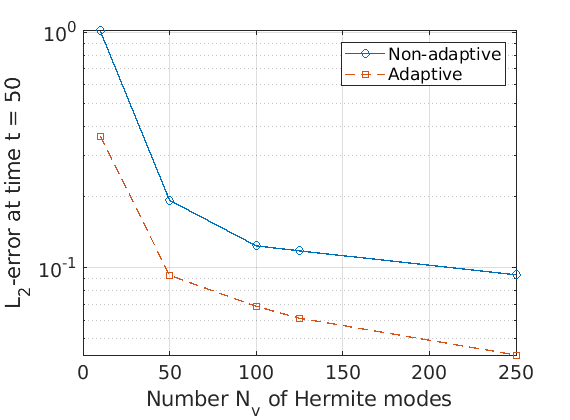}
  \caption{
    Two-stream instability. On the left, evolution of the $L_2$-error between the numerical solution and the reference solution for different numbers of Hermite modes. On the right, $L_2$-error at time $t=50$ vs. number of Hermite modes.
    }
  \label{fig:TSIerr}
\end{figure}

\section{Concluding remarks and outlook}
\label{sec:concluding:remarks}
In this work we have presented an adaptive spectral method for the numerical solution of the Vlasov--Poisson equations.
The discretization is based
on asymmetrically weighted Hermite functions in velocity,
and on Fourier basis for the spatial coordinate.
Since a poor choice of the basis functions might require spectral expansions
with a large number of modes to achieve even moderate accuracy,
we have introduced a scaling ($\alpha$) and shifting ($u$) of the velocity
variable that are adaptively updated in time.
The key aspects of the adaptive algorithm are: $(i)$ the mapping
between the approximation spaces that are associated with different
values of the Hermite parameters $\alpha$ and $u$; $(ii)$ the physics-based adaptivity criterion to
select and update $\alpha$ and $u$.
The space mapping has been constructed in such a way that adapting the Hermite parameters preserves the total mass, momentum and energy, even in the fully discrete problem, provided that a suitable temporal integrator, e.g. implicit midpoint rule, is used.
The strategy to adapt the Hermite parameters is based on
physics considerations with the aim of ``correcting'' the Hermite parameters according to the evolving dynamics.
Numerical experiments using both the adaptive and non-adaptive
algorithms prove numerically the effectiveness of the new adaptive approach in terms of providing more accuracy and stability of the solution for a fixed level of resolution as well as requiring less numerical viscosity to reduce filamentation effects. The same methodology can be easily extended to the general Vlasov-Maxwell equations.

While spectral methods have emerged recently as potentially attractive methods to study multiscale, micro/macro (fluid-kinetic) coupling problems for plasma physics applications, the adaptivity of the spectral expansion to capture locally the complexity of plasma dynamics and minimize the number of expansion terms needed to achieve sufficient accuracy remains a critical and yet essentially open problem. To the best of our knowledge, the only other approach that involves temporal adaptivity of the spectral expansion for plasma physics applications is the NR$xx$ method of \cite{NRxx} and its recent extension~\cite{NRxx2}. However, there are substantial differences between our approach and the NR$xx$ method, as highlighted in the introduction.

The only problem-dependent part of the proposed method is the adaptivity criterion: different criteria can be envisioned based on the properties
of the dynamics. These criteria could include physics-based criteria like the one proposed in this paper, but also other more mathematical criteria based for instance on the minimization of the error of the numerical scheme via some form of \textit{a posteriori} error estimation should be developed. Note that the flexibility of choosing different adaptivity criteria in a way that is completely decoupled from the underlying model equations is another distinct aspect of the proposed method relative to NR$xx$, where the specific choice of the adaptivity criteria is directly embedded in the method and changing the adaptivity criteria corresponds to a different method that actually solves different equations.
Additionally, in the presence of non-Maxwellian behavior and highly oscillatory solutions, an adaptivity based on average quantities might not be effective. In such cases, 
extensions of the proposed method to include both spatial and temporal adaptivity of the expansion basis will also need to be developed and are an important research direction for the future.

\section*{Acknowledgments}
The authors acknowledge fruitful discussions with Oleksandr Koshkarov, Gianmarco Manzini and Vadim Roytershteyn. The authors also acknowledge discussions with Juris Vencels in the initial phase of the work.
This work was supported by the Laboratory Directed
Research and Development Program
of Los Alamos National Laboratory under projects number 20170207ER and 20220104DR.
Los Alamos National Laboratory is operated by Triad National Security,
LLC, for the National Nuclear Security Administration of
U.S. Department of Energy (Contract No. 89233218CNA000001).
Computational resources for the simulations were provided by the Los Alamos National Laboratory Institutional Computing Program.


\bibliographystyle{plain}
\bibliography{biblio} 

\appendix
\section{Spatial discretization}\label{sec:appA}
For the spatial discretization of the Vlasov--Poisson problem
\eqref{eq:VlasovPoisson}, we rely on the numerical scheme proposed in \cite{GLD15} by using spectral methods.
Observe that the adaptive method proposed in this work can accommodate
other spatial and temporal discretizations.
We summarize the main ideas of the spatial approximation in this
appendix.

Let $\eta_k(x):=e^{\frac{2\pi ik}{L}x}$ denote the Fourier functions with wavenumber $k$.
Since $\Omega_x$ is bounded, the set of Fourier functions
$\{\eta_k(x)\}_{k\in\mathbb{Z}}$ form an orthonormal basis of
$L^2(\Omega_x)$.
The Fourier functions satisfy the orthogonality
relations
\begin{equation}\label{eq:Forth}
  \dfrac{1}{L}\int_{\Omega_x}\eta_k(x)\,\eta_h(x)\,\der x=\delta_{k+h,0},
\end{equation}
for all $h,k\in\mathbb{Z}$.

Let $N:=(N_v,N_x)\in\mathbb{N}\times\mathbb{N}$ and let us consider the approximation space $W^N:=\SPAN{\{\eta_k(x)\}_{k\in\Lambda_{N_x}}}$,
where
$\Lambda_{N_x}:=\{k\in\mathbb{Z}:\,-N_x\leq k\leq N_x\}$.
The discretized problem in each time interval $T_j$, $j\in\Lambda_{N_t}$, reads:
given the initial
  conditions $(f_{j}^N,E_{j}^N)\in (V^N_{j}\times W^N)\times W^N$, find $(f_{j+1}^N,E_{j+1}^N)\in (V^N_{j}\times W^N)\times W^N$
  such that
\begin{equation}\label{eq:HermFour}
\left\{
\begin{aligned}
    & \big(f_{j+1}^N-f_{j}^N,h^N\big)_{\Omega} +
    \Delta t_j \big(v\partial_x f_{j+1/2}^N + \dfrac{q}{m} E^N_{j+1/2}\,\partial_v f_{j+1/2}^N,h^N\big)_{\Omega} = 0,\\
    & \big(\partial_x E_{j+1}^N,H^N\big)_{\Omega_x} = \sum_s q^s(f_{j+1}^{s,N},H^N)_{\Omega},
  \end{aligned}\right.
\end{equation}
for all $h^N\in \widehat{V}^N_{j}\times W^N$ and $H^N\in W^N$.  
In the interval $T_j$, the approximated functions $f_j^N$ and $E_j^N$
can be represented in their phase space spectral expansion in
$V^N_{j}\times W^N$ and in $W^N$, respectively, as
\begin{equation}\label{eq:fNEexp}
\begin{aligned}
  & f_j^N(x,v) = \sum_{n\in\Lambda_{N_v}}\sum_{k\in\Lambda_{N_x}} \widehat{C}_{n,k}^j\,\psi^{\alpha_j,u_j}_n(v)\,\eta_k(x),\\
  & E_j^N(x) = \sum_{k\in\Lambda_{N_x}} \widehat{E}_k^j\,\eta_k(x),
\end{aligned}
\end{equation}
where the linear functionals $\{\widehat{C}^j_{n,k}: V^N_{j}\times
W^N\rightarrow\mathbb{C}\}_{n,k}$ and
$\{\widehat{E}^j_k:W^N\rightarrow\mathbb{C}\}_k$ are the expansion
coefficients defined as
\begin{equation*}
  \begin{aligned}
    & (f_j^N\in V^N_j\times W^N) \longmapsto \left\{\widehat{C}_{n,k}^j(f_j^N):=
      \int_{\Omega}f_j^N(x,v)\,\psi^{\alpha_j,u_j}_n(v)\,\eta_{-k}(x)\,\sqrt{\pi}\alpha^{-1}_je^{\xi_j^2}\,\der x\,\der v,
    \right\}_{(n,k)\in\Lambda_{N_v}\times\Lambda_{N_x}},\\
    & (E_j^N\in W^N)			  \longmapsto \left\{\widehat{E}_{k}^j(E_j^N):=\int_{\Omega_x}E_j^N(x)\,\eta_{-k}(x)\,\der x \right\}_{k\in\Lambda_{N_x}},
  \end{aligned}
\end{equation*}
and, by construction, they form a basis for the dual spaces of
$V^N_{j}\times W^N$ and of $W^N$, respectively.

By virtue of the orthogonality relations \eqref{eq:HForth} and \eqref{eq:Forth}, the
discrete problem \eqref{eq:HermFour} can be recast as a set of
$N_t(N_v+1)(N_x+1)$ coupled algebraic equations in the unknown
spectral coefficients. The resulting system reads:
  For each $j\in\Lambda_{N_t}$, given the initial
    spectral coefficients
    $\{\widehat{C}_{n,k}^{j}\}_{(n,k)\in\Lambda_{N_v}\times\Lambda_{N_x}}$
    and $\{\widehat{E}_k^j\}_{k\in\Lambda_{N_x}}$, find
    $\{\widehat{C}_{n,k}^{j+1}\}_{(n,k)\in\Lambda_{N_v}\times\Lambda_{N_x}}$
    such that
\begin{equation*}
\left\{ \begin{aligned}
    &\dfrac{\widehat{C}_{n,k}^{j+1}-\widehat{C}_{n,k}^j}{\Delta t_j}
    + \sqrt{\dfrac{n}{2}}\, \dfrac{2\pi}{L}ik\, \alpha_j\, \widehat{C}_{n-1,k}^{j+1/2}
    + \dfrac{2\pi}{L}ik\, u_j\, \widehat{C}_{n,k}^{j+1/2}
    + \sqrt{\dfrac{n+1}{2}}\,\dfrac{2\pi}{L} ik\,\alpha_j\,\widehat{C}_{n+1,k}^{j+1/2}\\[.2em]
    &\qquad -\sqrt{2n}\,\dfrac{q}{m}\,\dfrac{1}{\alpha_j}\sum_{\ell\in\Lambda_{N_x}}\,\widehat{C}_{n-1,\ell}^{j+1/2}\,\widehat{E}_{k-\ell}^{j+1/2}=0, 
    \qquad \forall\,(n,k)\in\Lambda_{N_v}\times\widehat{\Lambda}_{N_x},\\[.5em]
    & \dfrac{2\pi}{L}ik\, \widehat{E}_k^{j+1}-\sum_s q^s\alpha_j^s\,\widehat{C}_{0,k}^{s,j+1}=0,\qquad \forall\,k\in\widehat{\Lambda}_{N_x}\setminus\{0\},\\[.5em]
    & \widehat{E}_0^{j+1}=0,
  \end{aligned}\right.
\end{equation*}
with $\widehat{\Lambda}_{N_x}:=\{k\in\mathbb{Z}:\,0\leq k\leq N_x\}$.

Using the spectral expansion \eqref{eq:fNEexp} of $f^N_j$ in space and velocity,
the Hermite parameters in \eqref{eq:physAU} can be computed as
\begin{equation*}
  u_{j} = u_{j-1} + \dfrac{\alpha_{j-1}}{\sqrt{2}}\,\dfrac{\widehat{C}^{j-1}_{1,0}}{\widehat{C}^{j-1}_{0,0}},\qquad
  \alpha_{j} = \alpha_{j-1}\sqrt{1+\sqrt{2}\,\dfrac{\widehat{C}^{j-1}_{2,0}}{\widehat{C}^{j-1}_{0,0}}-
    \left(\dfrac{\widehat{C}^{j-1}_{1,0}}{\widehat{C}^{j-1}_{0,0}}\right)^2}.
\end{equation*}

\section{Two-stream instability: numerical study of the interplay between adaptivity and artificial collisionality}\label{sec:NumExpColl}
We consider the numerical test described in Section~\ref{sec:TSI} temporal interval $T=[0,50]$ with $\Delta t=0.05$.
The spectral discretization \eqref{eq:VFhL2w} of the Vlasov--Poisson
problem has $2N_x+1=101$ Fourier modes, and $N_v=100$ Hermite modes.
The tolerances for the Hermite parameters are set to
$\alpha_{\mathrm{tol}}=10^{-1}$ and $u_{\mathrm{tol}}=10^{-2}$.

In Figure~\ref{fig:TSIfnu}, we show the
distribution function for the electron species at time $t=50$ for
different values of the collisional coefficient $\nu$.
For any fixed value of $\nu$, the adaptive scheme gives better results in term of stability and accuracy.
On the other hand, as the value of $\nu$ increases, filamentation effects are mitigated in both methods (as expected, see also Section~\ref{sec:Coll}) but the non-adaptive method requires comparatively larger values of $\nu$.

\begin{figure}[H]
  \centering
  \includegraphics[width=0.3\textwidth]{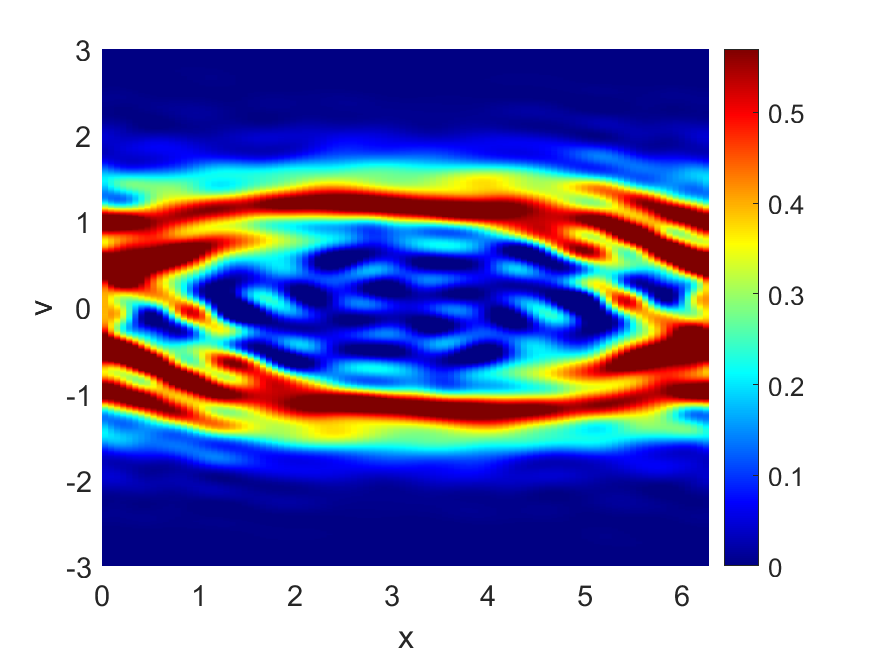}
  \includegraphics[width=0.3\textwidth]{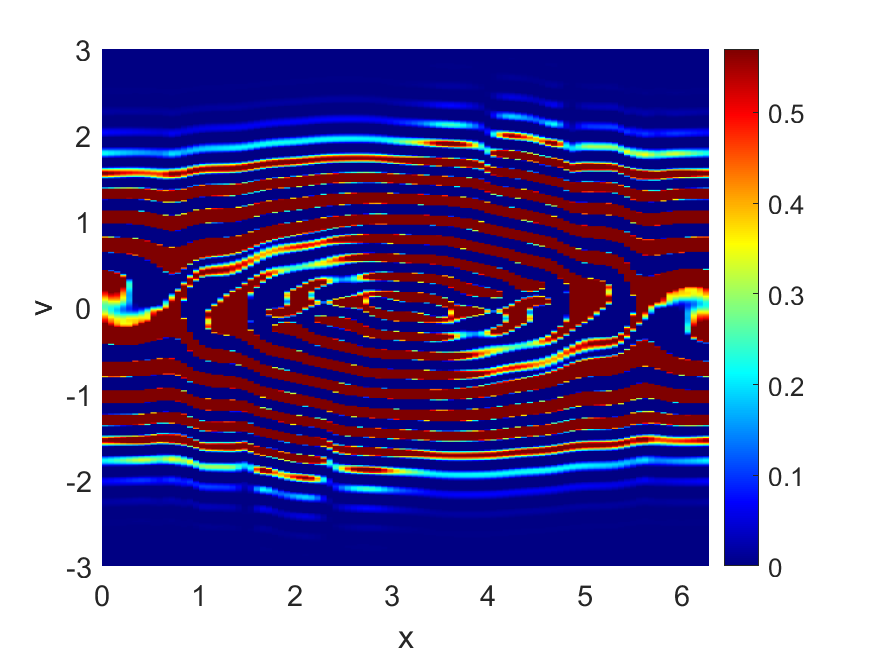}\\
  \includegraphics[width=0.3\textwidth]{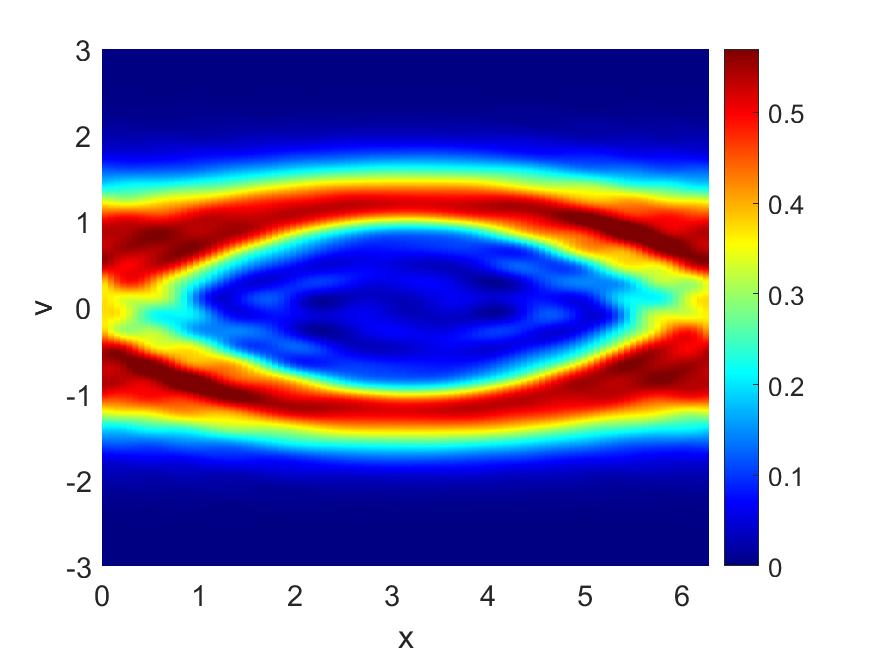}
  \includegraphics[width=0.3\textwidth]{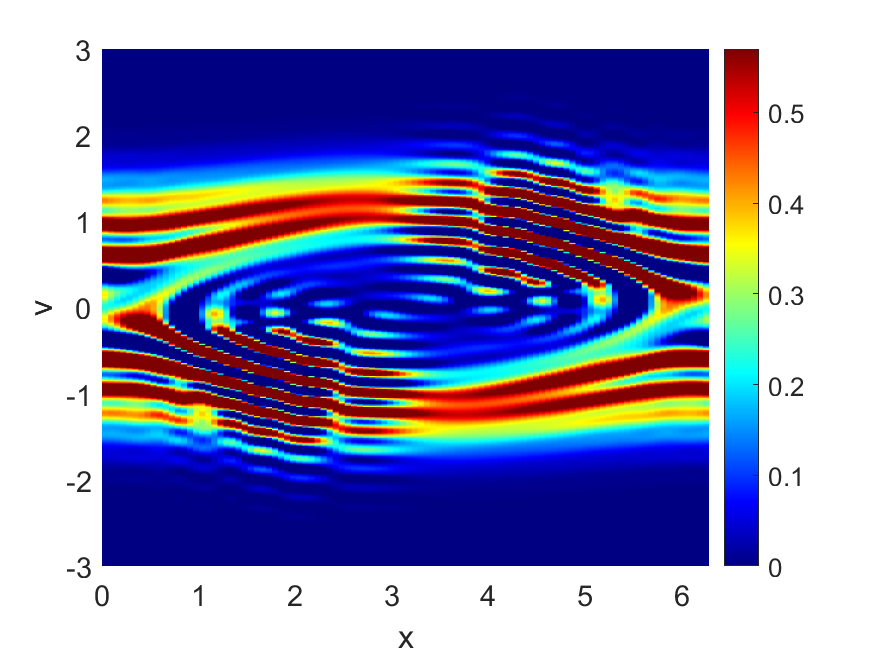}\\
  \includegraphics[width=0.3\textwidth]{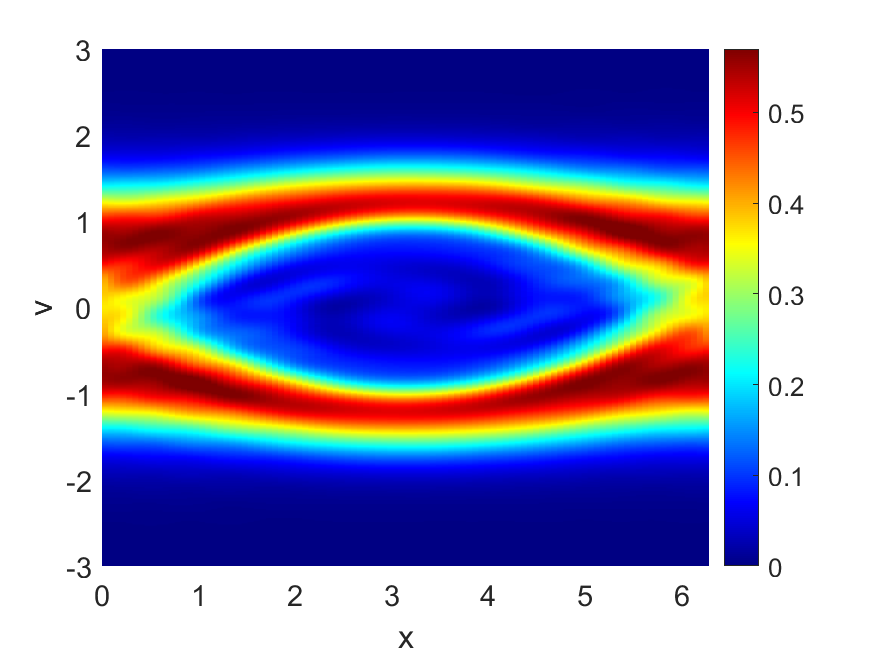}
  \includegraphics[width=0.3\textwidth]{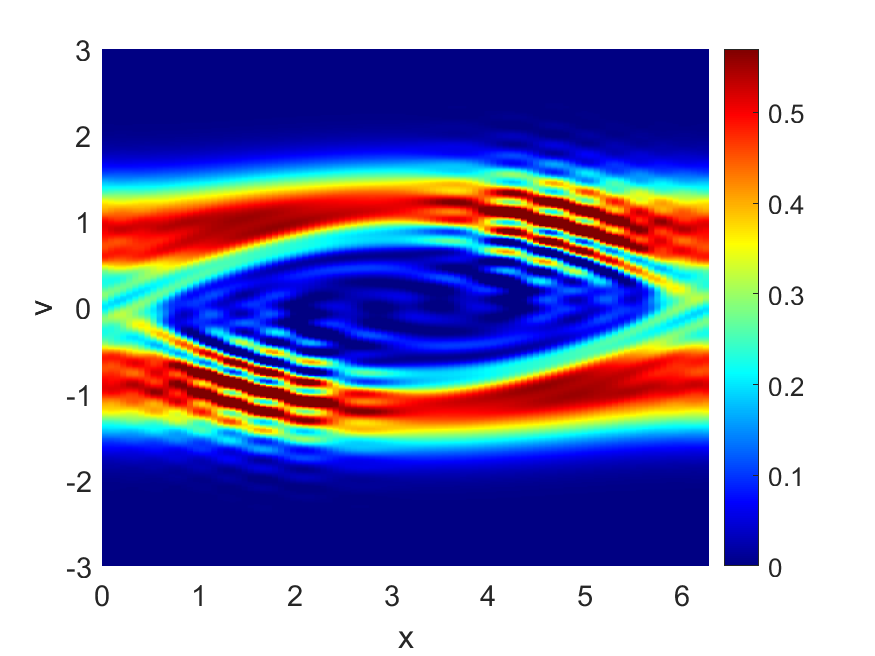}\\
  \includegraphics[width=0.3\textwidth]{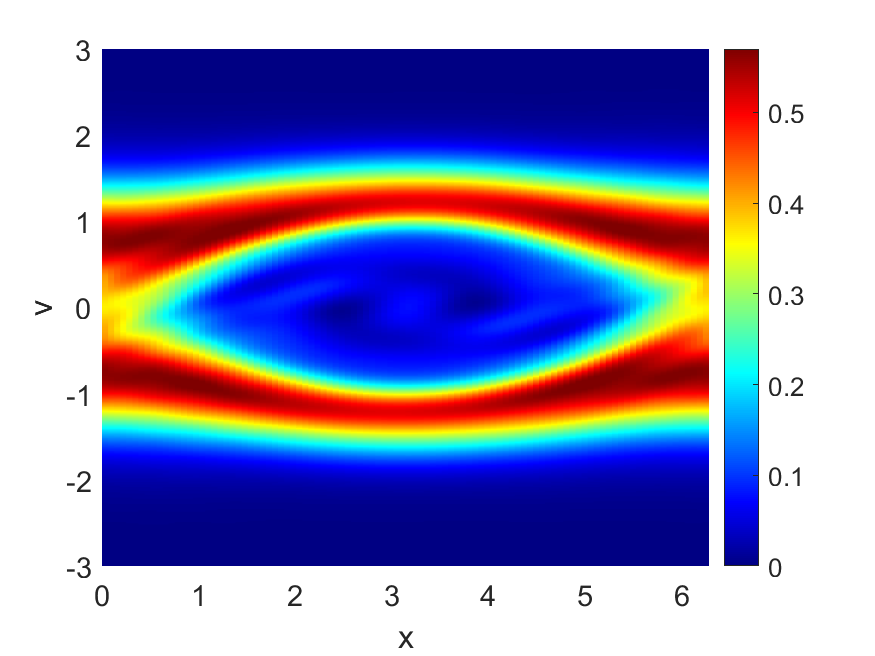}
  \includegraphics[width=0.3\textwidth]{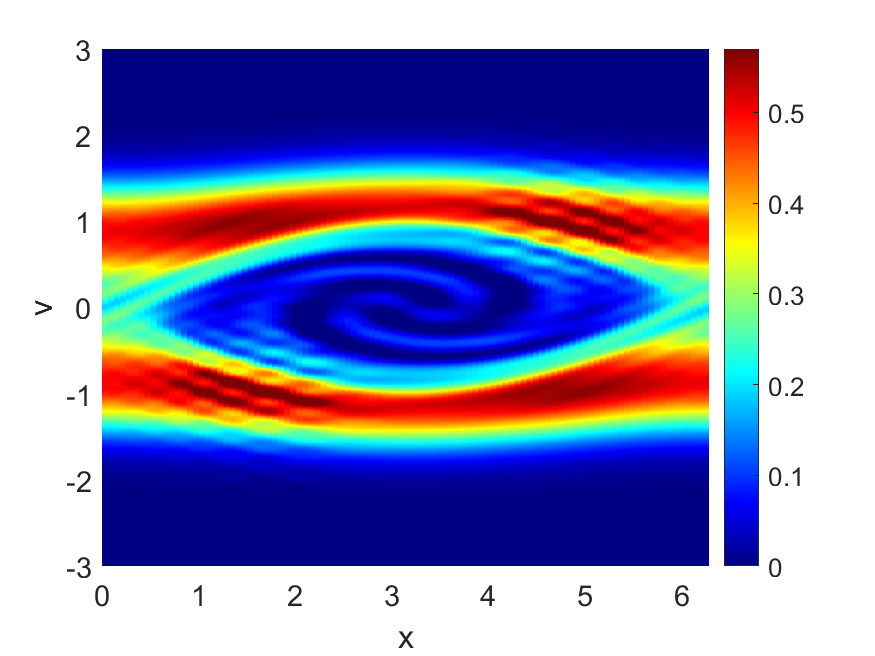}\\
  \includegraphics[width=0.3\textwidth]{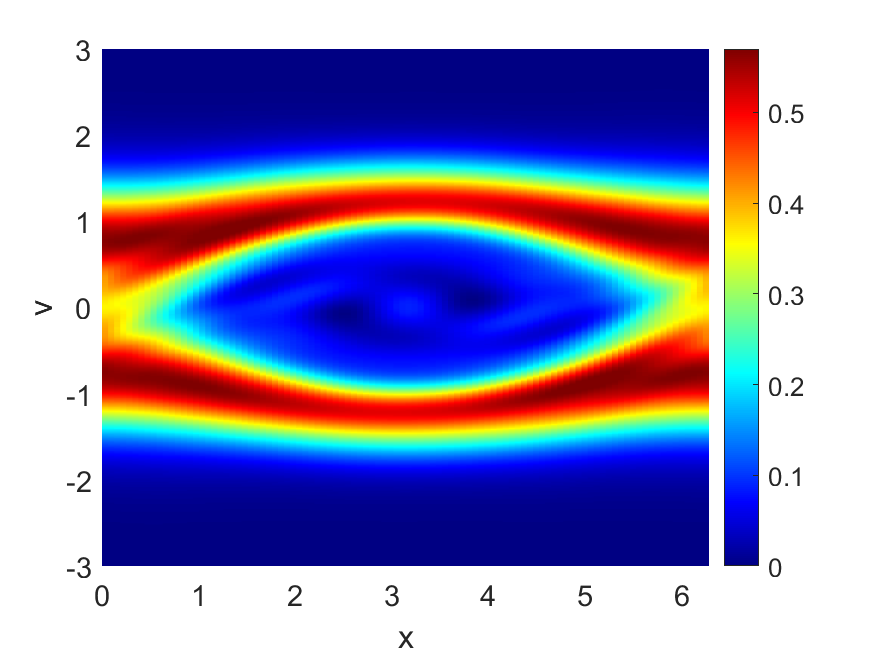}
  \includegraphics[width=0.3\textwidth]{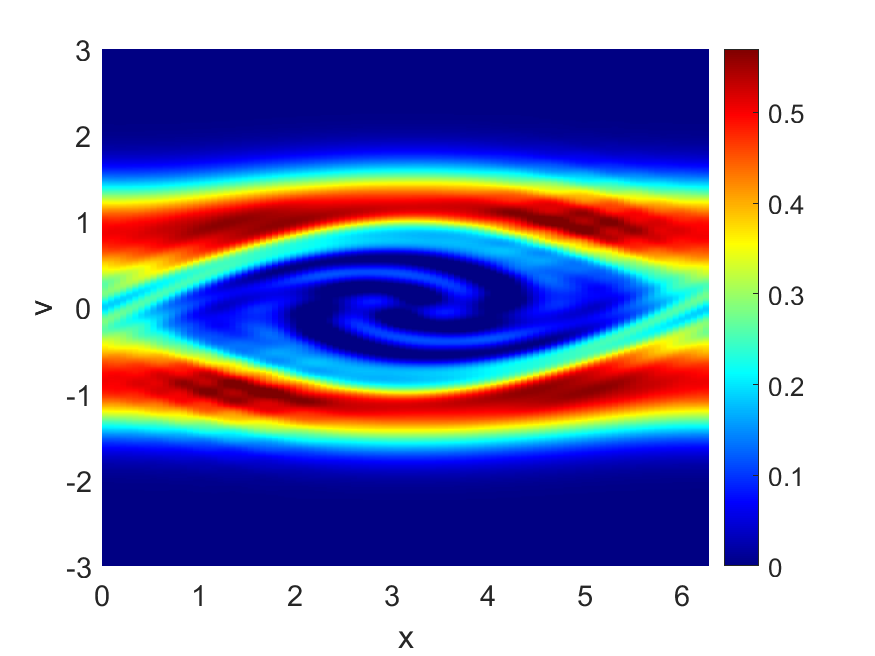}\\
  \includegraphics[width=0.3\textwidth]{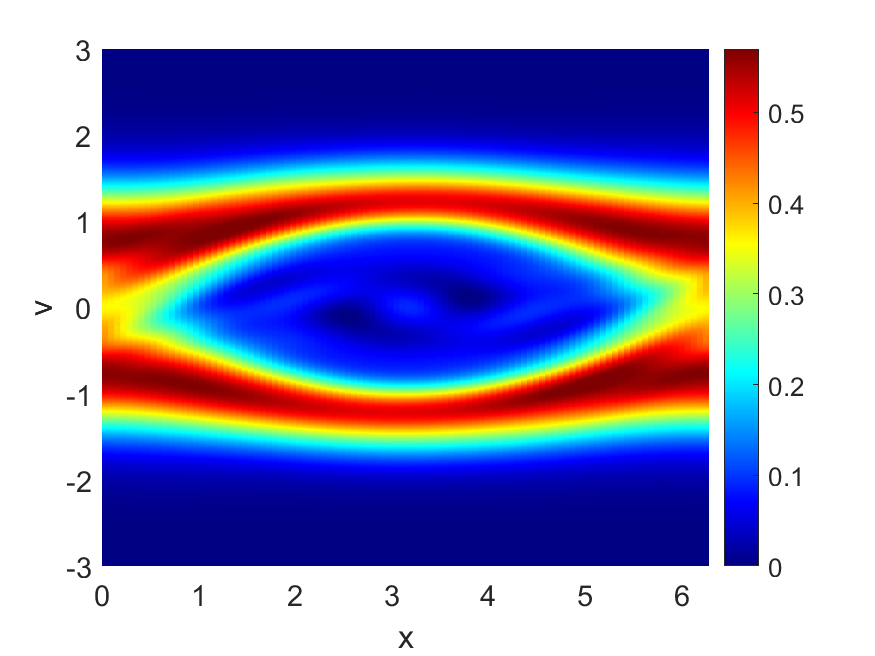}
  \includegraphics[width=0.3\textwidth]{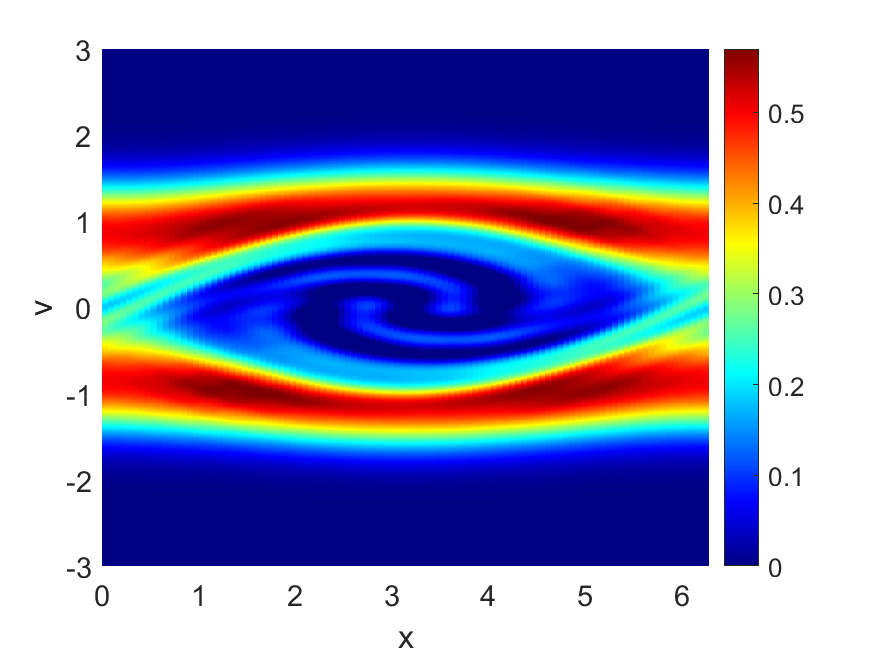}\\
  \caption{Two-stream instability.  Distribution function at time $t=50$ obtained with the adaptive (first column) and non-adaptive (second column) schemes. The value of the collisional coefficient $\nu$ ranges in the set $\{0,2,4,6,8,10\}$ from top to bottom. The number of spectral modes is $(N_v,N_x)=(100,50)$.}
  \label{fig:TSIfnu}
\end{figure}

In Table~\ref{tab:TSIfmaxmin} the maximum and minimum values of the distribution function for the electron species in the bounded domain $\Omega:=[0,L]\times[-3,3]$  are reported for different values of the collisional coefficient $\nu$.
The adaptive algorithm preserves the maximum value of the distribution function quite accurately, with a relative error smaller than $5\%$ for $\nu\geq 3$. For the non-adaptive algorithm, on the other hand, the maximum value of the distribution function varies wildly, but with a decreasing trend with $\nu$. For $\nu=5$, the maximum value of the distribution function is $34\%$ higher than the initial value. For higher values of $\nu$, it becomes more accurate and the relative error approaches $\sim 3\%$ for $\nu\ge10$. For both methods the minimum value of the distribution function is negative across all values of $\nu$ considered but the adaptive method performs significantly better.

\begingroup
\tabcolsep = 8pt
\def\arraystretch{1.1}
\begin{table}[H]
\centering
  \begin{tabular}{c | c c | c c}
    & \multicolumn{2}{c|}{$f^N_{\min}(t=50)$} & \multicolumn{2}{c}{$f^N_{\max}(t=50)$}  \\
    $\nu$ &\; non-adaptive\; &\; adaptive\; & \;non-adaptive\; & adaptive\\
    \hline
    $0$  \;& $-25.699$  & $-0.3286$ & $25.437$ & $0.9100$\\
    $1$  \;& $-7.3866$  & $-0.0743$ & $7.6759$ & $0.6811$ \\
    $2$  \;& $-2.3657$  & $-0.0368$ & $2.7673$ & $0.6290$ \\
    $3$  \;& $-0.8270$  & $-0.0206$ & $1.4128$ & $0.6029$ \\
    $4$  \;& $-0.2777$  & $-0.0208$ & $0.9465$ & $0.5884$ \\
    $5$  \;& $-0.1383$  & $-0.0275$ & $0.7535$ & $0.5799$ \\
    $6$  \;& $-0.1067$  & $-0.0350$ & $0.6627$ & $0.5798$\\
    $7$  \;& $-0.0973$  & $-0.0400$ & $0.6179$ & $0.5807$ \\
    $8$  \;& $-0.0930$  & $-0.0433$ & $0.5952$ & $0.5813$ \\
    $9$  \;& $-0.0994$  & $-0.0454$ & $0.5841$ & $0.5818$ \\
    $10$ \;& $-0.1088$  & $-0.0459$ & $0.5803$ & $0.5823$ \\
    $11$ \;& $-0.1219$  & $-0.0454$ & $0.5784$ & $0.5826$ \\
    $12$ \;& $-0.1338$  & $-0.0438$ & $0.5776$ & $0.5829$ \\
    $13$ \;& $-0.1495$  & $-0.0426$ & $0.5773$ & $0.5831$ \\
  \end{tabular}
  \caption{Maximum and minimum values of the distribution function for the electron species in the bounded domain $\Omega=[0,L]\times[-3,3]$ for different values of the collisional coefficient $\nu$.}
  \label{tab:TSIfmaxmin}
\end{table}
\endgroup

The evolution of the Hermite parameters (not shown) is basically independent of
the amplitude of the collisional term, confirming that it is affected predominantly by the macroscopic behavior of the numerical solution and not so much by the higher order Hermite modes which are damped more heavily by the specific form of the collisional operator chosen.

In Figure~\ref{fig:TSIEfield} the electric field is plotted on the
spatial domain $\Omega_x$ at time $t=50$ for different values of $\nu$.
The right panel shows that the non-adaptive scheme leads to spurious oscillations of the electric field if the collisional coefficient $\nu$ is not large enough. In the adaptive scheme, on the other hand, only the case $\nu=0$ shows some spurious oscillations. 
Note that the evolution of the $L_2$-norm of the electric field, and in particular its slope, is independent of the collision coefficient $\nu$ and the growth rate reproduced in all cases (the corresponding plot is not reported since the results obtained with different values of $\nu$ are visually indistinguishable).
Overall, by comparing the behavior of the adaptive and non-adaptive schemes, we conclude that the adaptive scheme performs better and requires lower value of the artificial collisionality to remove filamentation.

\begin{figure}[H]
  \centering
  \includegraphics[width=0.45\textwidth]{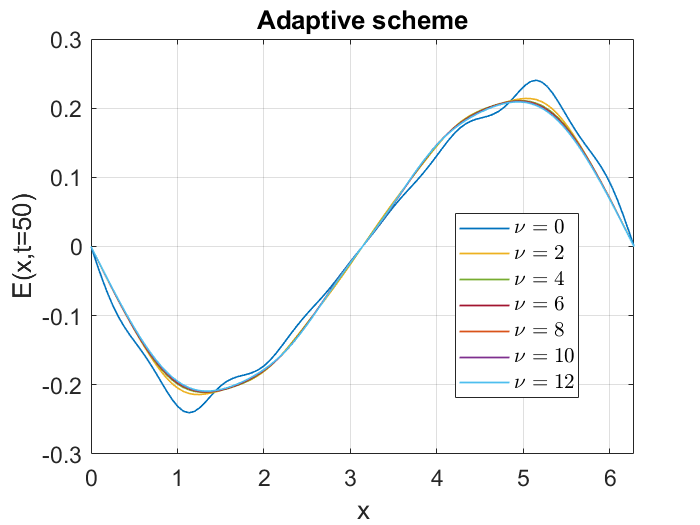}
  \includegraphics[width=0.45\textwidth]{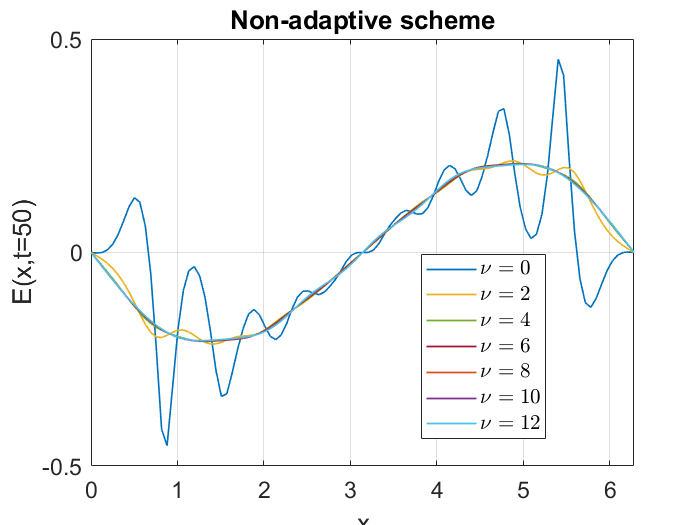}
  \caption{Two-stream instability. Electric field $E$ at time $t=50$ as a function of the space variable $x\in[0,L]$, for the adaptive (left) and non-adaptive (right) schemes. Different values of the collisional coefficient
    $\nu\in\{0,2,4,6,8,10,12\}$ are considered.}
  \label{fig:TSIEfield}
\end{figure}

\end{document}